\documentclass[onefignum,onetabnum]{siamart190516}
\usepackage{xcolor}
\usepackage{tikz}
\usepackage{mathrsfs}
\usepackage{tikz}

\numberwithin{equation}{section}
\renewcommand\figurename{$\mathbf{Fig.}$}
\renewcommand{\r}{\rho}

\renewcommand{\r}{\rho}
\newcommand{\R}{{\mathbb R}}

\newcommand{\dd}{{\rm d}}
\newcommand{\rs}{{\rm s}}

\makeatletter \@addtoreset{figure}{section} \makeatother
\renewcommand{\thefigure}{\arabic{section}.\arabic{figure}}

\usepackage{lipsum}
\usepackage{amsfonts}
\usepackage{graphicx}
\usepackage{epstopdf}
\usepackage{algorithmic}
\ifpdf
  \DeclareGraphicsExtensions{.eps,.pdf,.png,.jpg}
\else
  \DeclareGraphicsExtensions{.eps}
\fi

\newsiamremark{remark}{Remark}
\newsiamremark{hypothesis}{Hypothesis}
\crefname{hypothesis}{Hypothesis}{Hypotheses}
\newsiamthm{claim}{Claim}

\headers{Stability of Conic Shocks in Supersonic Flows}{Gui-Qiang G. Chen, Jie Kuang, and Yongqian Zhang}

\title{Stability of Conical Shocks in the Three-Dimensional Steady Supersonic Isothermal Flows
past Lipschitz Perturbed Cones\thanks{Submitted to the editors DATE.
\funding{The work of the first author was supported in part by the UK
Engineering and Physical Sciences Research Council
under Grant EP/L015811/1 and the Royal Society-Wolfson Research Merit Award (UK).
The work of the second author was supported in part by NSFC Project 11801549, NSFC Project 11971024,
and the Start-Up Research Grant (Project No. Y8S001104) from the Wuhan Institute of Physics
and Mathematics, Chinese Academy of Sciences.
The work of the third author was supported
in part by NSFC Project 11421061, NSFC Project 11031001, NSFC Project 11121101,
the 111 Project B08018 (China), and by the Shanghai Natural Science Foundation 15ZR1403900.
}}}

\author{Gui-Qiang G. Chen\thanks{ Mathematical Institute,
University of Oxford, Oxford, OX2 6GG, UK;
School of Mathematical Sciences,
Fudan University, Shanghai 200433, China;
Academy of Mathematics and Systems Science,
Chinese Academy of Sciences, Beijing 100190, China
  (\email{Gui-Qiang.Chen@maths.ox.ac.uk}).}
\and Jie Kuang\thanks{ Innovation Academy for Precision Measurement Science and
Technology, Chinese Academy of Sciences, Wuhan 430071, China;
Wuhan Institute of Physics and Mathematics, Chinese Academy of Sciences, Wuhan 430071, China;
School of Mathematical Sciences, Fudan University, Shanghai, 200433,
China (\email{jkuang@wipm.ac.cn}, \email{jkuang12@fudan.edu.cn}).}
\and Yongqian Zhang\thanks{ School of Mathematical Sciences, Fudan University,
   Shanghai 200433, China (\email{yongqianz@fudan.edu.cn}).} }

\usepackage{amsopn}

\begin{document}

\maketitle

\begin{abstract}
We are concerned with the structural stability of conical shocks in the three-dimensional
steady supersonic flows past Lipschitz perturbed cones
whose vertex angles are less than the critical angle.
The flows under consideration
are governed by the steady isothermal Euler equations for potential flow
with axisymmetry so that the equations contain a singular geometric source term.
We first formulate the shock stability problem as an initial-boundary value problem
with the leading conical shock-front as a free boundary,
and then establish the existence and
structural/asymptotic stability
of global entropy solutions of bounded variation ($BV$)
of the problem.
To achieve this, we first develop a modified Glimm scheme to construct approximate solutions
via self-similar solutions as building blocks
in order to incorporate with the geometric source term.
Then we introduce the Glimm-type functional, based on the local interaction estimates
between weak waves, the strong leading conical shock, and self-similar solutions,
as well as
the estimates of the center changes of the self-similar solutions.
To make sure of the decrease of the Glimm-type functional,
we choose appropriate weights by careful asymptotic analysis of the reflection coefficients
in the interaction estimates,
when the Mach number of the incoming flow is sufficiently large.
Finally, we establish the existence
of global entropy solutions involving a strong leading conical shock-front, besides weak waves,
under the conditions that the Mach number of the incoming flow is sufficiently large
and the weighted total variation of the slopes of the generating curve of the Lipschitz perturbed cone
is sufficiently small.
Furthermore,
the entropy solution is shown to  approach asymptotically the self-similar
solution that is determined by the incoming flow and the asymptotic tangent of the cone boundary
at infinity.
\end{abstract}

\begin{keywords}
   Conical shocks, structural stability, steady flow,
supersonic isothermal flow, perturbed cones, Lipschitz cones, entropy solutions, BV,
modified Glimm scheme, TV, Glimm-type functional, self-similar solutions,
interaction estimates, reflection coefficients, free boundary, asymptotic stability.
\end{keywords}

\begin{AMS}
 35B07, 35B20, 35D30, 76J20, 76L99, 76N10
\end{AMS}

\medskip
\section{Introduction}
$\,\,$ We are concerned with the structural stability of conical shocks
in the three-dimensional (3-D)
steady supersonic flows past Lipschitz perturbed cones whose vertex angles
are less than the critical angle.
The shock stability problem for steady supersonic flows past Lipschitz cones is fundamental
for the mathematical theory of multidimensional (M-D) hyperbolic
systems of conservation laws, since its solutions are time-asymptotic states and global attractors
of general entropy solutions of time-dependent initial-boundary value problems (IBVP) with rich nonlinear phenomena,
besides its importance to many areas of applications including aerodynamics;
see \cite{Anderson,Chen-Feldman2018,courant-friedrichs,dafermos2016} and the references cited therein.
As indicated in \cite{courant-friedrichs}, when a uniform supersonic flow
with constant speed from the far-field (minus infinity) hits a straight-sided symmetric cone
whose vertex angle is less than the critical angle,
there is a supersonic straight-sided conical shock attached to the vertex of the cone,
and the state between the
conical shock-front and the cone can be obtained by the shooting method, which is a self-similar
solution (see Fig. \ref{fig1.2}).
In this paper, we focus our analysis on the stability of the supersonic conical shock-front,
along with the background self-similar solution,
in the steady supersonic Euler flows that are isothermal and symmetric with respect
to the $x$--axis under the Lipschitz perturbation of the cones whose boundary surfaces in $\mathbb{R}^3$
are formed by the rotation of generating curves:
$y =b(x)$ for $x > 0$ around the $x$--axis (see Fig. \ref{fig1.1}).

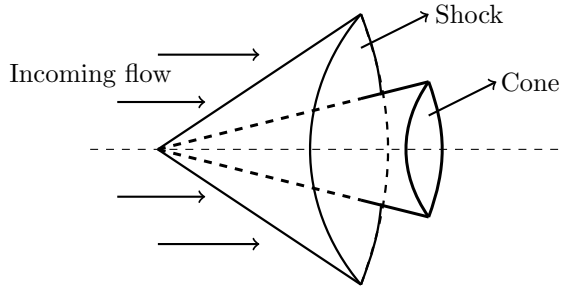
\begin{figure}[ht]
\begin{center}\label{fig1}
\begin{tikzpicture}[scale=0.90]
\draw [dashed] (-2,0) --(0,0)-- (5,0);
\draw [dashed][line width=0.04cm](-1,0)--(2, 0.75);
\draw [dashed][line width=0.04cm](-1,0)--(2, -0.75);
\draw [line width=0.04cm](2, 0.75)--(3, 1);
\draw [line width=0.04cm](2, -0.75)--(3, -1);
\draw [line width=0.04cm](3,1) to [out=-125, in=125](3,-1);
\draw [line width=0.04cm](3,1) to [out=-70, in=70](3,-1);
\draw [line width=0.03cm](-1,0)--(2, 2);
\draw [line width=0.03cm](-1,0)--(2, -2);
\draw [line width=0.03cm](2,2) to [out=-130, in=130](2,-2);
\draw [line width=0.04cm][line width=0.03cm](2,2) to [out=-70, in=95](2.3,0.8);
\draw [line width=0.04cm][line width=0.03cm](2.3,-0.8) to [out=-95, in=70](2,-2);
\draw [dashed][line width=0.03cm](2,2) to [out=-70, in=70](2,-2);
\draw [thick][->] (-1,1.4)-- (0.5,1.4);
\draw [thick][->] (-1.6,0.7)--(-0.3,0.7);
\draw [thick][->] (-1.6,-0.7)--(-0.3,-0.7);
\draw [thick][->] (-1,-1.4) --(0.5,-1.4);
\draw [thick][->] (2,1.5) --(3, 2);
\draw [thick][->] (3,0.5) --(4, 1);
\node at (-2,1.1) {Incoming flow};
\node at (3.6,2) {Shock};
\node at (4.5,1) {Cone};
\end{tikzpicture}
\caption{The strong straight-sided conical shock in the supersonic flow past a straight-sided cone}\label{fig1.2}
\end{center}
\end{figure}
More precisely, the governing $3$-D Euler equations for steady isothermal potential conical flows
are of the form:
\begin{equation}
\begin{cases}
\partial_{x}(\rho u)+\partial_y(\rho v)=-\frac{\rho v}{y},\\[1mm] \label{eq:1.1}
\partial_{x}v-\partial_{y}u=0,
\end{cases}
\end{equation}
together with the Bernoulli law:
\begin{eqnarray}
\frac{u^2+v^2}{2}+c^2\ln\rho=\frac{u^2_\infty}{2}+c^2\ln\rho_\infty, \label{eq:1.2}
\end{eqnarray}
where $(u, v)$ is the velocity in the $(x, y)$--coordinates,
$\rho$ is the flow density, and $U_\infty=(u_\infty,0)^\top$ and $\rho_\infty$
are the velocity and the density of the incoming flow, respectively.
The Bernoulli law \eqref{eq:1.2}
is derived from the constitutive relation for the isothermal gas between pressure $p$ and
density $\rho$:
\begin{eqnarray}
p=c^2\rho,  \label{eq:1.3}
\end{eqnarray}
where constant $c>0$ is the sound speed.

Without loss of generality, we may set $\rho_\infty=1$; otherwise, we can simply scale:
$
\rho \to \frac{\rho}{\rho_\infty},
$
in system \eqref{eq:1.1}--\eqref{eq:1.2} which is invariant in terms of the form.
For fixed sound speed $c>0$,
the Mach number:
$$
M_\infty=\frac{u_\infty}{c}
$$
is equivalent to $u_\infty$;
in particular,
the condition that the Mach number $M_\infty$ is sufficiently
large is equivalent to that the incoming velocity $u_\infty$ is sufficiently large.

\vspace{10pt}
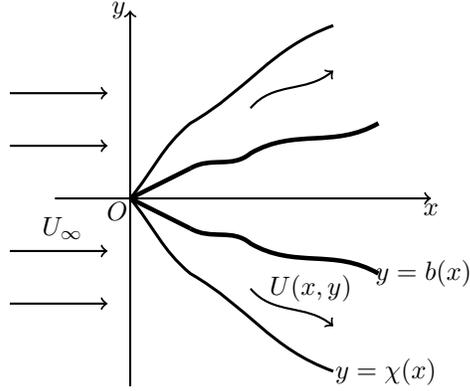
\begin{figure}[ht]
\begin{center}
\begin{tikzpicture}[scale=1.0]
\draw [thick][->] (-1,0) --(0,0)-- (4,0);
\draw [thick][<-] (0,2.5) --(0,0)-- (0,-2.5);
\draw [line width=0.06cm](0,0)--(0.8, 0.4);
\draw [line width=0.06cm](0,0)--(0.8, -0.4);
\draw [line width=0.06cm] (0.8,0.4)to [out=30, in=-140](1.6,0.6)to [out=30,in=-150](3.3,1.0);
\draw [line width=0.06cm] (0.8,-0.4)to [out=-30, in=140](1.6,-0.6)to [out=-30,in=150](3.3,-1.0);
\draw [line width=0.04cm] (0,0)to [out=50, in=-140](0.8,1.0)to  [out=30,in=-160](2.7,2.3);
\draw [line width=0.04cm] (0,0)to [out=-50, in=140](0.8,-1.0)to  [out=-30,in=160](2.7,-2.3);
\draw [thick][->] (1.6,1.2)to [out=50, in=-140] (2.7,1.7);
\draw [thick][->] (1.6,-1.2)to [out=-50, in=140] (2.7,-1.7);
\draw [thick][->] (-1.6,1.4)-- (-0.3,1.4);
\draw [thick][->] (-1.6,0.7)--(-0.3,0.7);
\draw [thick][->] (-1.6,-0.7)--(-0.3,-0.7);
\draw [thick][->] (-1.6,-1.4) --(-0.3,-1.4);
\node at (4.0,-0.15) {$x$};
\node at (-0.15, 2.5) {$y$};
\node at (-0.17, -0.17) {$O$};
\node at (-0.9,-0.4) {$U_{\infty}$};
\node at (2.4, -1.2) {$U(x,y)$};
\node at (3.9, -1) {$y=b(x)$};
\node at (3.4, -2.3) {$y=\chi(x)$};
\end{tikzpicture}
\end{center}
\caption{The strong conical shock $y=\chi(x)$ in the steady supersonic flow past a Lipschitz cone}\label{fig1.1}
\end{figure}

System  \eqref{eq:1.1} can be written in the form:
\begin{eqnarray} \label{eq:1.4}
\partial_{x}W(U)+\partial_{y}H(U)=G(U,y)
\end{eqnarray}
with $U=(u, v)^{\top}$, where
\begin{eqnarray*}
W(U)=(\rho u, v)^{\top}, \quad H(U)=(\rho v, -u)^{\top},
\quad G(U,y)=(-\frac{\rho v}{y}, 0)^{\top},
\end{eqnarray*}
and $\rho$ is a function of $U$ through the Bernoulli law \eqref{eq:1.2}.

\par When $\rho>0$ and $u>c$, $U$ can also be represented by $W=(\rho u, v)^\top$, {\it i.e.}, $U=U(W)$,
by the implicit function theorem, since the Jacobian:
$$
\det(\nabla_{U}W(U))=-\frac{\rho}{c^2}(u^{2}-c^2)<0.
$$
Regarding $x$ as the time variable, \eqref{eq:1.4} can be written as
\begin{eqnarray} \label{eq:1.4b}
\partial_{x}W+\partial_{y}H(U(W))=G(U(W),y).
\end{eqnarray}
Therefore, system \eqref{eq:1.1} becomes a hyperbolic
system of conservation laws with source terms of form \eqref{eq:1.4b}.
Such nonhomogeneous hyperbolic systems of conservation laws
also arise naturally in other problems from many important applications,
which exhibit rich phenomena;
for example,
see \cite{Chen-Feldman2018,ChenLevermoreLiu,ChenWagner,ChenWang,courant-friedrichs,dafermos2016}
and the references cited therein.

\par Throughout this paper, the following conditions are assumed:

\medskip
\begin{enumerate}
\item[\rm ($\mathbf{H_1}$)]
The Lipschitz generating curve
of the cone, $y=b(x)<0$ for $x > 0$,
is a small perturbation of line $y=b_{0}x$ for some constant $b_0 < 0$,
and satisfies
\begin{eqnarray*}
b(x)=b_0x \qquad \mbox{for $x\in[0,x_0]$}
\end{eqnarray*}
with some $x_0 > 0$, and
 \begin{eqnarray*}
 \|b'_{+}(\cdot)-b_{0}\|_{BV(\mathbb{R_{+}})}\leq \varepsilon \qquad\mbox{for some $\varepsilon>0$},
\end{eqnarray*}
where
$b'_{+}(x)=\lim _{y\rightarrow x+0}\frac{b(y)-b(x)}{y-x}\in BV([0,\infty))$.

\smallskip
\item[($\mathbf{H_2}$)]
The incoming flow velocity $U_{\infty}=(u_{\infty},0)^{\top}$ is supersonic:
\begin{equation*}
M_\infty > 1.
\end{equation*}
\end{enumerate}

\medskip
Given a perturbed generating curve $y=b(x)<0$ of the cone,
the problem is symmetric with respect to the $x$--axis.
Thus, it suffices to consider the problem in the following domain $\Omega$
in the half-space  $y\le 0$ outside the half-cone:
\begin{eqnarray*}
\Omega=\big\{(x,y)\,:\, x\geq0,\, y<b(x)\big\}
\end{eqnarray*}
with its boundary:
\begin{eqnarray*}
\partial{\Omega}=\big\{(x,y)\,:\, x\geq0,\, y=b(x)\big\},
\end{eqnarray*}
and the corresponding outer normal vector to $\partial{\Omega}$ at a differentiable point $x\in \partial{\Omega}$:
\begin{eqnarray*}
\mathbf{n}=\mathbf{n}(x,b(x))=\frac{(-b'(x),1)^\top}{\sqrt{1+{(b'(x))}^{2}}}.
\end{eqnarray*}

\medskip
\par With this setup, the shock stability problem can be formulated into the following
initial-boundary value problem (IBVP)
for system \eqref{eq:1.4}:

\medskip
\par $\mathbf{Cauchy}$ $\mathbf{Condition}$:
\begin{equation}
U|_{x=0} =U_\infty:=(u_\infty,0)^\top,   \label{eq:1.5}
\end{equation}

\par $\mathbf{Boundary}$ $\mathbf{Condition}$:
\begin{equation}
U\cdot\mathbf{n}\mid_{\partial\Omega} =0.  \label{eq:1.6}
\end{equation}

\medskip
We first introduce the notion of entropy solutions for problem
\eqref{eq:1.4}--\eqref{eq:1.6}.

\smallskip
\begin{definition}\label{def:1.1}
Consider IBVP \eqref{eq:1.4}--\eqref{eq:1.6} in $\Omega$.
A vector function $U(x,y)\in (BV_{\rm loc}\cap L^\infty)(\Omega)$ is an entropy solution
of \eqref{eq:1.4}--\eqref{eq:1.6}
if the following conditions are satisfied{\rm :}

\smallskip
 
{\rm (i)}\, For any test function $\phi\in C^{\infty}_{0}(\mathbb{R}^{2})$,
\begin{equation}\label{eq:1.7}
\,\, \int_{\Omega}\big\{W(U)\phi_{x}+H(U(W)\big)\phi_{y}
+G(U,y)\phi\big\}\,\dd x \dd y
+\int^{\infty}_{0}W(U_{\infty})\phi(0,y)\,\dd y= 0;
\end{equation}

{\rm (ii)}\, For any convex entropy pair $(\mathcal{E},\mathcal{Q})$ with respect to $W$ of \eqref{eq:1.4b}, i.e.,
$\nabla^2\mathcal{E}(W)\ge 0$ and $\nabla \mathcal{Q}(W)=\nabla\mathcal{E}(W)\nabla H(U(W))$,
\begin{eqnarray}\label{eq:1.8}
&&\quad\int_{\Omega}\big\{\mathcal{E}(W(U))\psi_x+\mathcal{Q}(W(U))\psi_y+\nabla_W\mathcal{E}(W(U))G(U)\psi\big\}\,\dd x \dd y\\
&&\quad +\int_{0}^{\infty}\mathcal{E}(W(U_{\infty}))\psi(0,y)\, \dd y\ge 0\qquad
\mbox{for any $\psi \in C_0^{\infty}(\R^2)$ with $\psi\ge 0$}.\nonumber
\end{eqnarray}
\end{definition}

\par We now state the main results of this paper.

\smallskip
\begin{theorem}[Main theorem]\label{thm:1.1}
Let conditions $(\mathbf{H_{1}})$--$(\mathbf{H_{2}})$ hold.
Assume that
\begin{eqnarray}\label{eq:1.9}
\int^{\infty}_{0}\big(1+|b(x)|\big)\,d\mu(x) <\varepsilon,
\end{eqnarray}
where $\mu(x)=T.V.\{b'_{+}(\tau)\,:\,\tau\in [0, x)\}$.
Then the following statements hold{\rm :}

\smallskip
{\rm (i)} {\rm (}Global existence{\rm )}{\rm :} If $M_{\infty}$ is sufficiently large and $\varepsilon$ is sufficiently small,
IBVP \eqref{eq:1.4}--\eqref{eq:1.6} admits a global entropy solution $U(x, y)$
with bounded total variation (TV) {\rm :}
\begin{eqnarray}\label{eq:1.10}
\underset{x>0}{\sup}\,\,{T.V.}\,\big\{U(x,y)\,:\,-\infty<y<b(x)\big\}   <\infty
\end{eqnarray}
in the sense of Definition {\rm \ref{def:1.1}}.
The entropy solution $U(x,y)$ contains a strong leading shock-front $y=\chi(x)=\int^{x}_{0}s(\tau)\,d\tau$
with $s(x)\in BV(R_{+})$, which is a small perturbation of the strong straight-sided conical shock-front
$y = s_{0}x$,
and $U(x,y)$
between the leading shock-front and the cone surface
is a small perturbation of the background self-similar solution of the straight-sided cone case,
where $s_0$ denotes the slope of the corresponding straight-sided shock-front when
the straight-sided cone is given by $y=b_{0}x$.

\smallskip
{\rm (ii)} $($Asymptotic behavior{\rm ):}
For the entropy solution $U(x,y)$ constructed in {\rm (i)},
\begin{eqnarray}\label{eq:1.13}
\lim_{x\rightarrow \infty}\sup\big\{|U(x,y)-\varpi(\sigma_{\infty}; O_{\infty})|\,:\, \chi(x)<y<b(x)\big\}=0
\end{eqnarray}
with $\varpi(\sigma_{\infty}; O_{\infty})$ satisfying
\begin{eqnarray}\label{eq:1.14}
\varpi(s_{\infty}; O_{\infty})=\Theta(s_{\infty}),\quad\, \varpi(b'_{\infty}; O_{\infty})\cdot (-b'_{\infty},1)=0,
\end{eqnarray}
where
\begin{eqnarray}\label{eq:1.11}
s_{\infty}=\lim_{x\rightarrow \infty}s(x),\quad b'_{\infty}=\lim_{x\rightarrow \infty}b'_+(x),
\end{eqnarray}
$\varpi(\sigma_{\infty}; O_{\infty})$ is the state of the self-similar solution with
$\sigma_{\infty}=\frac{y}{x-X^{*}_{\infty}}$ and $O_{\infty}=(X^{*}_{\infty}, 0)$
as its self-similar variable and center, respectively, for some $X^{*}$ determined by
the asymptotic limit of $b'_+(x)$ as $x\to \infty$,
and $\Theta(s)$ denotes the state connected to state $U_{\infty}$ by the strong leading shock-front
of speed $s$.
\end{theorem}

\medskip
Some efforts have been made on the shock stability problem for
the perturbed cones that are small perturbations
of the straight-sided cone during the last three decades.
The local piecewise smooth solutions for polytropic potential flow near the cone vertex were given
in \cite{chen,chen-li}
for both a symmetrically perturbed cone and pointed body, respectively.
The global existence of weak solutions was first analyzed via
a modified Glimm scheme by Lien-Liu \cite{lien-liu}
for the uniform supersonic isentropic Euler flow past a piecewise straight-sided cone,
provided that the cone has a small opening angle, the initial strength of the shock-front
is sufficiently weak, and the Mach number of the incoming flow is sufficiently large.
It is further considered in Wang-Zhang \cite{wang-zhang} for
supersonic potential flow for the adiabatic exponent $\gamma \in (1,3)$ over a symmetric Lipschitz cone
with arbitrary opening angle that is less than the critical angle,
so that a global weak solution could be constructed,
which is a small perturbation of the self-similar solution under
the conditions that the total variation of the slopes of the perturbed generating curves of the
cone is small and the Mach number of the incoming flow is sufficiently large.

Another concern is whether global piecewise smooth solutions could be constructed
when the surface of the perturbed cone is smooth.
Using the weighted energy methods,
the global existence of piecewise smooth solutions
was established in Chen-Xin-Yin \cite{chen-xin-yin}
for the $3$-D axisymmetric potential flow past a symmetrically perturbed cone under the assumptions that the attached
angle is sufficiently small and the Mach number of the incoming flow is sufficiently large.
This result was also extended to the M-D potential flow case (see \cite{li-witt-yin} for more details).
In \cite{xin-yin},
the global existence of
the M-D conical shock solutions was established when the uniform supersonic incoming flow with
large Mach number passes through a generally curved sharp cone under a certain boundary condition on the
cone surface.
On the other hand, by using the delicate expansion of the background solution,
the global existence and stability of a steady conical shock wave was established in Cui-Yin \cite{cui-yin-1, cui-yin-2}
for the symmetrically perturbed supersonic flow past an infinitely long conic body, when the
the vertex angle is less than the critical angle.
More recently, by constructing new background solutions that allow the incoming flows to tend to the speed limit,
the global existence of steady symmetrically conical shock solutions was established in Hu-Zhang \cite{hu-zhang}
when a supersonic incoming potential flow hits a symmetrically perturbed cone with the opening angle
less than the critical angle.
We also remark that some important results have been obtained on the stability of M-D transonic shocks under
symmetric perturbations of the straight-sided cones or the straight-sided wedges,
as well as on Radon measure solutions for steady compressible Euler equations of hypersonic-limit conical flows;
see \cite{Chen-Chen-Xiang, chen-fang, chen-fang-2017, Qu-Yuan, xu-yin} and the references cited therein.

In this paper,
we establish the global existence and structural/asymptotic stability of conical shock-front solutions in $BV$
in the flow direction when the isothermal flows ({\it i.e.}, $\gamma=1$) past Lipschitz perturbed cones
that are small perturbations of the straight-sided one.
Mathematically, our problem can be formulated as a free boundary problem governed by
two-dimensional steady isentropic irrotational Euler flows with geometric structure.
There are two difficulties for solving this problem:
One is the singularity generated by the geometric source term,
and the other is that, for our case $\gamma=1$,
the two genuinely nonlinear characteristics are superposed into a degenerate one
when the Mach number of the incoming flow tends to infinity, which is delicate to handle in
the construction of approximate solutions.

To overcome these obstacles and make sure of the non-increase of the ongoing designed Glimm-type functional,
we first develop a modified Glimm scheme to construct
approximate solutions $U_{\Delta x,\vartheta}(x,y)$
via the self-similar solutions as building blocks in order to
incorporate them with the geometric source term.
To achieve this, we make careful asymptotic expansions of the self-similar solutions
up to second order with respect to $M_\infty^{-1}$.
In addition to the shock waves and rarefaction waves
generated by solving the Riemann problem,
there is another new type of discontinuity generated by the center changes
and the corresponding updated self-similar variables
of the self-similar solutions, owing to the Lipschitz perturbation of the cone.
In order to deal with this new discontinuity, we introduce
new functionals $L_{\rm c}$, $Q^{(1)}_{\rm wc},\, Q^{(2)}_{\rm wc}$, and $Q_{\rm ce}$
(see Definitions \ref{def:7.1}--\ref{def:7.2}) in the construction of
the Glimm-type functional to control the center changes.
Finally, in order to ensure the decrease of the Glimm-type functional,
we make a more precise asymptotic expansion analysis of the background solutions
with respect to the Mach number $M_\infty$ of the incoming flow,
obtain their expansion formulas
when $M_{\infty}$ sufficiently large, and then make full use of
the reflection coefficients
$K_{\rm r}, K_{\rm w}, K_{\rm s}, $ and $\mu_{\rm w}$
of the weak waves reflected from both the boundary
and the strong leading shock, and
the self-similar solutions reflected from the strong leading shock
to derive that
$$
|K_{\rm r}|\big(|K_{\rm w}|+|K_{\rm s}||\mu_{\rm w}|\big)
=1-(8b^{4}_0+2b^{2}_0+1)m_0^{-1} M_\infty^{-1}
 +O(M^{-2}_{\infty})+O(e^{-m_{0}M^{2}_{\infty}}),
$$
which is strictly less than $1$ when $M_{\infty}$ is sufficiently large,
where $m_0=\frac{b_0^2}{2(1+b_0^2)}>0$.
We do this expansion with respect to sufficiently large $M_{\infty}$
in order to overcome the superposed singularity caused for the case that $\gamma=1$.
Based on this, we can choose some appropriate weights,
independent of $M_{\infty}$,
in the construction of the Glimm-type functional and
then show that the functional is monotonically decreasing.
With these, the convergence of the approximate solutions and the existence of
an entropy solution are followed by the standard approach
for the Glimm-type scheme as in \cite{glimm,lax};
see also \cite{chen-zhang-zhu,dafermos2016,smoller}.

For the asymptotic behavior of the entropy solution,
we need further estimates of the approximation solutions
$U_{\Delta x,\vartheta}(x,y)$. The key point here is that
a new term $\mathcal{C}_{\Delta x,\vartheta}(x)$ is introduced
to measure the total variation for
the changes of centers $X^{*}_{\Delta x, \vartheta}$
in $U_{\Delta x,\vartheta}(x,y)$ and show that this term eventually
approaches zero by further estimates of the approximate solutions, so that
$X^{*}_{\Delta x, \vartheta}$ tends to a constant $X^{*}_{\vartheta}$
for $\Delta x\rightarrow 0$.
This is different from the wedge case that has been handled in \cite{chen-zhang-zhu, zhang}.
In addition, we prove that the total variation of the weak waves
approaches zero as $x\rightarrow \infty$.
Then, by employing the Glimm-Lax theory \cite{glimm-lax}, we obtain the asymptotic behavior of
the entropy solution that tends to a self-similar solution
with $X^{*}_{\infty}=\lim_{x\rightarrow \infty}X^{*}_{\vartheta}(x, b(x))$ as its center.

\smallskip

The rest of this paper is organized as follows:
In \S 2, we recall some basic facts for the homogeneous system of \eqref{eq:1.1},
which are required for subsequent
developments.
In \S 3, we analyze the background solutions for steady supersonic flows past
the unperturbed straight-sided cones
and obtain some detailed asymptotic estimates for the self-similar solutions as $M_{\infty}\to \infty$.
In \S 4, we solve two types of Riemann problems, while a modified Glimm scheme
is developed for the construction of approximate solutions in \S 5.
The local wave interaction estimates are given in \S 6 for large ${M_{\infty}}$.
In \S 7, we construct the Glimm-type functional and prove its monotonicity
that leads to the existence theory by following the standard procedure
of \cite{glimm,lax}; see also \cite{chen-zhang-zhu,dafermos2016,smoller}.
In \S 8, we analyze the asymptotic behavior of the entropy solutions.
Finally, in Appendix A, we give a detailed proof of Lemma \ref{lem:2.1}.

\section{Homogeneous System}
In this section, we present some basic properties
of the homogeneous system of \eqref{eq:1.1}, {\it i.e.}, $G(U,y)\equiv 0$.
For this case, system \eqref{eq:1.4} can be reduced to the following conservation form:
\begin{eqnarray}\label{eq:2.1}
\partial_{x}W(U)+\partial_{y}H(U)=0.
\end{eqnarray}
For $u> c$, system \eqref{eq:2.1} is strictly hyperbolic and has two distinct eigenvalues:
\begin{equation*}
\lambda_{j}(U)=\frac{uv+(-1)^{j}c\sqrt{u^{2}+v^{2}-c^2}}{u^{2}-c^2}
\qquad \mbox{for $j=1, 2$},
\end{equation*}
 and the corresponding two right-eigenvectors:
\begin{equation*}
r_j(U)=e_j(U)(-\lambda_{j}(U), 1)^\top \,\,\qquad \mbox{for $j=1, 2$},
\end{equation*}
where $e_j(U)>0$ can be chosen so that $r_j(U)\cdot\nabla_{U}\lambda_{j}(U)\equiv 1$ for $j=1,2$.

The fact that $e_j(U)>0, j=1,2$, is a consequence of the following lemma whose proof is given
in Appendix A.

\smallskip
\begin{lemma}\label{lem:2.1}
If $\lambda_{j}(U)$ is the $j$-th eigenvalue of \eqref{eq:2.1} and $r_j(U)$ is the corresponding eigenvector
satisfying $r_j(U)\cdot\nabla_{U}\lambda_{j}(U)\equiv 1$ for  $u>c$ for $j=1,2$, then
\begin{eqnarray}\label{eq:2.2}
e_{j}(U)&=\frac{\sqrt{M^{2}-1}}{c^2 M^6}\big(u\sqrt{M^{2}-1}+(-1)^{j+1} v\big)^{3}>0\qquad \mbox{for $j=1, 2$},
\end{eqnarray}
where $M=\frac{q}{c}$ is the Mach number and $q=\sqrt{u^{2}+v^{2}}$ is the fluid speed.
\end{lemma}

\section{Properties of the Background Solutions}
\label{sec:alg}
In this section, we study the conical flow past a straight-sided cone, {\it i.e.},
$b(x) = b_{0}x$ for $x\ge 0$.
According to \cite{courant-friedrichs},
problem  \eqref{eq:1.4}--\eqref{eq:1.6} admits a self-similar solution
$(u(\sigma), v(\sigma), \rho(\sigma))$
with $\sigma=\frac{y}{x}$ as its self-similar variable for this case.
Then it can be reduced to
a boundary value problem of an ordinary differential equation, whose solution
consists of a straight-sided conical shock-front issuing from the cone vortex,
when $|b_{0}|$ is less than the critical angle (see Fig. \ref{fig3.1}).

\smallskip
Let $y = s_{0}x$ be the location of the shock-front.
Then problem \eqref{eq:1.1}--\eqref{eq:1.5} (with $\rho_\infty=1$ by scaling) becomes
\begin{equation}\label{eq:3.1}
\begin{cases}
(\sigma u-v)\rho_\sigma +\sigma\rho u_{\sigma}-\rho v_{\sigma}=\frac{\rho v}{\sigma},
 &s_{0}<\sigma<b_{0}, \\[1mm]
u_{\sigma}+\sigma v_{\sigma}=0, &s_{0}<\sigma<b_{0}, \\[1mm]
\frac{c}{\rho}\rho_{\sigma}+u u_{\sigma}+ v v_{\sigma}=0,  &s_{0}<\sigma<b_{0}, \\[1mm]
\rho(us_{0}-v)=u_{\infty}s_{0},&\sigma= s_{0}, \\[1mm]
u+ vs_{0}=u_{\infty},&\sigma= s_{0}, \\[1mm]
v-ub_{0}=0, &\sigma=b_{0},
\end{cases}
\end{equation}
and
\begin{eqnarray}\label{eq:3.2}
(u(\sigma), v(\sigma))=(u_{\infty},0),&\ \ \ \ \sigma<s_{0}.
\end{eqnarray}
System $\eqref{eq:3.1}_{1}$--$\eqref{eq:3.1}_{3}$ can also be rewritten in an equivalent form as
\begin{equation}\label{eq:3.3}
\begin{cases}
u_{\sigma}=\frac{c^2v}{(1+\sigma^{2})c^2-(v-\sigma u)^{2}},\\[2mm]
v_{\sigma}=-\frac{c^2v}{\sigma\left((1+\sigma^{2})c^2-(v-\sigma u)^{2}\right)},\\[2mm]
\rho_{\sigma}=\frac{\rho v(v-\sigma u)}{\sigma\left((1+\sigma^{2})c^2-(v-\sigma u)^{2}\right)}.
\end{cases}
\end{equation}

\vspace{10pt}
\begin{figure}[ht]
\begin{center}\label{fig1a}
\begin{tikzpicture}[scale=1.0]
\draw [thick][->] (-1,0) --(0,0)-- (4,0);
\draw [thick][<-] (0,2.5) --(0,0)-- (0,-2.5);
\draw [line width=0.05cm](0,0)--(3.2, 1.2);
\draw [line width=0.05cm](0,0)--(3.2, -1.2);
\draw [line width=0.04cm] (0,0)--(2.4,2.4);
\draw [line width=0.04cm] (0,0)--(2.4,-2.4);
\draw [thick][->] (1.6,-1.2)--(2.8,-2.0);
\draw [thick][->] (-1.6,1.4)-- (-0.3,1.4);
\draw [thick][->] (-1.6,0.7)--(-0.3,0.7);
\draw [thick][->] (-1.6,-0.7)--(-0.3,-0.7);
\draw [thick][->] (-1.6,-1.4) --(-0.3,-1.4);
\node at (4.0,-0.15) {$x$};
\node at (-0.15, 2.5) {$y$};
\node at (-0.17, -0.17) {$O$};
\node at (-0.9,-0.4) {$U_{\infty}$};
\node at (2.6, -1.4) {$U(\sigma)$};
\node at (3.8, -1.2) {$y=b_0x$};
\node at (3.0, -2.5) {$y=s_{0}x$};
\end{tikzpicture}
\caption{ Steady supersonic flow past an unperturbed straight-sided cone}\label{fig3.1}
\end{center}
\end{figure}
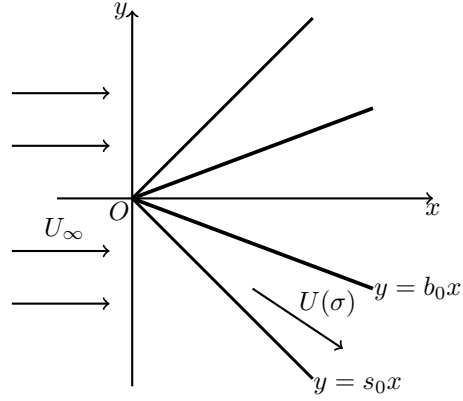

To study the self-similar solution, we need some properties of the shock polar.

\begin{lemma}\label{lem:2.7}
Let $b_{+}<0$.
Then there exist constants $K'>0$ and $K''\in (0, K')$ independent of
$M_{\infty}$ such that, for $M_{\infty}$ sufficiently large, the following
system of equations{\rm :}
\begin{eqnarray}
&&\rho_+(u_{+}s_{+}-v_{+})=u_{\infty}s_{+},\label{eq:3.4a}\\[1mm]
&&u_{+}+ v_{+}s_{+}=u_{\infty}, \label{eq:3.4b}\\[1mm]
&&u_{+}b_{+}-v_{+}=0,\label{eq:3.4c}\\[1mm]
&&\frac{u^{2}_{+}+v^{2}_{+}}{2}+c^2\ln\rho_{+}=\frac{u^{2}_{\infty}}{2}, \label{eq:3.4d}
\end{eqnarray}
has a unique solution $(u_{+}, v_{+}, \rho_{+}, s_{+})$ with
\begin{equation}
s_{+}\in \big(b_{+}-K'e^{-m_{+}M^{2}_{\infty}},\  b_{+}-K''e^{-m_{+}M^{2}_{\infty}}\big)
\qquad\mbox{for $m_{+}:=\frac{b^{2}_+}{2(1+b^{2}_{+})}$}.
\label{eq:3.5}
\end{equation}
In addition,
\begin{eqnarray}
&&u_{+}=\Big(\frac{1}{1+b^{2}_{+}}+O(1)e^{-m_{+}M^{2}_{\infty}}\Big)u_{\infty},\\[1mm] \label{eq:3.6}
&&v_{+}=\Big(\frac{b_{+}}{1+b^{2}_{+}}+O(1)e^{-m_{+}M^{2}_{\infty}}\Big)u_{\infty},\\[1mm] \label{eq:3.7}
&&\rho_{+}=
\exp\big\{m_{+}M^{2}_{\infty}\big(1+O(1)e^{-m_{+}M^{2}_{\infty}}\big)\big\}, \label{eq:3.8}
\end{eqnarray}
where $O(1)$ is independent of $M_{\infty}$.
\end{lemma}

\smallskip
\begin{proof} We divide the proof into three steps.

\smallskip
1. Equations \eqref{eq:3.4b}--\eqref{eq:3.4c} yield
\begin{equation}
u_{+}= \frac{u_{\infty}}{1+b_{+}s_{+}},\quad\,\,
v_{+}=\frac{b_{+}u_{\infty}}{1+b_{+}s_{+}}, \label{eq:3.10}
\end{equation}
which implies that $u_+>c$ for sufficiently large $M_\infty$.
Using $\eqref{eq:3.4a}$, we have
\begin{eqnarray}
\rho_{+}=\frac{s_{+}(1+s_{+}b_{+})}{s_{+}-b_{+}}.\label{eq:3.11}
\end{eqnarray}
Then substituting \eqref{eq:3.10}--\eqref{eq:3.11} into \eqref{eq:3.4d} leads to
\begin{eqnarray}
\frac{1}{2}\Big(\frac{1+b^{2}_{+}}{(1+b_{+}s_{+})^{2}}-1\Big)
+\ln(\frac{s_{+}(1+b_{+}s_{+})}{s_{+}-b_{+}})M^{-2}_{\infty}=0.\label{eq:3.12}
\end{eqnarray}

\smallskip
2. In order to solve \eqref{eq:3.12}, we define
\begin{eqnarray}
\varphi(s):=\frac{1}{2}\Big(\frac{1+b^{2}_{+}}{(1+b_{+}s)^{2}}-1\Big)
+\ln(\frac{s(1+b_{+}s)}{s-b_{+}})M^{-2}_{\infty}
\qquad\mbox{for $s<b_{+}$}.\label{eq:3.13}
\end{eqnarray}
Since $b_{+}<0$, it is direct to verify that
\begin{eqnarray}
\lim_{s\rightarrow b_{+}-}\varphi(s)=\lim_{s\rightarrow -\infty}\varphi(s)=\infty.\label{eq:3.14}
\end{eqnarray}
In addition, for $K>0$,
\begin{eqnarray}
\varphi(b_{+}-Ke^{-m_{+}M^{2}_{\infty}})=h(K, M_{\infty}), \label{eq:3.15}
\end{eqnarray}
where
\begin{equation}
\begin{aligned}
h(K, M_{\infty})&=\frac{1}{2}\Big(\frac{1+b^{2}_{+}}{(1+b^{2}_{+}-b_{+}Ke^{-m_{+}M^{2}_{\infty}})^{2}}-1\Big)
+\frac{b^{2}_{+}}{2(1+b^{2}_{+})}\\
&\ \ \ + \big(\ln|b_{+}-Ke^{-m_{+}M^{2}_{\infty}}|
+\ln(1+b^{2}_{+}-b_{+}Ke^{-m_{+}M^{2}_{\infty}})-\ln K\big)M^{-2}_{\infty}. \label{eq:3.16}
\end{aligned}
\end{equation}
Note that, for $K>0$,
\begin{eqnarray*}
h(K, M_{\infty})<\big(\ln|b_{+}-Ke^{-m_{+}M^{2}_{\infty}}|
+\ln\big(1+b^{2}_{+}-b_{+}Ke^{-m_{+}M^{2}_{\infty}}\big)-\ln K\big)M^{-2}_{\infty}.
\end{eqnarray*}
Then, for $M_{\infty}>\frac{\sqrt{\big|\ln(1+2b^{2}_{+}-2\sqrt{b^{2}_{+}(b^{2}_{+}+1)}\,)\big|}}{\sqrt{m_{+}}}$,
 we can choose appropriate $K'>0$ such that
\begin{eqnarray*}
\begin{aligned}
h(K', M_{\infty})&<-\frac{m_{+}}{\ln(1+2b^{2}_{+}-2\sqrt{b^{2}_{+}(b^{2}_{+}+1)}\,)}\\
&\quad \times \Big\{\ln\big((1+2b^{2}_{+}-2\sqrt{b^{2}_{+}(b^{2}_{+}+1)}\,)K'-b_{+}\big)\\
 &\qquad\,\,\,\, +\ln\big(1+b^{2}_{+}-(1+2b^{2}_{+}-2\sqrt{b^{2}_{+}(b^{2}_{+}+1)}\,)b_{+}K'\big)-\ln K' \Big\}
 <0.
\end{aligned}
\end{eqnarray*}
On the other hand, since
$\displaystyle \lim_{K\rightarrow 0+}\ln K=-\infty$,
then, for $M_{\infty}$ sufficiently large,
we can also choose another constant $K''\in (0, K')$ so that
\begin{eqnarray*}
h(K'', M_{\infty})>0.
\end{eqnarray*}
These lead to
\begin{eqnarray*}
\varphi(b_{+}-K'e^{-m_{+}M^{2}_{\infty}})<0, \qquad \varphi(b_{+}-K''e^{-m_{+}M^{2}_{\infty}})>0,
\end{eqnarray*}
which implies that $\varphi(s)=0$ has two solutions that
lie in $(-\infty,\ b_{+}-K'e^{-m_{+}M^{2}_{\infty}})$ and
$(b_{+}-K'e^{-m_{+}M^{2}_{\infty}},\ b_{+}-K''e^{-m_{+}M^{2}_{\infty}})$, respectively.

\smallskip
3.  The properties of the shock polar indicate
that $\varphi(s)=0$ has at most two solutions in $(-\infty, b_{+})$.
Therefore, $\varphi(s)=0$ has a unique solution in
$(b_{+}-K'e^{-m_{+}M^{2}_{\infty}}, \ b_{+}-K''e^{-m_{+}M^{2}_{\infty}})$,
which gives the uniqueness of $(u_{+}, v_{+}, \rho_{+}, s_{+})$ and
\begin{eqnarray*}
s_{+}\in (b_{+}-K'e^{-m_{+}M^{2}_{\infty}}, \, b_{+}-K''e^{-m_{+}M^{2}_{\infty}}).
\end{eqnarray*}
Then, by \eqref{eq:3.10}--\eqref{eq:3.11}, we obtain
the desire estimates \eqref{eq:3.6}--\eqref{eq:3.8}.
\end{proof}

Denote
\begin{eqnarray*}
\varphi(s, b):=\frac{1}{2}\Big(\frac{1+b^{2}}{(1+bs)^{2}}-1\Big)+\ln(\frac{s(1+bs)}{s-b})M^{-2}_{\infty}.
\end{eqnarray*}

\begin{lemma}\label{lem:2.8}
For $M_{\infty}$ sufficiently large and $s\in [5b_{0}, b_{0}]$,
$\varphi(s, b)=0$ has a unique solution $b=b(s)$ with $b(s)\in (s,0)$.
Moreover,
\begin{eqnarray*}
b(s)=s+O(1)e^{-m_{0}M^{2}_{\infty}},
\end{eqnarray*}
where $m_{0}=\frac{b^{2}_{0}}{2(1+b^{2}_{0})}$, and $O(1)$ depends
only on $b_{0}<0$ but is independent of $M_{\infty}$.
\end{lemma}

\begin{proof}
We differentiate $\varphi(s,b)$ with respect to $b$ to obtain
\begin{eqnarray*}
\frac{\partial \varphi(s,b)}{\partial b}
=\frac{(b-s)^{2}+(1+s^{2})(1+bs)^{2}M^{-2}_{\infty}}{(b-s)(1+bs)^{3}}.
\end{eqnarray*}
To estimate the zero points of $\frac{\partial \varphi(s,b)}{\partial b}$ in $b$, let
\begin{eqnarray*}
\varphi_{1}:=(b-s)^{2}+(1+s^{2})(1+bs)^{2}M^{-2}_{\infty}.
\end{eqnarray*}
Then
\begin{eqnarray*}
\frac{\partial \varphi_{1}}{\partial b}=2\big(b-s+(1+s^{2})(1+bs)sM^{-2}_{\infty}\big).
\end{eqnarray*}
For sufficiently large $M_{\infty}$,
$\frac{\partial \varphi_{1}}{\partial b}>0$.
Thus, $\varphi_{1}=0$ has at most one solution in $(s,0)$,
which implies that $\frac{\partial \varphi(s,b)}{\partial b}=0$ has at
most one zero point in $(s,0)$.
On the other hand, by a direct computation, we have
\begin{eqnarray*}
\lim_{b\rightarrow s+}\varphi(s,b)=\infty,
\end{eqnarray*}
and
\begin{eqnarray*}
\begin{aligned}
\varphi(s,s+Ke^{-m_{0}M^{2}_{\infty}})&=\frac{1}{2}\Big(\frac{1+(s+Ke^{-m_{0}M^{2}_{\infty}})^{2}}{(1+s^{2}+sKe^{-m_{0}M^{2}_{\infty}})^{2}}-1\Big)
+\frac{b^{2}_{0}}{2(1+b^{2}_{0})}\\
&\ \ \ +\big(\ln|s|+\ln(1+s^{2}+sKe^{-m_{0}M^{2}_{\infty}})-\ln K\big)M^{-2}_{\infty}\\
&<\frac{b^{2}_{0}}{2(1+b^{2}_{0})}-\frac{s^{2}}{2(1+s^{2})}
+\big(\ln|s|+\ln(1+s^{2})-\ln K\big)M^{-2}_{\infty}<0
\end{aligned}
\end{eqnarray*}
for appropriate $K>0$ and sufficiently large $M_{\infty}$,
which imply the existence of $b(s)$.
\end{proof}

\smallskip
\begin{lemma}\label{lem:2.9}
Let $(u(s_{0}), v(s_{0}), \rho(s_{0}))$ be a state on the shock polar
passing through  $(u_{\infty}, 0, \rho_\infty)$ with speed $s_0$.
Then the following two statements are equivalent{\rm :}

\smallskip
{\rm (i)}\,
The density increases across the shock in the flow direction{\rm :}
\begin{eqnarray}
\rho(s_{0})>1=\rho_\infty. \label{eq:3.17}
\end{eqnarray}

{\rm (ii)}\,
The shock speed $s_0$ must be between $\lambda_{1}(s_{0})$ and $\lambda_{1}(U_{\infty})${\rm :}
\begin{eqnarray}
\lambda_{1}(s_{0})<s_{0}<\lambda_{2}(s_{0}), \qquad  s_{0}<\lambda_{1}(U_{\infty}), \label{eq:3.18}
\end{eqnarray}
$\qquad\quad\,\,\,$ where
$$
\qquad\qquad \lambda_{j}(s_{0})=\frac{u(s_0)v(s_0)+(-1)^j c\sqrt{u^{2}(s_0)+v^{2}(s_0)-c^2}}{u^{2}(s_0)-c^2}\qquad\mbox{for $j=1,2$}.
$$
\end{lemma}

\begin{proof}  We divide the proof into two steps.

\medskip
1. Case $\rm (i)\Rightarrow \rm(ii)$.  By the Bernoulli law and the Rankine-Hugoniot relation in \eqref{eq:3.1}, we have
\begin{eqnarray*}
\frac{\alpha^{2}}{\alpha^{2}-1}\ln\alpha=\frac{s^{2}_{0}}{2(1+s^{2}_{0})}M^{2}_{\infty},
\,\,\,\,
\frac{\ln\alpha}{\alpha^{2}-1}=\frac{\big(v(s_0)-s_0u(s_0)\big)^{2}}{2c^2(1+s^{2}_{0})}
\qquad\,\,\, \mbox{for $\alpha:=\rho(s_{0})$}.
\end{eqnarray*}

Denote
\begin{eqnarray*}
f(\alpha):=\frac{\alpha^{2}}{\alpha^{2}-1}\ln \alpha\qquad\,\, \mbox{for $\alpha>1$}. \label{eq:3.21}
\end{eqnarray*}
Then
\begin{eqnarray*}
f'(\alpha)=\frac{\alpha g(\alpha)}{(\alpha^{2}-1)^{2}}\qquad\,\, \mbox{for $\alpha>1$}, \label{eq:3.22}
\end{eqnarray*}
where $g(\alpha)=\alpha^{2}-2\ln \alpha-1$.  Since
$\left. g'(\alpha)\right|_{\{\alpha>1\}}>0$,
then $g(\alpha)>g(1)=0$ for any $\alpha>1$, which implies that
\begin{eqnarray*}
f'(\alpha)\big|_{\{\alpha>1\}}>0.
\end{eqnarray*}
Then
\begin{eqnarray*}
\frac{s^{2}_{0}}{2(1+s^{2}_{0})}M^{2}_{\infty}=f(\alpha)>\lim_{\alpha\rightarrow1^{+}}f(\alpha).\label{eq:3.24}
\end{eqnarray*}
Applying L'H\^{o}pital's rule gives that $\lim_{\alpha\rightarrow1^{+}}f(\alpha)=\frac{1}{2}$. Therefore, we have
\begin{eqnarray*}
\frac{s^{2}_{0}}{2(1+s^{2}_{0})}M^{2}_{\infty}>\frac{1}{2},
\end{eqnarray*}
which yields that
\begin{eqnarray*}
s_{0}<-\frac{1}{\sqrt{M^{2}_{\infty}-1}}=\lambda_1(U_\infty).\label{eq:2.42}
\end{eqnarray*}
In the same way, we can show that, for $\alpha=\rho(s_0)>1$,
\begin{eqnarray*}
\frac{\ln\alpha}{\alpha^{2}-1}=\frac{\big(v(s_0)-s_0u(s_0)\big)^{2}}{2c^2(1+s^{2}_{0})}<\frac{1}{2},
\end{eqnarray*}
which implies that $s_{0}\in (\lambda_{1}(s_{0}), \lambda_2(s_0))$.

\medskip
2. Case $\rm(ii) \Rightarrow \rm(i)$.  On the contrary, assume that $\rho(s_{0})\leq 1=\rho_{\infty}$.
Then
\begin{eqnarray*}
f'(\alpha)\big|_{\{0<\alpha\leq1\}}\leq0,
\end{eqnarray*}
so that
$f(\alpha)\leq \lim_{\alpha\rightarrow1^{-}}f(\alpha)$ for $\alpha\in(0,1)$.
It follows that
\begin{eqnarray*}
\frac{s^{2}_{0}}{2(1+s^{2}_{0})}M^{2}_{\infty}<\frac{1}{2},
\end{eqnarray*}
that is,
\begin{eqnarray*}
s_{0}>-\frac{1}{\sqrt{M^{2}_{\infty}-1}},
\end{eqnarray*}
which contradicts \eqref{eq:3.18}. The proof is complete.
\end{proof}

\smallskip
Let $\Theta(s)=(\tilde{u}(s), \tilde{v}(s))$ be the states on the parameterized shock polar of $S_1^{-}(U_\infty)$
as defined in \S 4.2. Then we have the following lemma:

\begin{lemma}\label{lem:3.4a}
For $s<\lambda_1(U_\infty)$, $\frac{\tilde{v}(s)}{\tilde{u}(s)}$ is a strictly monotone increasing
function with respect to $s$.
\end{lemma}

\smallskip
\begin{proof}
From Lemma 3.1, we know that there is only one intersection point between the straight line $v=bu$ with $b<0$
and the shock polar $S^{-}_{1}(U_\infty)$ in the supersonic region.
This implies that the flow angle $\theta(s)=\arctan(\frac{\tilde{v}(s)}{\tilde{u}(s)})$ is a strictly monotone
function of $s$. Furthermore, from the properties of the shock polar $S^{-}_{1}(U_{\infty})$
(or see \cite{courant-friedrichs,zhang} for more details), we also see
that $\theta(s)<0=\theta(\lambda_{1}(U_{\infty}))$ for $s<\lambda_{1}(U_{\infty})$.
\end{proof}

\smallskip
Now we consider the conical flows. We recall some properties of the apple curves in \cite{courant-friedrichs}.
Given a constant state $(u^{0}_{1}, v^{0}_{1})$ on the shock polar through state $(u_{\infty}, 0)$  (see Fig. \ref{fig3.2}),
let $(u_{1}(\sigma), v_{1}(\sigma))$ be the solution of $\eqref{eq:3.3}_1$--$\eqref{eq:3.3}_2$ with initial data
\begin{eqnarray*}
(u_{1}, v_{1})\big|_{\sigma=\sigma_0}=(u^{0}_{1}, v^{0}_{1}) \qquad\,\,\,\mbox{for $\sigma_0=\frac{u_{\infty}-u^{0}_{1}}{ v^{0}_{1}}$}.
\end{eqnarray*}
Then we can continue the solution, $(u_{1}(\sigma), v_{1}(\sigma))$, till endpoint $(u_{1}(\sigma_{e}), v_{1}(\sigma_{e}))$
so that $\frac{v_{1}(\sigma_{e})}{u(\sigma_{e})}=\sigma_{e}$.
The collection of the end states forms an apple curve through $(u_{\infty}, 0)$.
The solution, $(u(\sigma), v(\sigma))$, of  $\eqref{eq:3.3}_1$--$\eqref{eq:3.3}_2$
can be found by the shooting method
(see \cite{courant-friedrichs} for more details). Therefore, we see that
\begin{eqnarray}
v(b_0)-u(b_0)b_0=0, \qquad (v(\sigma)- \sigma u(\sigma))\big|_{\{s_{0}<\sigma<b_{0}\}}\neq0.
\label{eq:3.26}
\end{eqnarray}

\vspace{10pt}
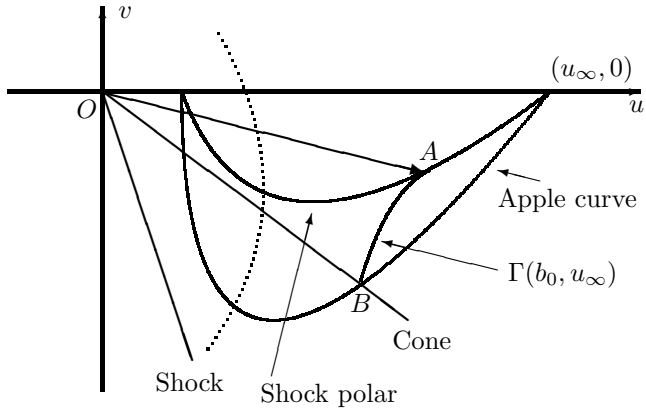
\begin{figure}[ht]
\begin{center}
\setlength{\unitlength}{0.7mm}
\begin{picture}(50,63)(45,-35)
\linethickness{1.2pt}
\put(7,7){\vector(1,0){120}} \put(25,-50){\vector(0,1){73}}
\thicklines
\qbezier(40,7)(55,-35)(110,7)
\qbezier(40,7)(40,-80)(110,7)
\qbezier(74,-29)(79,-12)(87, -8)
\put(25.5,6.7){\vector(4,-1){61}}
\put(25,7){\line(4,-3){58}}
\put(25,7){\line(1,-3){17}}
\thinlines
\put(57,-47){\vector(1,4){7.7}}
\put(100,-27){\vector(-4,1){22}}
\put(108,-10){\vector(-4,1){8}}
\qbezier(74,-29)(79,-12)(87, -8)
\thicklines
\qbezier[50](45,-42)(65,-14)(47,18)
\put(125,3.2){$u$}
\put(28,21){$v$}
\put(20,2){$O$}
\put(85,-6){$A$}
\put(72,-35){$B$}
\put(110,10){$(u_{\infty}, 0)$}
\put(35,-50){Shock}
\put(55,-52){Shock polar}
\put(80,-42){Cone}
\put(102,-29){$\Gamma(b_{0}, u_{\infty})$}
\put(100,-15){Apple curve}
\end{picture}
\end{center}\vspace{45pt}
\caption{The apple curve and the shock polar for the self-similar solutions}\label{fig3.2}
\end{figure}

Indeed, we have the following lemma.

\smallskip
\begin{lemma}\label{lem:2.10}
For state $(u(s_{0}), v(s_{0}), \rho(s_{0}))$ on the shock polar
through $(u_{\infty}, 0, \rho_\infty)$ with speed $s_0$,
\begin{eqnarray}
u(s_{0})>0, \quad\,\, v(s_{0})<0,\quad\,\, v(s_{0})-s_{0}u(s_{0})>0,\label{eq:3.27}
\end{eqnarray}
so that
\begin{eqnarray}
v(\sigma)-\sigma u(\sigma)>0 \qquad \mbox{for ${s_{0}<\sigma<b_{0}}$}.\label{eq:3.28}
\end{eqnarray}
\end{lemma}
\begin{proof}
Using the Rankine-Hugoniot relations for \eqref{eq:3.1} and Lemma \ref{lem:2.8} and noting that $s_{0}<0$,
we have
\begin{eqnarray*}
&&u(s_{0})=\frac{1}{1+s^{2}_{0}}\Big(1+\frac{s_0^2}{\rho(s_{0})}\Big)u_{\infty}>0,\\[1mm]
&&v(s_{0})=\frac{s_{0}}{1+s^{2}_{0}}\Big(1-\frac{1}{\rho(s_{0})}\Big)u_{\infty}<0,\\[1mm]
&&v(s_{0})-s_{0}u(s_{0})=-\frac{s_{0}u_{\infty}}{\rho(s_{0})}>0.
\end{eqnarray*}
Hence, by \eqref{eq:3.26}, we obtain \eqref{eq:3.28} for ${s_{0}<\sigma<b_{0}}$.
\end{proof}

\smallskip
Now we state some properties about the self-similar solutions of problem \eqref{eq:3.1}--\eqref{eq:3.2}.

\smallskip
\begin{lemma}\label{lem:2.11}
For $\sigma \in (s_{0}, b_{0})$, the free boundary problem \eqref{eq:3.1}--\eqref{eq:3.2}
admits a unique solution $(u(\sigma),v(\sigma),\rho(\sigma))$ that satisfies the following properties{\rm :}
\begin{eqnarray}
&& c^2(1+\sigma^{2})-\big(v(\sigma)-\sigma u(\sigma)\big)^{2}>0,\label{eq:3.29}\\[5pt]
&& v(\sigma)<0,\quad\,\,\rho({\sigma})>0,\label{eq:3.30}\\[5pt]
&& u_{\sigma}(\sigma)<0,\quad\,\, v_{\sigma}(\sigma)<0,\quad\,\, \rho_{\sigma}(\sigma)>0. \label{eq:3.31}
\end{eqnarray}
\end{lemma}

\begin{proof} We divide the proof into four steps.

\medskip
1. Lemma \ref{lem:2.7} implies that the straight line $v=bu$ intersects the shock polar through $(u_{\infty}, 0)$.
Then, from the structure of the apple curve given in \cite{courant-friedrichs}, problem \eqref{eq:3.1}--\eqref{eq:3.2}
has a unique solution $(u(\sigma), v(\sigma),\rho(\sigma))$.

\medskip
2. We now prove \eqref{eq:3.29}--\eqref{eq:3.31}. Define
\begin{eqnarray*}
\sigma_{*}:=\sup\Big\{\sigma_{0}\,:\, 0<\frac{v(\sigma)-\sigma u(\sigma)}{\sqrt{1+\sigma^{2}}}<c, \ v(\sigma)<0,\  \sigma \in[s_{0}, \sigma_{0}]\Big\}.
\end{eqnarray*}
By Lemmas \ref{lem:2.9}--\ref{lem:2.10}, we have
\begin{eqnarray*}
0<\frac{v(s_{0})-s_{0}u(s_{0})}{\sqrt{1+s^{2}_{0}}}<c,\qquad v(s_0)<0.
\end{eqnarray*}
Therefore, $\sigma_{*}\geq s_{0}$.

\smallskip
3. We now prove $\sigma_{*}\geq b_{0}$.  On the contrary, assume that $\sigma_{*}<b_{0}$.
Then
\begin{eqnarray*}
&&\Big(\frac{v(\sigma_{*})-\sigma_{*} u(\sigma_{*})}{\sqrt{1+\sigma_{*}^{2}}}-c\Big)v(\sigma_{*})=0,\label{eq:3.32}\\
&&v(\sigma)-\sigma u(\sigma)>0,\  \ 0<\frac{v(\sigma)-\sigma u(\sigma)}{\sqrt{1+\sigma^{2}}}< c \qquad\mbox{for $\sigma\in[s_0, \sigma_{*})$}.
\end{eqnarray*}
By \eqref{eq:3.3}, we have
\begin{eqnarray*}
&&u_{\sigma}(\sigma)<0,\quad v_{\sigma}(\sigma)<0, \quad \rho_{\sigma}(\sigma)>0  \quad\quad\,\,\,\,\mbox{for $\sigma \in[s_{0}, \sigma_{*})$}.
\end{eqnarray*}
Denote
\begin{eqnarray*}
h(\sigma):=\frac{ v(\sigma)-\sigma u(\sigma)}{\sqrt{1+\sigma^{2}}}.
\end{eqnarray*}
Then
\begin{eqnarray*}
h'(\sigma)
=\frac{v_{\sigma}(\sigma)-\sigma u_{\sigma}(\sigma)}{\sqrt{1+\sigma^{2}}}-\frac{u(\sigma)+\sigma v(\sigma)}{(1+\sigma^{2})^{\frac{3}{2}}} <0
\qquad\,\,\mbox{for $\sigma \in [s_{0}, \sigma_{*})$},
\end{eqnarray*}
which implies that
\begin{eqnarray*}
\frac{v(\sigma_{*})-\sigma_{*} u(\sigma_{*})}{\sqrt{1+\sigma_{*}^{2}}}<\frac{ v(s_{0})-s_{0} u(s_{0})}{\sqrt{1+s_{0}^{2}}}
<c\,v(\sigma_{*})<v(s_{0})<0.
\end{eqnarray*}
This leads to a contradiction to \eqref{eq:3.32}. Thus, $\sigma_{*}=b_{0}$.

\medskip
4. From \eqref{eq:3.1}--\eqref{eq:3.2}, we have
\begin{eqnarray*}
\rho({\sigma})>0, \quad u_{\sigma}(\sigma)<0,\quad v_{\sigma}(\sigma)<0 , \quad \rho_{\sigma}(\sigma)>0.
\end{eqnarray*}
This completes the proof.
\end{proof}

\smallskip
\begin{lemma}\label{lem:2.12}
For sufficiently large $M_{\infty}$,
solution $\big(u(\sigma),v(\sigma),\rho(\sigma)\big)$ of \eqref{eq:3.1}--\eqref{eq:3.2} satisfies
\begin{eqnarray}
&& s_{0}=b_{0}+O(1)e^{-m_{0}M^{2}_{\infty}},\label{eq:3.33}\\[5pt]
&& u(\sigma)=\Big(\frac{1}{1+b^{2}_{0}}+O(1)e^{-m_{0}M^{2}_{\infty}}\Big)u_{\infty},\label{eq:3.34}\\[1mm]
&& v(\sigma)=\Big(\frac{b_{0}}{1+b^{2}_{0}}+O(1)e^{-m_{0}M^{2}_{\infty}}\Big)u_{\infty},\label{eq:3.35}\\[1mm]
&& \rho(\sigma)=\exp\big\{m_0M^{2}_{\infty}\big(1+O(1)e^{-2m_{0}M^{2}_{\infty}}\big)\big\}\label{eq:3.36}
\end{eqnarray}
for $\sigma\in [s_0, b_0]$,
where $m_{0}:=\frac{b^{2}_{0}}{2(1+b^{2}_{0})}$, and the bound of $O(1)$ is independent of $M_{\infty}$.
In particular, the largeness of $M_{\infty}$ implies that
\begin{eqnarray}
u(\sigma)>c\qquad\,\,\mbox{for any $\sigma\in (s_{0}, b_{0})$}. \label{eq:3.37}
\end{eqnarray}
\end{lemma}

\begin{proof}
We first prove (3.28). To do this, for given $b_0<0$, consider problem \eqref{eq:3.4a}--\eqref{eq:3.4d}
of the planar shock polar solution with $b_{+}=b_0$.
Then $m_{+}=m_0$.
By Lemma 3.1, we know that solution $s_{+}$ of problem \eqref{eq:3.4a}--\eqref{eq:3.4d} satisfies
\begin{eqnarray}\label{eq:3.31a}
s_{+}>b_0-\tilde{K}' e^{-m_0 M^{2}_{\infty}},
\end{eqnarray}
where
$\tilde{K}'>0$ independent of $M_{\infty}$.

In order to obtain the conical shock with slope $s_0$, we set
\begin{eqnarray*}
b_{1}=\frac{v(s_{0})}{u(s_{0})}.
\end{eqnarray*}
Then, by Lemma \ref{lem:2.10}, we have
\begin{eqnarray*}
b_{0}u(s_{0})\leq b_{0}u(b_0)=v(b_{0})\leq v(s_{0})=b_{1}u(s_{0}),
\end{eqnarray*}
which leads to
\begin{eqnarray*}
b_{0}\leq b_{1}\leq 0.
\end{eqnarray*}
By Lemma \ref{lem:3.4a}, we can further deduce that
\begin{eqnarray}\label{eq:3.31b}
s_{+}<s_0{\color{red}.}
\end{eqnarray}

Since $s_{0}<b_{0}$, we obtain by estimates \eqref{eq:3.31a}--\eqref{eq:3.31b} that
\begin{eqnarray}\label{eq:3.31c}
s_{0}=b_{0}+O(1)e^{-m_{0}M^{2}_{\infty}}.
\end{eqnarray}
On the other hand, in the same way as in the proof of Lemma 3.1, we can prove
\begin{eqnarray*}
\varphi(s_{0}, b_{1})=0.
\end{eqnarray*}
Then, by  Lemma 3.2, we have
\begin{eqnarray}\label{eq:3.31d}
\begin{split}
&&b_{1}=s_{0}+O(1)e^{-m_{0}M^{2}_{\infty}}=b_{0}+O(1)e^{-m_{0}M^{2}_{\infty}}.
\end{split}
\end{eqnarray}
Since $(u(s_{0}), v(s_{0}))$ solves the equations:
\begin{eqnarray*}
u(s_{0})+s_{0}v(s_{0})=u_{\infty},\qquad u(s_{0})b_{1}-v(s_{0})=0,
\end{eqnarray*}
then, employing estimates \eqref{eq:3.31c}--\eqref{eq:3.31d}, we have
\begin{eqnarray*}
\begin{split}
 u(s_{0})&=\frac{1}{1+b_{1}s_{0}}=\Big(\frac{1}{1+b^{2}_{0}}+O(1)e^{-m_{0}M^{2}_{\infty}}\Big)u_{\infty},\\
 v(s_{0})&=\frac{b_{1}}{1+b_{1}s_{0}}\Big(\frac{b_{0}}{1+b^{2}_{0}}+O(1)e^{-m_{0}M^{2}_{\infty}}\Big)u_{\infty}.
\end{split}
\end{eqnarray*}
Therefore, using the monotonicity of $(u(\sigma), v(\sigma))$ again that
\begin{eqnarray*}
&& u(s_{0})b_{0}\leq u(\sigma)b_{0}\leq u(b_0)b_{0}=v(b_{0})\leq v(\sigma)\leq v(s_{0}),
\end{eqnarray*}
we derive estimates \eqref{eq:3.34}--\eqref{eq:3.35}.
Finally, by the Bernoulli law, together with the estimates of
$(u(\sigma), v(\sigma))$, we can obtain \eqref{eq:3.36}.
Moreover, for $M_\infty$ sufficiently large, by \eqref{eq:3.34},
we can obtain \eqref{eq:3.37}.
\end{proof}

\smallskip
\begin{lemma}\label{lem:2.13}
For $M_{\infty}$ sufficiently large, the following
asymptotic expansions hold{\rm :} For any $\sigma\in [s_{0}, b_{0}]$,
\begin{eqnarray}
&&\qquad \lambda_{1}(\sigma)=b_{0}-(1+b^{2}_{0})^{\frac{3}{2}}M_{\infty}^{-1}
+O(1)M^{-2}_{\infty}+O(1)e^{-m_{0}M^{2}_{\infty}},\label{eq:3.38}\\[1mm]
&&\qquad \lambda_{2}(\sigma)=b_{0}+ (1+b^{2}_{0})^{\frac{3}{2}}M_{\infty}^{-1}+O(1)M^{-2}_{\infty}+
O(1)e^{-m_{0}M^{2}_{\infty}},\label{eq:3.39}\\[1mm]
&&\qquad \frac{e_1(\sigma)}{u_{\infty}}=\frac{1}{(1+b^{2}_{0})^{2}}+ 3b_{0}(1+b^{2}_{0})^{-\frac{3}{2}}M_{\infty}^{-1}
+O(1)M^{-2}_{\infty}+O(1)e^{-m_{0}M^{2}_{\infty}},\label{eq:3.40}\\[1mm]
&&\qquad \frac{e_2(\sigma)}{u_{\infty}}=\frac{1}{(1+b^{2}_{0})^{2}}
- 3b_{0}(1+b^{2}_{0})^{-\frac{3}{2}}M_{\infty}^{-1}+O(1)M^{-2}_{\infty}
+O(1)e^{-m_{0}M^{2}_{\infty}},\label{eq:3.41}\\[1mm]
&&\qquad \frac{e_1(\sigma)}{e_2(\sigma)}=1+ 6b_{0}(1+b^{2}_{0})^{\frac{1}{2}}\,M_{\infty}^{-1}
+O(1)M^{-2}_{\infty}+O(1)e^{-m_{0}M^{2}_{\infty}},\label{eq:3.42}
\end{eqnarray}
where $e_{j}(\sigma)=e_{j}(U(\sigma)), \ j=1,2$, and the universal bound of $O(1)$ is independent of $M_{\infty}$.
\end{lemma}

\begin{proof}
By Lemma \ref{lem:2.10} and Taylor's formula
when $M_{\infty}$ is large enough, we know that
\begin{align*}
\lambda_{1}(\sigma)&=\frac{\frac{v(\sigma)}{u(\sigma)}-\frac{1}{\sqrt{M^2-1}}}
{1+\frac{v(\sigma)}{u(\sigma)}\frac{1}{\sqrt{M^2-1}}}
=\frac{b_0- \sqrt{1+b^{2}_0}M_\infty^{-1}+O(1)e^{-m_{0}M^{2}_{\infty}}}
{1+ b_0\sqrt{1+b^{2}_0}M_\infty^{-1}+O(1)e^{-m_{0}M^{2}_{\infty}}}\nonumber \\[1mm]
&=b_{0}- (1+b^{2}_{0})^{\frac{3}{2}}M_{\infty}^{-1}
+O(1)M^{-2}_{\infty}+O(1)e^{-m_{0}M^{2}_{\infty}}.
\end{align*}
The proof of \eqref{eq:3.39} is entirely similar.

Next, for $e_{j}(\sigma)$ in \eqref{eq:2.2} with $j=1,2$, using the same method again when $M_{\infty}$ is large enough,
we have
\begin{align*}
\frac{e_{1}(\sigma)}{u_\infty}&=\sqrt{\frac{M^{2}}{M^2_\infty}-\frac{1}{M^2_\infty}}
\bigg(\frac{u}{q}\sqrt{1-\frac{1}{M^2}}+\frac{v}{q}\frac{1}{M}\bigg)^3 \nonumber \\[1mm]
&=\sqrt{\frac{1}{1+b^{2}_0}-\frac{1}{M^2_\infty}}
\bigg(\sqrt{\frac{1}{1+b^{2}_0}}\sqrt{1-\frac{(1+b^{2}_0)}{M^2_\infty}}+\frac{b_0}{M_\infty}\bigg)^3
+O(1)e^{-m_{0}M^{2}_{\infty}} \nonumber \\[1mm]
&=\frac{1}{(1+b^{2}_{0})^{2}}+ 3b_{0}(1+b^{2}_{0})^{-\frac{3}{2}}M_{\infty}^{-1}
+O(1)M^{-2}_{\infty}+O(1)e^{-m_{0}M^{2}_{\infty}}.
\end{align*}
The proof of $\frac{e_{2}(\sigma)}{u_\infty}$ is similar.
Finally, we combine \eqref{eq:3.40} with \eqref{eq:3.41} directly to obtain \eqref{eq:3.42}.
This completes the proof.
\end{proof}

\smallskip
\begin{lemma}\label{lem:2.14}
For $M_{\infty}$ sufficiently large, the following
estimates hold{\rm :} For any $\sigma\in [s_{0}, b_{0}]$,
\begin{eqnarray}
&&\qquad\quad\,  u_\sigma(\sigma)=\Big(\frac{b_{0}}{(1+b^{2}_{0})^{2}}+O(1)e^{-m_{0}M^{2}_{\infty}}\Big)u_{\infty},\label{eq:3.43}\\[1mm]
&&\qquad\quad\, v_\sigma(\sigma)=-\Big(\frac{1}{(1+b^{2}_{0})^{2}}+O(1)e^{-m_{0}M^{2}_{\infty}}\Big)u_{\infty},\label{eq:3.44}\\[1mm]
&&\qquad\quad\, u_\sigma(\sigma)+\lambda_{1}(\sigma)v_\sigma(\sigma)
 =\frac{c}{\sqrt{1+b^{2}_{0}}}+O(1)M_{\infty}e^{-m_{0}M^{2}_{\infty}}
   +O(1)M^{-1}_{\infty}e^{-m_{0}M^{2}_{\infty}},\label{eq:3.45}\\[1mm]
&&\qquad\quad\, u_\sigma(\sigma)+\lambda_{2}(\sigma)v_\sigma(\sigma)
=-\frac{c}{\sqrt{1+b^{2}_{0}}}+O(1)M_{\infty}e^{-m_{0}M^{2}_{\infty}}
 +O(1)M^{-1}_{\infty}e^{-m_{0}M^{2}_{\infty}},\label{eq:3.46}
\end{eqnarray}
where the bound of $O(1)$ is independent of $M_{\infty}$.
\end{lemma}

\smallskip
\begin{proof}
According to \eqref{eq:3.3}, Lemma \ref{lem:2.10}, and Taylor's formula, we have
\begin{align*}
u_\sigma(\sigma)&=\frac{c^2 v(\sigma)}{c^2(1+\sigma^{2})-(\sigma u(\sigma)-v(\sigma))^{2}}
=\frac{\frac{c^2b_{0}}{1+b^{2}_{0}}+O(1)e^{-m_{0}M^{2}_{\infty}}}{c^2\big(1+b^{2}_{0}\big)+O(1)e^{-m_{0}M^{2}_{\infty}}} u_{\infty}\\
&=\Big(\frac{b_{0}}{(1+b^{2}_{0})^{2}}+O(1)e^{-m_{0}M^{2}_{\infty}}\Big)u_{\infty}.
\end{align*}
On the other hand, since
$v_\sigma(\sigma)=-\frac{1}{\sigma}u_\sigma(\sigma)$,
we finally obtain \eqref{eq:3.44}.

By Lemma \ref{lem:2.11} and a direct computation,
\begin{eqnarray*}
&&u_\sigma(\sigma)+\lambda_{1}(\sigma)v_\sigma(\sigma)\\
&&=\Big(\frac{b_{0}}{(1+b^{2}_{0})^{2}}+O(1)e^{-m_{0}M^{2}_{\infty}}\Big)u_{\infty}
\\
&& \quad -\Big(\frac{1}{(1+b^{2}_{0})^{2}}+O(1)e^{-m_{0}M^{2}_{\infty}}\Big)
\Big(b_{0}-(1+b^{2}_{0})^{\frac{3}{2}}\,M_{\infty}^{-1}
+O(1)M^{-2}_{\infty}+O(1)e^{-m_{0}M^{2}_{\infty}}\Big)u_{\infty}\\
&& =\frac{c}{\sqrt{1+b^{2}_{0}}}+O(1)M_{\infty}e^{-m_{0}M^{2}_{\infty}}+O(1)M^{-1}_{\infty}e^{-m_{0}M^{2}_{\infty}}.
\end{eqnarray*}
In the same way, we can prove \eqref{eq:3.46}. This completes the proof.
\end{proof}

\section{Riemann Solutions for the Homogeneous System}
In this section, we analyze the solutions of the Riemann problem for the homogeneous system \eqref{eq:2.1}
with piecewise constant initial data:
\begin{equation}
\left. U\right|_{\{x=x_{0}\}}
=\begin{cases}
      U_{a} \qquad\mbox{for $y>y_{0}$},\\[1mm]
      U_{b} \qquad\mbox{for $y<y_{0}$},
\end{cases}\label{eq:4.1}
\end{equation}
where the constant states $U_{a}$ and $U_{b}$ denote the $above$ state and $below$
state with respect to line $y=y_{0}$, respectively, which are near the states
of the background conical flow.

\subsection{Riemann problem involving only weak waves}

Denote by $\Gamma(b_{0}, u_{\infty})$ the curve formed by the states on the conical
flow constructed in \S 3, so that $\Gamma(b_{0}, u_{\infty})$ is the curve formed
by state $(u(\sigma), v(\sigma))^{\top}$ that is the
solution of \eqref{eq:3.1}--\eqref{eq:3.2}.
Then, on the solution curve $\Gamma(b_{0}, u_{\infty})$ of the conical
flow, we have the following properties.

\smallskip
\begin{lemma}\label{lem:2.16}
If $U_{b}\in \Gamma(b_{0}, u_{\infty})$, then
\begin{eqnarray}
&&\quad \lim_{M_{\infty}\rightarrow\infty}\frac{\det\big(r_{1}(U_{b}),r_{2}(U_{b}) \big)}{M_{\infty}}
=\frac{2c^2}{(1+b^{2}_{0})^{\frac{5}{2}}},\\ [1.5mm]\label{eq:4.2}
&&\quad \lim_{M_{\infty}\rightarrow\infty}
\frac{\det\big(r_{1}(U_{b}),r_{2}(U_{b})\big)}{\big((u_{b})_\sigma+\lambda_{j}(\sigma)(v_{b})_\sigma\big)M_\infty}
=(-1)^{j+1}\frac{2c}{(1+b^{2}_{0})^{2}},\qquad\, j=1, 2. \label{eq:4.3}
\end{eqnarray}
\end{lemma}

\begin{proof}
By Lemma \ref{lem:2.11}, we have
\begin{eqnarray*}\begin{aligned}
\frac{\det\big(r_{1}(U_{b}),r_{2}(U_{b})\big)}{M_{\infty}}&=
\frac{c\, e_{1}(U_{b})e_{2}(U_{b})}{u_{\infty}}\big(\lambda_{2}(U_{b})-\lambda_{1}(U_{b})\big) \\[1.5mm]
&=\frac{e_{1}(U_{b})}{u_{\infty}}\frac{e_{2}(U_{b})}{u_{\infty}}\Big(2 c^2(1+b^{2}_0)^{\frac{3}{2}}
+O(1)M^{-1}_{\infty}+O(1)M_{\infty}e^{-m_{0}M^{2}_{\infty}}\Big).
\end{aligned}
\end{eqnarray*}
Then it follows that
\begin{eqnarray*}
\lim_{M_{\infty}\rightarrow \infty}\frac{\det\big(r_{1}(U_{b}),r_{2}(U_{b})\big)}{M_{\infty}}
=\frac{2c^2}{(1+b^{2}_0)^{\frac{5}{2}}}.
\end{eqnarray*}

Next, we turn to the proof of \eqref{eq:4.3}.
By Lemma \ref{lem:2.14}, we see that, for $j=1$,
\begin{eqnarray*}
\lim_{M_{\infty}\rightarrow \infty}
\frac{\det\big(r_{1}(U_{b}),r_{2}(U_{b})\big)}{\big((u_{b})_\sigma+\lambda_{1}(\sigma)(v_{b})_\sigma\big)M_{\infty}}
=\frac{2c^2}{(1+b^{2}_{0})^{\frac{5}{2}}}\frac{\sqrt{1+b^{2}_{0}}}{c}
=\frac{2c}{(1+b^{2}_{0})^{2}}.
\end{eqnarray*}
 The proof for $j=2$ is similar.
\end{proof}

Using the results in  \cite{zhang} (see also \cite{chen-zhang-zhu,smoller}) and Lemma \ref{lem:2.16},
we have the following solvability result.

\smallskip
\begin{proposition}\label{prop:4.1}
Given states $\Gamma(b_{0}, u_{\infty})$ defined above, then, for $M_\infty$ sufficiently large,
there exists a small constant $\hat{\varepsilon}>0$ such that,
for any states $U_{b}$ and $U_{a}$ lying in $O_{\hat{\varepsilon}}(\Gamma(b_{0}, u_{\infty}))$
with radius $\hat{\varepsilon}$ and center $\Gamma(b_{0}, u_{\infty})$,
the Riemann problem \eqref{eq:2.1} and \eqref{eq:4.1} admits  a unique admissible solution consisting of
at most two elementary waves{\rm :} one for the $1$-characteristic field and the other for the $2$-characteristic field.
Moreover, states $U_{b}$ and $U_{a}$ can be connected by
\begin{eqnarray}
U_{a}=\Phi_{2}(\varepsilon_{2};\Phi_{1}(\varepsilon_{1}; U_{b}))\label{eq:4.4}
\end{eqnarray}
with $\Phi_{j}\in C^{2}$,  $\Phi_{j}|_{\varepsilon_{j}=0}= U_{b}$,
and $\frac{\partial\Phi_{j}}{\partial\varepsilon_{j}}\big|_{\varepsilon_{j}=0}=r_{j}(U_{b})$ for $j=1, 2$.
\end{proposition}

\smallskip
\begin{remark}\label{rem:4.1}
For simplicity, we set
\begin{eqnarray}
\Phi(\varepsilon_{1},\varepsilon_{2}; U_{b})=\Phi_{2}(\varepsilon_{2};\Phi_{1}(\varepsilon_{1}; U_{b})),\label{eq:4.5}
\end{eqnarray}
and denote $\{U_{b}, U_{a}\}$ as the solution of the following equation{\rm :}
\begin{eqnarray}
U_{a}=\Phi(\varepsilon_{1},\varepsilon_{2}; U_{b}), \label{eq:4.6}
\end{eqnarray}
that is, $\{U_{b}, U_{a}\}=\{\varepsilon_{1},\varepsilon_{2}\}$ throughout the paper.
\end{remark}

For the statements above,  the following interaction estimate was given in Glimm \cite{glimm} for weak waves
(also see \cite{wang-zhang,yong,zhang}).

\smallskip
\begin{lemma}\label{lem:4.2}
Let $U_{b}\in \Gamma(b_{0}, u_{\infty})$, $\alpha$, $\beta$, and $\gamma$ satisfy
\begin{eqnarray}\vspace{5pt}
\Phi(\gamma; U_{b})=\Phi(\alpha;\Phi(\beta; U_{b})).\label{eq:4.7}
\end{eqnarray}
Then
\begin{eqnarray}
\gamma=\alpha+\beta+O(1)Q^{0}(\alpha,\beta), \label{eq:4.8}
\end{eqnarray}
where
$$
Q^{0}(\alpha,\beta)=\sum\{|\alpha_{i}||\beta_{j}|\,:\, \alpha_{i}\ and\ \beta_{j} \ approach\},
$$
and $O(1)$ depends continuously on $M_{\infty}< \infty$.
\end{lemma}

\subsection{Riemann problem involving a strong leading shock-front}
Denote by $S_{1}(U_{\infty})$  the part of the shock polar corresponding
to the $1$-characteristic field. Let
\begin{eqnarray*}
S^{-}_{1}(U_{\infty})=\big\{(u,v)\in S_{1}(U_{\infty})\,:\, c^2\leq u^{2}+v^{2}\leq u^{2}_{\infty},\ v<0\big\}
\qquad \mbox{for $U_{\infty}=(u_{\infty},0)^{\top}$}.
\end{eqnarray*}

Following the ways in \cite{wang-zhang, zhang} in a neighborhood $O_{\hat{\varepsilon}}(\Gamma(b_{0},u_{\infty}))$ of $\Gamma(b_{0},u_{\infty})$,
we can parameterize the shock polar $S^{-}_{1}(U_{\infty})\cap O_{\hat{\varepsilon}}(\Gamma(b_{0},u_{\infty}))$
for the homogeneous system $\eqref{eq:2.1}$
through $U_{\infty}$ by a $C^{2}$--function $\Theta: s\mapsto \Theta(s,U_\infty)$,
that is, $\Theta(s,U_\infty)$ is the state that can be connected to $U_\infty$ by a shock
with slope
$s$ and left-state $U_\infty$.
In the following, we write $\Theta(s,U_\infty)$ as $\Theta(s)$ for simplification,
and denote by $\tilde{u}(s)$ and $\tilde{v}(s)$
the components of $\Theta(s)$, {\it i.e.}, $\Theta(s)=(\tilde{u}(s), \tilde{v}(s))^{\top}$.
Moreover, on the shock polar, we have the following.

\smallskip
\begin{lemma}\label{lem:2.17}
For $M_{\infty}$ sufficiently large, the following expansions hold{\rm :}
\begin{eqnarray}
&&\frac{\tilde{u}(s_0)}{u_{\infty}}=\frac{1}{1+b^{2}_0}+O(1)e^{-m_{0}M^{2}_{\infty}},\label{eq:4.9}\\[1mm]
&&\frac{\tilde{v}(s_0)}{u_{\infty}}=\frac{b_{0}}{1+b^{2}_0}+O(1)e^{-m_{0}M^{2}_{\infty}},\label{eq:4.10}\\[1mm]
&&\frac{\tilde{u}_{s}(s_0)}{u_{\infty}}=-\frac{2b_0}{(1+b^{2}_0)^{2}}
+O(1)M^{2}_{\infty}e^{-m_{0}M^{2}_{\infty}} +O(1)e^{-m_{0}M^{2}_{\infty}},\label{eq:4.11}\\[1mm]
&&\frac{\tilde{v}_{s}(s_0)}{u_{\infty}}=\frac{1-b^{2}_{0}}{(1+b^{2}_0)^{2}}
+O(1)M^{2}_{\infty}e^{-m_{0}M^{2}_{\infty}} +O(1)e^{-m_{0}M^{2}_{\infty}},\label{eq:4.12}
\end{eqnarray}
and, for $j=1,2$,
\begin{eqnarray}
&\,\qquad\quad\, \frac{\tilde{u}_{s}(s_0)+\lambda_{j}(s_0)\tilde{v}_{s}(s_0)}{u_{\infty}}
&\,=-\frac{b_0}{1+b^{2}_0}+(-1)^{j}(1-b^{2}_0)(1+b^{2}_0)^{-\frac{1}{2}}\,M_{\infty}^{-1}+O(1)M^{-2}_{\infty}\label{eq:4.13}\\[1.5mm]
&&\,\quad +O(1)M^{2}_{\infty}e^{-m_{0}M^{2}_{\infty}}+O(1)e^{-m_{0}M^{2}_{\infty}},\nonumber\\[2mm]
&\,\qquad\quad\, \frac{\tilde{u}_{s}(s_0)+\lambda_{1}(s_0)\tilde{v}_{s}(s_0)}{\tilde{u}_{s}(s_0)+\lambda_{2}(s_0)\tilde{v}_{s}(s_0)}
&\,= 1+ 2b^{-1}_{0}(1-b^{2}_0)\sqrt{1+b^{2}_0}\,M_{\infty}^{-1}+O(1)M^{-2}_{\infty}\label{eq:4.14}\\[1.5mm]
&&\,\quad +O(1)M^{2}_{\infty}e^{-m_{0}M^{2}_{\infty}}+O(1)e^{-m_{0}M^{2}_{\infty}},\nonumber
\end{eqnarray}
where $\tilde{u}_{s}(s_0)=\frac{\partial\tilde{u}}{\partial s}(s_0),
\tilde{v}_{s}(s_0)=\frac{\partial\tilde{v}}{\partial s}(s_0)$,
and the bound of $O(1)$ is independent of $M_{\infty}$.
\end{lemma}

\smallskip
\begin{proof}
The first two expansions are directly from Lemma \ref{lem:2.10}.
To obtain the other expansions, we first
see that, on the shock polar, the Rankine-Hugoniot conditions hold:
\begin{eqnarray}
&&\tilde{\rho}(s)\big(\tilde{u}(s)s-v(s)\big)=u_{\infty}s,\ \label{eq:4.15}\\[1mm]
&&\tilde{u}(s)+ \tilde{v}(s)s = u_{\infty},\label{eq:4.16}
\end{eqnarray}
and the Bernoulli law
\begin{eqnarray}
\frac{\tilde{u}^2(s)+\tilde{v}^2(s)}{2}+c^2\ln\tilde{\rho}(s)=\frac{u^2_\infty}{2}.
\label{eq:4.17}
\end{eqnarray}
We take the derivative of \eqref{eq:4.15}--\eqref{eq:4.17} with respect to $s$ and then let $s=s_0$ to obtain
\begin{eqnarray*}
A_{11}(s_{0})\frac{\tilde{u}_s(s_0)}{u_\infty}+A_{12}(s_0)\frac{\tilde{v}_s(s_0)}{u_\infty}=B_{1}(s_{0}),\\[1mm]
A_{21}(s_{0})\frac{\tilde{u}_s(s_0)}{u_\infty}+A_{22}(s_{0})\frac{\tilde{v}_s(s_0)}{u_\infty}=B_2(s_0),
\end{eqnarray*}
where
\begin{eqnarray*}
&&A_{11}(s_0)=\frac{s_0(\tilde{u}^{2}(s_0)-c^2)-\tilde{u}(s_0)\tilde{v}(s_0)}{c^2},
  \quad A_{12}(s_0)=\frac{c^2-\tilde{v}^2(s_0) +s_0\tilde{u}(s_0)\tilde{v}(s_0)}{c^2},\\[1mm]
&&A_{21}(s_0)=-1,\quad A_{22}(s_0)=-s_0, \quad B_{1}(s_0)=\frac{\tilde{\rho}(s_0) \tilde{u}(s_0)+u_\infty}{\tilde{\rho}(s_0)u_\infty},
 \quad B_{2}(s_0)=\frac{\tilde{v}(s_0)}{u_{\infty}}.
\end{eqnarray*}
For $M_\infty$ sufficiently large, it follows from \eqref{eq:4.9}--\eqref{eq:4.10} that
\begin{eqnarray*}
&&A_{11}(s_0)=-b_0+O(1)M^{2}_{\infty}e^{-m_{0}M^{2}_{\infty}}+O(1)e^{-m_{0}M^{2}_{\infty}},\\[1mm]
&&A_{12}(s_0)=1+O(1)M^{2}_{\infty}e^{-m_{0}M^{2}_{\infty}},\quad\,\,
   A_{22}(s_0)=-b_0+O(1)e^{-m_{0}M^{2}_{\infty}},\\[1mm]
&&B_{1}(s_0)=\frac{1}{1+b^{2}_0}+O(1)e^{-m_{0}M^{2}_{\infty}},\quad\,\,
 B_{2}(s_0)=\frac{b_0}{1+b^{2}_0}+O(1)e^{-m_{0}M^{2}_{\infty}}.
\end{eqnarray*}
Thus, by Cramer's rule, we have
\begin{align*}
\frac{\tilde{u}_s(s_0)}{u_\infty}&=\frac{A_{22}(s_0)B_{1}(s_0)-A_{12}(s_0)B_{2}(s_0)}{A_{11}(s_0)A_{22}(s_0)-A_{12}(s_0)A_{21}(s_0)}\\[1mm]
&=-\frac{2b_0}{(1+b^{2}_0)^{2}}+O(1)M^{2}_{\infty}e^{-m_{0}M^{2}_{\infty}}
  +O(1)e^{-m_{0}M^{2}_{\infty}},\\[1.5mm]
\frac{\tilde{v}_s(s_0)}{u_\infty}&=\frac{A_{11}(s_0)B_{2}(s_0)-A_{21}(s_0)B_{1}(s_0)}{A_{11}(s_0)A_{22}(s_0)-A_{12}(s_0)A_{21}(s_0)}\\[1mm]
&=\frac{1-b^{2}_0}{(1+b^{2}_0)^{2}}+O(1)M^{2}_{\infty}e^{-m_{0}M^{2}_{\infty}}
+O(1)e^{-m_{0}M^{2}_{\infty}}.
\end{align*}
Next, for $j=1$, we use Lemma \ref{lem:2.11} and \eqref{eq:4.11}--\eqref{eq:4.12} to obtain
\begin{align*}
\frac{\tilde{u}_{s}(s_0)+\lambda_{j}(s_0)\tilde{v}_{s}(s_0)}{u_{\infty}}
&=-\frac{2b_0}{(1+b^{2}_0)^{2}}+O(1)M^{2}_{\infty}e^{-m_{0}M^{2}_{\infty}}+O(1)e^{-m_{0}M^{2}_{\infty}}\\[1mm]
&\quad\, +\Big(b_{0}- (1+b^{2}_{0})^{\frac{3}{2}}M_{\infty}^{-1}
+O(1)M^{-2}_{\infty}+O(1)e^{-m_{0}M^{2}_{\infty}}\Big)\\[1mm]
&\qquad\,\, \times\Big(\frac{1-b^{2}_{0}}{(1+b^{2}_0)^{2}}
+O(1)M^{2}_{\infty}e^{-m_{0}M^{2}_{\infty}}+O(1)e^{-m_{0}M^{2}_{\infty}}\Big)\\[1mm]
& =-\frac{b_0}{1+b^{2}_0}-(1-b^{2}_0)(1+b^{2}_0)^{-\frac{1}{2}}M_{\infty}^{-1}
 +O(1)M^{-2}_{\infty}\\[1mm]
&\quad\, +O(1)M^{2}_{\infty}e^{-m_{0}M^{2}_{\infty}}+O(1)e^{-m_{0}M^{2}_{\infty}}.
\end{align*}
The case, $j=2$, can be handled similarly. Finally, by the Taylor formula, we have
\begin{eqnarray*}
&&\frac{\tilde{u}_{s}(s_0)+\lambda_{1}(s_0)\tilde{v}_{s}(s_0)}{\tilde{u}_{s}(s_0)+\lambda_{2}(s_0)\tilde{v}_{s}(s_0)}\\
&&=\frac{\frac{b_0}{1+b^{2}_0}+ (1-b^{2}_0)(1+b^{2}_0)^{-\frac{1}{2}}M_{\infty}^{-1}
+O(1)M^{-2}_{\infty}  +O(1)M^{2}_{\infty}e^{-m_{0}M^{2}_{\infty}}+O(1)e^{-m_{0}M^{2}_{\infty}}}
{\frac{b_0}{1+b^{2}_0}- (1-b^{2}_0)(1+b^{2}_0)^{-\frac{1}{2}}M_{\infty}^{-1}
+O(1)M^{-2}_{\infty} +O(1)M^{2}_{\infty}e^{-m_{0}M^{2}_{\infty}}+O(1)e^{-m_{0}M^{2}_{\infty}}}\\[5pt]
&&=1+2b^{-1}_{0}(1-b^{2}_0)\sqrt{1+b^{2}_0}\,M_{\infty}^{-1}+O(1)M^{-2}_{\infty}
+O(1)M^{2}_{\infty}e^{-m_{0}M^{2}_{\infty}} +O(1)e^{-m_{0}M^{2}_{\infty}}.
\end{eqnarray*}
This completes the proof.
\end{proof}

\smallskip
According to Lemma \ref{lem:2.17}, we can obtain the solvability of the above Riemann problem
near the strong shock as below.

\smallskip
\begin{proposition}\label{prop:4.2}
For $M_\infty$ sufficiently large, there exists a constant $\delta_{0}>0$ such that, for states $U_{b}=U_{\infty}$
and $U_{a}\in O_{\hat{\varepsilon}}(\Gamma(b_{0}, u_{\infty}))\cap O_{\delta_{0}}(\Theta(s_{0}))$,
the Riemann problem \eqref{eq:2.1} and \eqref{eq:4.1} admits a unique admissible solution
that contains a strong $1$-shock and a $2$-weak wave of the $2$-characteristic field -- either a $2$-shock or $2$-rarefaction wave.
\end{proposition}

\section{Construction of Approximate Solutions}\setcounter{equation}{0}
In this section, we construct global approximate solutions of the initial-boundary value problem
\eqref{eq:1.1}--\eqref{eq:1.6} under the assumptions of Theorem \ref{thm:1.1}.
We develop a modified Glimm scheme with the Riemann solutions of the homogeneous system \eqref{eq:2.1}
and the local self-similar solutions of problem \eqref{eq:3.1}--\eqref{eq:3.2} as building blocks
in order to incorporate the geometric source term.

To do this, denote $\Delta x$ and $\Delta y$ as mesh lengths in $x$ and $y$, respectively, and
$\Delta \sigma$ as a uniform grid size  for the self-similar variable $\sigma$.
The initial numerical grid sizes $\Delta x$ and $\Delta\sigma$
are suitably chosen so that the usual Courant-Friedrichs-Lewy condition holds:
\begin{eqnarray*}
\frac{\Delta y}{\Delta x}>2\,\underset{i=1,2}{\max}\big\{\underset{U}{\sup}|\lambda_{i}(U)|\big\}.
\end{eqnarray*}
We also choose a set of points $\{A_{k}\}_{k=0}$ with $A_{k}=(x_k, b_k)$, where
$x_k=x_{0}+k\Delta x$ and $b_k=b(x_k)$ for $k=0,1,\dots$.

Define
\begin{eqnarray*}
b_{\Delta }(x)=b_k+\frac{b_{k+1}-b_k}{\Delta x}(x-x_{k}) \qquad\,\, \mbox{for $x\in[x_{k},x_{k+1})$ and $k\geq0$}.
\end{eqnarray*}
Let
\begin{eqnarray*}
&&\Omega_{\Delta x, k}=\big\{(x,y)\,:\, x_k\leq x< x_{k+1}, \ y< b_{\Delta }(x)\big\},\,\,\,
\Omega_{\Delta x}=\big\{(x,y)\,:\, x>0, \ y< b_{\Delta }(x)\big\},\\[2mm]
&&\Gamma_{\Delta x, k}=\big\{(x,y)\,:\, x_{k}\leq x<x_{k+1}, \ y=b_{\Delta }(x)\big\},\,\,\,
\Gamma_{\Delta x}=\big\{(x,y)\,:\, x>0, \ y=b_{\Delta }(x)\big\}.
\end{eqnarray*}
Denote
\begin{eqnarray*}
&&\theta_{0}=\arctan b_{0}, \qquad\theta_{k}=\arctan(\frac{b_{k}-b_{k-1}}{\Delta x}) \,\,\,\,\mbox{for $k>0$},\\
&&\omega_{0}=\arctan(\frac{b(x_{0})-b(0)}{x_{0}}),\qquad \omega_{k}=\theta_{k+1}-\theta_{k} \,\,\,\,\mbox{for $k\geq0$},
\end{eqnarray*}
so that $\omega_k$ represents the change of angle at the turning point $A_k$ for each $k\ge 0$.

From  hypothesis $\mathbf{(H_{1})}$, when $x>x_{0}$,
the cone boundary is approximated by a set of line segments ${\Gamma_{\Delta x,k}}$
with $\Gamma_{\Delta x,k}=A_{k}A_{k+1}$ for $k\geq0$,  so that the slope of $\Gamma_{\Delta x,k}$
is negative and uniformly bounded.
Then we can extend $\Gamma_{\Delta x,k}$ so that the extension of $\Gamma_{\Delta x, k}$
and the $x$--axis intersect at point $(X^{*}_k,0)$ with
\begin{eqnarray}\label{eq:5.0}
X^{*}_k=x_{k-1}-\frac{b_{k-1}}{b_{k}-b_{k-1}}\Delta x;
\end{eqnarray}
point $(X^{*}_k,0)$ is called the center of the self-similar solution for each $k\ge 0$.
Moreover, by a direct computation, we have the following.

\smallskip
\begin{lemma}\label{lem:3.1}
For $k>0$,
\begin{eqnarray}
X^{*}_k-X^{*}_{k-1}=O(1)b_{k-1}(\tan\theta_{k}-\tan\theta_{k-1}),\label{eq:5.1}
\end{eqnarray}
where $O(1)$  depends only on $b_{0}$, independent of $k$.
\end{lemma}

\smallskip
\begin{proof}
By \eqref{eq:5.0}, we have
\begin{align*}
X^{*}_{k}-X^{*}_{k-1}&=x_{k-1}-\frac{b_{k-1}}{b_{k}-b_{k-1}}\Delta x
-\Big(x_{k-2}-\frac{b_{k-2}}{b_{k-1}-b_{k-2}}\Delta x\Big)\\[5pt]
&=\Big(1-\frac{b_{k-1}}{b_{k}-b_{k-1}}+\frac{b_{k-2}}{b_{k-1}-b_{k-2}}\Big)\Delta x\\[5pt]
&=\frac{b_{k-1}}{\tan\theta_{k}\tan\theta_{k-1}}(\tan\theta_{k}-\tan\theta_{k-1}),
\end{align*}
which leads to the desire result by assumption $(\mathbf{H}_{1})$.
\end{proof}

We now describe the construction of the difference scheme and corresponding approximate solutions.
In region $\{(x,y)\,:\, 0<x\leq x_0, y<b_0 x\}$,
the approximate solution is defined as the unperturbed conical flow with center at $(0,0)$.
For $x=x_0$, the grid points are the intersection points of $x=x_0$ with the self-similar rays centered at $(0,0)$:
\begin{eqnarray*}
y=(\tan \omega_{0}+h\Delta \sigma)x \qquad\mbox{for $h=0,-1,-2, \cdots$}.
\end{eqnarray*}

Choose an equi-distributed sequence $\vartheta=(\vartheta_0,\dots,\vartheta_k,\dots)\in\Pi^{\infty}_{k=0}(-1, 1)$.
Suppose that the approximate solution
$U_{\Delta x,\vartheta}(x,y)$ has been
defined for $x<x_k$,
and the grid points have been
defined for $x\leq x_k$ for $k\geq1$.
The approximate solution $U_{\Delta x, \vartheta}(x_l,y)$ is a piecewise smooth solution
of problem \eqref{eq:3.1}--\eqref{eq:3.2} on each vertical grid line $x=x_{l}+$ for $l<k$.
That is, at any continuous point $(x,y)$ of this approximate solution, it has the form:
\begin{eqnarray*}
U_{\Delta x,\vartheta}(x,y)=U_{\rm self}(\sigma(x,y)),
\end{eqnarray*}
where $\sigma(x,y)=\frac{y}{x-X^{*}}$, $U_{\rm self}(\sigma)$ is a self-similar solution of
system \eqref{eq:3.1}--\eqref{eq:3.2}, and $X^{*}=X^{*}(x,y)$ (called the center
of $U_{\rm self}$) is a piecewise constant and right-continuous function.
As part of the induction hypothesis,
we also assume that center $X^{*}$ of the constructed self-similar solution has been specified on
$\{x=x_l,\ y_{h-1}(l)<y<y_{h}(l)\}$ for $l< k,\ h=0,-1,\dots$, and $X^{*}\in \{X^{*}\}_{j\geq0}$
for $x< x_{k}$, where $y=y_{h}(l)$ is the grid points on $x=x_l$ and $y_0(l)=b(x_l)$.

Then we define the approximate solution $U_{\Delta x,\vartheta}(x,y)$
and the numerical grids inductively for regions
$\Omega_{\Delta x, k}$
for $k \geq1$.
The construction of the approximate
solution on $\Omega_{\Delta x}$ between $x_k\leq x< x_{k+1}$ is based on the following three cases:

\medskip
\subsection{Case\ 1: Away from the cone boundary in region
$\{x_k\leq x<x_{k+1}\}\cap\Omega_{\Delta x}$}
\label{5.1}
We construct the approximate solution $U_{\Delta x,\vartheta}(x,y)$
in the following four steps:

\smallskip
{\rm (i)} Define the approximate solution on any interval
$y_{h}(k)<y<y_{h+1}(k), h \leq -1$, on line $x=x_k$.
Let $U(x_{k}, y)$ be the solution of system \eqref{eq:3.1}--\eqref{eq:3.2}
with the following initial data given at the mesh point:
\begin{eqnarray}
U(x_k,a_{k,h})=U_{\Delta x, \vartheta}(x_k-,a_{k,h}), \label{eq:5.2}
\end{eqnarray}
where $a_{k,h}$ is a random choice point and can be represented
as $y_{h}(k)+\vartheta_k(y_{h+1}(k)-y_{h}(k))$.
This is the Cauchy problem \eqref{eq:3.1}--\eqref{eq:3.2} of the ordinary differential system,
whose solution is self-similar with variable $\sigma$
that is defined below.
It should be noted that the initial value above does not uniquely determine
the non-autonomous system \eqref{eq:3.3}, and the center of the self-similar solution
needs to be specified.
We specify center $X^{*}(x_{k},y)$ to be the center of the self-similar solution
$U_{\Delta x,\vartheta}(x_k-,a_{k,h})$
through the random choice method. That is,
\begin{eqnarray*}
X^{*}(x_{k},y)=X^{*}(x_{k}-, a_{k,h})\qquad\,\,\mbox{for $y_{h}(k)<y<y_{h+1}(k)$}.
\end{eqnarray*}
In other words, when the center on line $x=x_k-$ is defined,
the center for $\{x=x_k+,\, y_{h}(k)<x<y_{h+1}(k)\}$
is the same as the center for $(x_k-, a_{k,h})$.
Then this yields the self-similar variable:
\begin{eqnarray*}
\sigma=\frac{y}{x_{k}-X^{*}(x_{k},y)}.
\end{eqnarray*}

\smallskip
{\rm (ii)} The approximate solution on line $x=x_k$ defined above may
have discontinuities on the grid points $(x_k, y_h(k)), h=-1,-2,\dots$.
Therefore, we construct $U_{\Delta x,\vartheta}(x,y)$ in
$\{(x,y)\,:\, x_k < x< x_{k+1},\ y_{h}(k)<y<y_{h+1}(k)\}$ by solving a series of
Riemann problems of system \eqref{eq:2.1} with the initial data:
\begin{equation}\label{eq:5.3}
U_{\Delta x,\vartheta}(x,y)
=\begin{cases}
U_{\Delta x,\vartheta}(y_h(k)+,x_k+)\quad\,
 &\mbox{for $y>y_h(k)$}, \\[1mm]
U_{\Delta x,\vartheta}(y_h(k)-,x_k+)\quad\,
 &\mbox{for $y<y_h(k)$}.
\end{cases}
\end{equation}
That is,
if $U_{\Delta x, \vartheta}(y_h(k)\pm,x_k+)
\in O_{\hat{\varepsilon}}\big(\Gamma(b_{0}, u_{\infty})\big)\cap O_{\delta_{0}}(\Theta(s_{0}))$,
then it follows from Proposition 4.1
that this Riemann problem is
solvable, and the solution is a function of
$\xi = \frac{y-y_h(k)}{x-x_k}$ and consists of shocks and/or
rarefaction waves.

\smallskip
{\rm (iii)} To include the information of the geometric lower order term, we make
a so-called {\it self-similar modification} for the approximate solution
constructed above.

Let
\begin{eqnarray*}
\sigma=\sigma(x,y)=\frac{y}{x-X^{*}(x,y)}.
\end{eqnarray*}
From the above steps, $\sigma(x,y)$ is well defined and satisfies
\begin{eqnarray*}
\lambda_{1}(U_{\Delta x,\vartheta}(x,y))<\sigma(x,y)<\lambda_{2}(U_{\Delta x,\vartheta}(x,y)).
\end{eqnarray*}
Denote
\begin{eqnarray*}
\sigma_{h-\frac{1}{2}}(k)=\sigma(x_{k}-, \frac{y_{h-1}+y_{h}}{2}).
\end{eqnarray*}
Then, along ray $\frac{y-y_h(k)}{x-x_k} = \xi$
for each $\xi$, the approximate solution $U_{\Delta x,\vartheta}(x,y)$
in $\{x_k<x<x_{k+1}, \sigma_{h-\frac{1}{2}}(k)<\sigma(x,y)<\sigma_{h+\frac{1}{2}}(k)\}$
is defined as the solution of equation \eqref{eq:3.3} with the initial data $U(\xi)$ at
$x=x_k+0$,
where
\begin{align*}
&\sigma=\frac{y}{x-X^{*}_{k,h}},\quad X_{k,h}^*=X(x_k-, a_{k,h}) &&\mbox{for $\xi>\xi_{k,h}$},\qquad\\[1mm]
&\sigma=\frac{y}{x-X^{*}_{k,h-1}},\quad X_{k, h-1}^*=X(x_k-, a_{k, h-1}) &&\mbox{for $\xi<\xi_{k,h}$},\qquad
\end{align*}
and $U(\xi)$ is the solution of the Riemann problem given above.
For this, the center keeps invariant along the rays.

\smallskip
{\rm (iv)} Finally, as in \cite{lien-liu}, the grid lines between $x=x_k$ and
$x=x_{k+1}$ are defined by the rays going through every grid point on
$x=x_k$, and the  numerical grid points on $x=x_{k+1}$
are defined to be the interaction points between the corresponding grid lines and $x=x_{k+1}$.
The new centers on $x=x_{k+1}$ inherit those centers on $x=x_{k}+$ through the random choice.
Then we obtain the approximate solution in region $\{x_k\leq x<x_{k+1}\}\cap\Omega_{\Delta x}$
and extend it to the whole domain $\Omega_{\Delta x}$ by induction.

\medskip
\subsection{ Case\ 2: On the cone boundary $\{x_k\leq x<x_{k+1}\}\cap\Gamma_{\Delta x}$}
In general, a $1$-wave is produced and emerges into the domain owing to the
turning angle of the cone boundary.
It can be a shock or rarefaction wave depending on the change of the boundary angle
toward (or away from) the flow.
Meanwhile, the $2$-wave issuing from
$(x_k,y_{-1}(k))$ is reflected on the boundary, and a $1$-wave is formed.
We define the approximate solution in the following three steps:

\smallskip
{\rm (i)} As before, we first solve problem \eqref{eq:3.1}--\eqref{eq:3.2} on
$\{x=x_k,\ y_{-1}(k)<y<y_0(k)\}$,
where the center is chosen the same as the center of the initial data.

\smallskip
{\rm (ii)} Next, we solve the initial-boundary problem of system
\eqref{eq:2.1} with initial data:
\begin{eqnarray*}
U(x_{k}, y) =U_{\Delta x, \vartheta}(x_k-, y_0(k)-)\qquad\,\, \mbox{for $y_{-1}(k)<y<y_0(k)$},
\end{eqnarray*}
and with the boundary condition on
$\Gamma_{\Delta x,k}$:
\begin{eqnarray*}
v= \sigma_0(k) u,
\end{eqnarray*}
where
\begin{eqnarray*}
 \sigma_0(k)=\frac{y_0(k+1)-y_0(k)}{x_{k+1}-x_{k}}.
\end{eqnarray*}
The solution of this problem contains only a $1$-wave.
Between the lower edge of the
$1$-wave and the cone boundary, the center is chosen as the
intersection point of the ray through $(x_k,y_0(k))$ with slope
$\sigma_0(k)$ and the $x$--axis, {\it i.e.},
$(x_k-\frac{y_0(k)}{\sigma_0(k)},0)$.
We point out that the center changes of the self-similar solutions
in the whole domain between the cone boundary and the leading
shock-front are due to the changes of the cone boundary slopes.
As to the centers below the lower edge of the $1$-wave,
it has been defined for
{Case\ 1} in \S 5.1.

\smallskip
{\rm (iii)} We also make a {\it self-similar modification} for this solution as in {\bf Case\ 1}.
Then the approximate solution is extended to
$$
\big\{(x,y)\,:\, x_{k}\leq x<x_{k+1},\,y_{0}(k)+\frac{1}{2}(y_{0}-y_{-1}(k))<y<y_{0}(k)+\sigma_0(k)(x-x_{k})\big\}
$$
as before with center $x_k-\frac{y_0(k)}{\sigma_0(k)}$.

\medskip
\subsection{Case\ 3: Near the leading conical shock-front next to the uniform upstream flow traced continuously}
Suppose that the approximate solution has been constructed for $x<x_{k}$.
Let $(x, y_{\r{s}}(x))$ be the locus of the
front of the strong leading $1$-shock.
Suppose that $y_{h_{\rs}-1}(k)<y_{\rs}(x)<y_{h_{\rs}+1}(k)$.
As in \cite{lien-liu, wang-zhang},
interval $y_{h_{\rs}-1}(k)<y_{\rs}(x)<y_{h_{\rs}+1}(k)$ is called the front region at $x=x_{k}$.
Inside the front region, we first
solve the self-similar solution of system $\eqref{eq:3.3}$ with the initial data:
\begin{eqnarray*}
U(x_k,a_{k,h_{\rs}})=U_{\Delta x, \vartheta}(x_k-,a_{k,h_{\rs}})
\end{eqnarray*}
and self-similar variable:
\begin{eqnarray*}
\sigma=\frac{y}{x_{k}-X^{*}(x_{k},a_{k,h_{\rs}})}.
\end{eqnarray*}
The solution is denoted as $U_{\rm self}(x_{k}, y)$ that satisfies
\begin{eqnarray*}
\lambda_{1}(U_{\rm self}(x_{k}, y))<\sigma(x_{k}, y)<\lambda_{2}(U_{\rm self}(x_{k}, y)).
\end{eqnarray*}

\smallskip
Next, we solve the Riemann problem of system \eqref{eq:2.2} with the
initial data:
\begin{equation*}
U|_{\{x=x_k\}}
=\begin{cases}
U_{\rm self}(x_{k}, y) &\qquad\mbox{for $y_{\rs}(k) <y<y_{h_{\rs}+1}(k)$},\\[1.5mm]
U_{\infty}  &\qquad \mbox{for $y<y_{\rs}(k)$}.
\end{cases}
\end{equation*}
The solution, $U(x,y)$, contains a weak $2$-wave and a strong $1$-shock
denoted by $y=\chi_{\Delta x, \vartheta}(x)$ with
speed $s_{k+1}$.
Solve equation \eqref{eq:3.3} again on interval $y_{\rs}(k) <y<y_{h_{\rs}+1}(k)$
with initial value $U(x_{k}, y_{\rs}(k)+)= U_{+}$ and self-similar variable
$\sigma=\frac{y}{x_{k}-X^{*}(x_{k},a_{k,h_{\rs}})}$.
Denote its solution by $U_{-}(\sigma)$.
Now we can define the approximate solution in the front region:
\begin{equation*}
U_{\Delta x, \vartheta}(x,y)
=\begin{cases}
 U_{-}(\sigma)\ \ &\mbox{for $y_{\rs}(k) <y<y_{h_{\rs}+1}(k)$},\\[1.5mm]
U_{\infty}\ \ \quad &\mbox{for $y<y_{\rs}(k)$}.
\end{cases}
\end{equation*}
The discontinuities at point $(x_{k}, y_{h_{\rs}+1})$ are resolved in the same way as in {\bf Case\ 1}.
Moreover, we must also specify the center of the self-similar variable near the leading shock-front as
\begin{eqnarray*}
X^{*}(x,y)=X^{*}_{k,h_{\rs}} \qquad\,\,
\mbox{for $x_{k}<x<x_{k+1}$ and $s_{k+1}<\frac{y-y_{h_{\rs}+1}(k)}{x-x_k}<\xi_{k,h_{\rs}+1}$}.
\end{eqnarray*}
In this way, we complete the construction of the difference scheme and corresponding approximate
solutions $U_{\Delta x, \vartheta}(x,y)$ globally
in $\Omega_{\Delta x}\cup \Gamma_{\Delta x}$.

\smallskip
\section{Local Interaction Estimates}\setcounter{equation}{0}

In this section, we establish some uniform estimates of the approximate solutions
constructed in \S 5. For these,
the following formulas are used:

\smallskip
{\rm (i)}\, If $f\in C^{1}(\mathbb{R})$, then
\begin{equation}
f(x)-f(0)=x\int_{0}^{1}f_{x}(\mu x)\,\dd\mu\qquad \mbox{for any $x\in \mathbb{R}$}.
\label{eq:6.1}
\end{equation}

{\rm (ii)}\, If $f\in C^{2}(\mathbb{R}^{2})$, then, for any $(x,y)\in \mathbb{R}^{2}$,
\begin{equation}
f(x,y)-f(x,0)-f(0,y)+f(0,0)=xy\int_{0}^{1}\int_{0}^{1}f_{xy}(\mu x,\tau y)\,\dd\mu \dd \tau.
\label{eq:6.2}
\end{equation}

\begin{lemma}\label{lem:6.2}
Let $\sigma_{k}=\frac{y_{k}}{x_{k}-X}$ and $\bar{\sigma}_{k}=\frac{y_{k}}{x_{k}-\bar{X}}$ for $k=1,2$,
with $X, \bar{X}>0$ as their centers.
Then
\begin{eqnarray}\label{eq:6.0}
\Delta\bar{\sigma}=\Delta\sigma+O(1)|X-\bar{X}||\Delta\sigma|+O(1)|X-\bar{X}||x_2^{-1}-x_1^{-1}|,
\end{eqnarray}
where $\Delta\sigma=\sigma_{2}-\sigma_{1}$, $\Delta\bar{\sigma}=\bar{\sigma}_{2}-\bar{\sigma}_{1}$, and
$O(1)$ is independent of $X,\ \bar{X}$, and $\Delta\sigma$.
\end{lemma}

\smallskip
\begin{proof}
Since
$$
\bar{\sigma}_j-\sigma_j=\frac{y_{j}}{x_{j}-\bar{X}}-\frac{y_{j}}{x_{j}-X}
=\frac{y_j}{(x_j-X)(x_j-\bar{X})}(X-\bar{X})\qquad\mbox{for $j=1,2$},
$$
then
\begin{align*}
\Delta\bar{\sigma}-\Delta\sigma &=\Big(\frac{y_{2}}{(x_{2}-\bar{X})(x_{2}-X)}-\frac{y_{1}}
{(x_{1}-\bar{X})(x_{1}-X)}\Big)\big(\bar{X}-X\big)\\[5pt]
&=\frac{1}{x_2-\bar{X}}\Big(\frac{y_{2}}{x_{2}-X}-\frac{y_{1}}{x_{1}-X}\Big)\big(\bar{X}-X\big)\\[5pt]
&\ \  \ +\frac{y_1}{x_1-X}\frac{x_1x_2}{(x_1-\bar{X})(x_2-\bar{X})}
\big(x_2^{-1}-x_1^{-1}\big)\big(\bar{X}-X\big)\\[5pt]
&=O(1)|X-\bar{X}||\Delta\sigma|+O(1)|X-\bar{X}||x_2^{-1}-x_1^{-1}|,
\end{align*}
where $O(1)$ is independent of $X,\ \bar{X}$, and $\Delta\sigma$.
\end{proof}

From now on, we use Greek letters $\alpha$, $\beta$, $\gamma$, and $\delta$ to represent the elementary waves
in the approximate solution, and $\alpha_{i}$, $\beta_{i}$, $\gamma_{i}$, and $\delta_{i}$, $i=1,2$, denote the
corresponding $i$-th components of the respective waves.
To avoid confusion, when $U(x,y)=(u,v)^\top(x,y)$ is the solution of problem
\eqref{eq:3.1}--\eqref{eq:3.2},
we often use $\varpi(\sigma; O)$ to stand for the states,
where $\sigma=\frac{y}{x-X^{*}_{k}}$ is the self-similar variable with
$O=(X^{*}_{k},0)$ as the corresponding center.
In addition, we use $\Psi=\Psi(\sigma-\sigma_{0},\sigma_{0};\varpi(\sigma_{0}; O))$
to be the solution of system \eqref{eq:3.3} with initial data:
\begin{eqnarray*}
&\left.\Psi\right|_{\sigma=\sigma_{0}}=\varpi(\sigma_{0}; O),
\end{eqnarray*}
where $ \varpi(\sigma_{0}; O)\in O_{\hat{\varepsilon}}(\Gamma(b_{0},u_{\infty}))$.

As in \cite{chen-zhang-zhu,smoller, wang-zhang, zhang},
a curve $\textsl{I}$ is called a mesh curve if
$\textsl{I}$ is a space-like curve that consists of the line segments
joining the random points one by one in turn.
Then $\textsl{I}$ divides region $\Omega_{\Delta x}$ into two parts:
$\textsl{I}^{-}$ and $\textsl{I}^{+}$, where $\textsl{I}^{-}$
denotes the part containing line $x=x_{0}$.
For any two mesh curves $\textsl{I}$ and $\textsl{J}$,
we use $\textsl{J}>\textsl{I}$ to represent that every mesh point of
curve $\textsl{J}$ is either on $\textsl{I}$ or contained in $\textsl{I}^{+}$.
In particular, we call $\textsl{J}$ an immediate successor to $\textsl{I}$,
provided that $\textsl{J}>\textsl{I}$, and every mesh point of $\textsl{J}$ except the one
is on $\textsl{I}$.

Let
\begin{eqnarray*}
 \Omega^{+}_{\Delta x, j}=\Omega_{\Delta x, j}\cap \{y> \chi_{\Delta x, \vartheta}(x)\},\qquad
 \Omega^{-}_{\Delta x, j}=\Omega_{\Delta x, j}\cap \{y< \chi_{\Delta x, \vartheta}(x)\},
\end{eqnarray*}
where curve $y=\chi_{\Delta x, \vartheta}(x)$ is the approximate strong leading shock-front
with speed $s_{\Delta x, \vartheta}(x)$.

\smallskip
We make the following inductive hypotheses:

\medskip
\noindent
{\rm $(\mathbf{P}1)_{(k)}$}\,\, The approximate solution $U_{\Delta x, \vartheta}(x,y)$
has been defined in
\begin{eqnarray*}
\{0\leq x\leq k\Delta x\}\cap \Omega_{\Delta x}.
\end{eqnarray*}

\medskip
\noindent
{\rm $(\mathbf{P}2)_{(k)}$}\,\,
For any $(x,y)\in  \Omega^{+}_{\Delta x, j}$,
\begin{eqnarray*}
U_{\Delta x, \vartheta}
\in O_{\hat{\varepsilon}}(\Gamma(b_{0}, u_{\infty}))
\cap O_{\delta_{0}}(\Theta(s_{0})),
\end{eqnarray*}
$\quad\,\,\,\,$ and, for any $(x,y)\in \Omega^{-}_{\Delta x, j}$ with $0\leq j\leq k$,
\begin{eqnarray*}
U_{\Delta x, \vartheta}=U_{\infty}.
\end{eqnarray*}

\noindent
{\rm $(\mathbf{P}3)_{(k)}$}\,\,
For any weak wave $\alpha$,
\begin{eqnarray*}
\lambda_{1}(U_{\Delta x, \vartheta}(x_{\alpha}-,\cdot))
<\sigma(x_{\alpha}-,\cdot)<\lambda_{2}(U_{\Delta x, \vartheta}(x_{\alpha}-,\cdot)),
\end{eqnarray*}
where $(x_{\alpha}, y_{\alpha})$ denotes the point with $x_{\alpha}\in \{x_{j}\,:\, 0\leq j\leq k\}$
from where the
weak wave $\alpha$ issues,
while $\sigma(x_{\alpha}, y_{\alpha})$ represents the corresponding
self-similar
variable,
and  $U_{\Delta x, \vartheta}(x,y)$ stands for its approximate solution.

\medskip
Then we prove in the following subsections that, under suitable conditions,
$U_{\Delta x, \vartheta}$ can be defined in  $\{0\leq x\leq (k+1)\Delta x\}\cap \Omega_{\Delta x}$
and satisfies $(\mathbf{P}1)_{(k+1)}$--$(\mathbf{P}3)_{(k+1)}$.
As in \cite{glimm} (also see \cite{chen-zhang-zhu,dafermos2016,smoller}), we carry out this step by considering any pairs
of the mesh curves $\textsl{I}$ and $\textsl{J}$ with $\textsl{J}$ as an immediate successor to
$\textsl{I}$, where $\textsl{I}$ and $\textsl{J}$ are in
$\{(k-1)\Delta x\leq x\leq (k+1)\Delta x\}\cap\Omega_{\Delta x}$.

Now let $\Lambda$ be the diamond between $\textsl{I}$ and $\textsl{J}$.
Suppose that
\begin{equation*}
\quad U_{\Delta x, \vartheta}(x,y)
\in O_{\hat{\varepsilon}}(\Gamma(b_{0}, u_{\infty}))
\cap O_{\delta_{0}}(\Theta(s_{0}))\quad \mbox{for any
$(x,y)\in \textsl{I}\cap (\Omega^{+}_{\Delta x, k}\cup\Omega_{\Delta x, k+1})$},
\end{equation*}
and
\begin{eqnarray*}
 \lambda_{1}(U_{\Delta x, \vartheta}(x_{\alpha}-,\cdot))
<\sigma(x_{\alpha}-,\cdot)<\lambda_{2}(U_{\Delta x, \vartheta}(x_{\alpha}-,\cdot))
\,\,\,\, \mbox{for any weak wave $\alpha$ crossing $\textsl{I}$}.
\end{eqnarray*}

\subsection{$\Lambda$ is between $y=b_{\Delta}(x)$
and $y=\chi_{\Delta x, \vartheta}(x)$}
In this section, we consider the interactions involving only weak waves.
By the construction of the approximate solution, the waves entering $\Lambda$
are denoted by $\alpha=(\alpha_1, 0)$ and $\beta=(\beta_1, \beta_2)$
that issue from $(x_{k-1},y_{h}(k-1))$ and $(x_{k-1},y_{h-1}(k-1))$, respectively.
Let $\delta=(\delta_1, \delta_2)$ be the set of waves issuing
from $(x_{k},y_{h-1}(k))$ (see Fig. \ref{fig6.1}).

\medskip
\begin{figure}[ht]
\begin{center}\vspace{-40pt}
\setlength{\unitlength}{0.7mm}
\begin{picture}(50,105)(-15,-60)
\put(-30,-40){\line(0,1){70}}\put(10,-43){\line(0,1){70}}\put(50,-44){\line(0,1){70}}
\put(-35,24){\line(6,-1){96}}
\put(-35,-13){\line(6,-1){96}}
\put(-30,-14){\line(5,1){35}}\put(-30,-14){\line(3,-2){30}}
\put(-30,23){\line(5,-4){30}}
\put(10,-20.5){\line(2,-1){30}}\put(10,-20.5){\line(6,1){37}}
\thicklines
\put(-30,-20){\line(3,-1){40}}\put(-30,-20){\line(4,3){40}}
\put(10,10){\line(2,-1){40}}\put(10,-33.5){\line(5,3){40}}
\put(-5,8){$\alpha$}
\put(2,-13){$\beta$}
\put(43,-36){$\delta$}
\put(-32,-43){$_{x_{k-1}}$}\put(8.5,-45){$_{x_{k}}$}
 \put(46,-45){$_{x_{k+1}}$}
\put(-49,-17){$_{y_{h-1}(k-1)}$}
\put(-45,21){$_{y_{h}(k-1)}$}
\put(-4,-24){$_{y_{h-1}({k})}$}
\put(0,14){$_{y_{h}({k})}$}
\end{picture}
\end{center}\vspace{-30pt}
\caption{Interactions involving only weak waves}\label{fig6.1}
\end{figure}
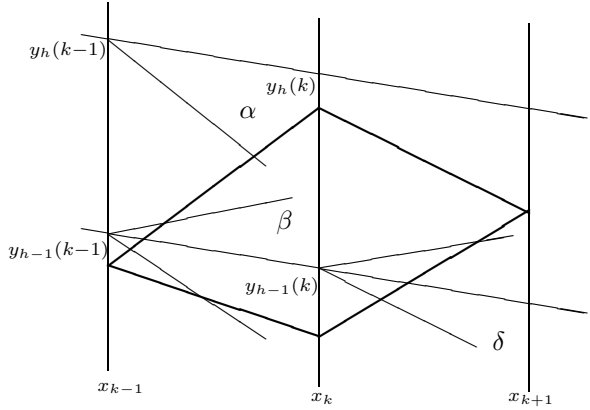

We now consider the case:
\begin{eqnarray}
&&\varpi(\sigma_{1};O_{1})=\Phi(\alpha;\varpi(\bar{\sigma}_{1};O_{2})),\label{eq:6.3}\\[5pt]
&&\varpi(\bar{\sigma}_{2};O_{2})=\Phi(\beta;\varpi(\tilde{\sigma}_{3};O_{3})),\label{eq:6.4}\\[5pt]
&&\varpi(\sigma_{2};O_{1})=\Phi(\delta;\varpi(\tilde{\sigma}_{3};O_{3}))\label{eq:6.5},
\end{eqnarray}
where $O_{1}=(X,0),\ O_{2}=(\bar{X},0)$, and $O_{3}=(\tilde{X},0)$.

For notational convenience, we denote $\tilde{x}_{0}:=|\bar{X}-X|, \tilde{x}_{1}:=|\tilde{X}-\bar{X}|$,
and  $U_{b}:=\varpi(\tilde{\sigma}_{3};O_{3})$ from now on.

\smallskip
\begin{lemma}\label{lem:6.3}
For the waves described above,
\begin{eqnarray}\label{eq:6.6}
\delta=\alpha+\beta+O(1)Q(\Lambda),
\end{eqnarray}
where $Q(\Lambda)=Q^{0}(\Lambda)+Q^{1}(\Lambda)+Q^{c}(\Lambda)$ with
\begin{eqnarray*}
&&Q^{0}(\Lambda)=\sum\big\{|\alpha_i||\beta_j|\,:\,\alpha_i\ and\ \beta_j\ approach\big\},\\[1mm]
&&Q^{1}(\Lambda)=|\alpha||\Delta\sigma|,\\[1mm]
&&Q^{c}(\Lambda)=\big(|\Delta\sigma|+|x_k^{-1}-x_{k-1}^{-1}|\big)\tilde{x}_{0},
\end{eqnarray*}
and $\Delta\sigma=\sigma_{1}-\sigma_{2}$,
and $O(1)$ depends continuously on $M_{\infty}$ but independent of $\alpha,\beta,\Delta\sigma$, and $x_{0}$.
\end{lemma}

\smallskip
\begin{proof}
We combine \eqref{eq:6.3}--\eqref{eq:6.5} to obtain
\begin{eqnarray}
\Psi(\Delta\sigma,\sigma_2;\Phi(\delta;U_{b}))=
\Phi(\alpha;\Psi(\Delta\bar{\sigma},\bar{\sigma}_{2};\Phi(\beta; U_{b}))).\label{eq:6.7}
\end{eqnarray}
Lemma \ref{lem:2.14} yields
\begin{align*}
&\lim_{M_{\infty}\rightarrow \infty}
\det\Big( \frac{\partial\Phi(\delta; U_{b})}
{\partial(\delta_1,\delta_2)}\Big|_{\{\alpha=\beta=\Delta\sigma=\Delta\bar{\sigma}=x_{0}=0\}}\Big)\frac{1}{M_{\infty}}\\[1mm]
&\,\,=\lim_{M_{\infty}\rightarrow\infty}
\frac{\det(r_1(U_{b}),r_2(U_{b}))}{M_{\infty}}
=\frac{2 c^2}{(1+b^{2}_{0})^{\frac{5}{2}}}\neq 0.
\end{align*}
Then, by the implicit function theorem, system
\eqref{eq:4.7} has a unique $C^{2}$--solution:
$$
\delta=\delta(\alpha,\beta, \Delta\sigma,\Delta\bar{\sigma},x_{0}; U_{b})
$$
in a neighborhood of
$(\alpha,\beta,\Delta\sigma,\Delta\bar{\sigma},x_{0};U_{b})=(0,0,0,0,0; U_{\infty})$.

Let $\delta'=\delta(\alpha,\beta, \Delta\sigma,\Delta\sigma, x_{0}; U_{b})$.
By \eqref{eq:6.1},
we have
\begin{eqnarray}
\delta=\delta'
+ K'|\Delta\bar{\sigma}-\Delta\sigma|,\label{eq:6.8}
\end{eqnarray}
where $\displaystyle K'=\int\frac{\partial\delta}{\partial(\Delta\bar{\sigma})}\,\dd\xi$, and $\delta'$ solves the equation:
\begin{eqnarray*}
\Psi(\Delta\sigma,\sigma_2;\Phi(\delta'; U_{b}))=
\Phi(\alpha;\Psi(\Delta \sigma,\bar{\sigma}_{2};\Phi(\beta; U_{b}))).
\end{eqnarray*}
Moreover, by Lemma \ref{lem:2.14} again,
\begin{align*}
&\lim_{M_{\infty}\rightarrow \infty}\Big(M_{\infty}
\left.\frac{\partial\delta}{\partial(\Delta\bar{\sigma})}\right|_{\{\alpha=\beta=\Delta\sigma=\Delta\bar{\sigma}=x_{0}=0\}}\Big)\\[1mm]
&\,\,=\lim_{M_{\infty}\rightarrow \infty}\frac{c^2M_{\infty}}{\det\big(r_{1}(U_{b}),r_{2}(U_{b})\big)}
\left(\begin{array}{c}
\frac{\dd u_b}{\dd\tilde{\sigma}_3}+\lambda_{2}(U_{b})\frac{\dd v_b}{\dd\tilde{\sigma}_3}\\[1mm]
-\frac{\dd u_b}{\dd\tilde{\sigma}_3}-\lambda_{1}(U_{b})\frac{\dd v_b}{\dd\tilde{\sigma}_3}
\end{array}\right)\\[1mm]
&\,\,=-\frac{c\,(1+b^{2}_{0})^{2}}{2}(1,1)^{\top}.
\end{align*}
By Lemma \ref{lem:4.2}, we have
\begin{eqnarray*}
\delta'=\alpha+\beta+O(1)\big(Q^{0}(\Lambda)+Q^{1}(\Lambda)\big).
\end{eqnarray*}
Substituting this formula into \eqref{eq:6.8} and combining then with  Lemma \ref{lem:6.2},
we conclude \eqref{eq:6.6}.
\end{proof}

\subsection{$\Lambda$ covers the part of $y=b_{\Delta }(x)$
but none of $y=\chi_{\Delta x, \vartheta}(x)$}
We now consider the wave interactions near the approximate boundary.
Suppose that $\Lambda$ is the diamond centered at $(x_{k},y_{0}(k))$.
We denote the two waves entering $\Lambda$ by $\alpha_1$ and $\beta_2$ that issue from the
grid points $(x_{k-1},y_{0}(k-1))$ and $(x_{k-1},y_{-1}(k-1))$, respectively.
Let $\delta_1$ be the $1$-wave issuing from the grid point $(x_{k},y_{0}(k))$
with $U_2$ as its above state (see Fig. \ref{fig6.2}).
Suppose
that the center below the weak waves $\alpha_{1}$ and $\delta_1$ is $O_{1}=(X,0)$,
between $\alpha_{1}$ and the boundary is $O_{2}=(\bar{X},0)$,
and above $\delta_1$ is  $O_{3}=(\hat{X},0)$.

\smallskip
\begin{figure}[ht]
\begin{center}\vspace{-40pt}
\setlength{\unitlength}{0.7mm}
\begin{picture}(50,105)(-15,-55)
\put(-30,-25){\line(0,1){60}}\put(10,-26){\line(0,1){60}}\put(50,-28){\line(0,1){60}}
\put(-35,-10){\line(6,-1){96}}
\put(-30,-11){\line(5,2){32}}\put(-30,-11){\line(5,1){35}}
\put(-30,26.5){\line(5,-2){32}}
\put(10,20){\line(5,-4){30}}\put(10,20){\line(2,-1){32}}
\thicklines
\put(-36.2,27.7){\line(6,-1){46}}
\put(10,20){\line(6,1){46}}
\thicklines
\put(-30,12){\line(5,-3){40}}\put(-30,12){\line(5,2){40}}
\put(10,28){\line(3,-1){40}}\put(10,-12){\line(3,2){40}}
\put(4,11){$\alpha_1$}
\put(4,0){$\beta_2$}
\put(44,-2){$\delta_1$}
\put(-15,26){$U_1$}
\put(32,13){$U_2$}
\put(-32,-28){$_{x_{k-1}}$}\put(8,-30){$_{x_{k}}$} \put(46,-30){$_{x_{k+1}}$}
\put(-47,-15){$_{y_{-1}(k-1)}$}\put(-44,24){$_{y_{0}(k-1)}$}
\put(11,-23){$_{y_{-1}({k})}$}\put(1,18){$_{y_{0}({k})}$}
\end{picture}
\end{center}\vspace{-30pt}
\caption{Reflection at the boundary}\label{fig6.2}
\end{figure}
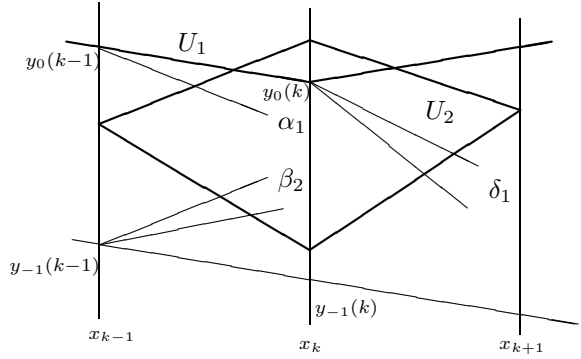

Denote
\begin{eqnarray*}
&&\sigma_{0}=\frac{y_{0}(k)}{x_{k}-X},\quad\, \sigma_{1}=\frac{y_{-1}(k-1)}{x_{k-1}-X},\quad\,  \sigma_{2}=\frac{y_{0}(k-1)}{x_{k-1}-X},\\[1mm]
&& \bar{\sigma}_{0}=\frac{y_{0}(k-1)}{x_{k-1}-\bar{X}},\quad\, \hat{\sigma}_{0}=\frac{y_{0}(k)}{x_{k}-\hat{X}},
\end{eqnarray*}
and let
\begin{eqnarray*}
&&U_{1}=\varpi(\bar{\sigma}_{0}; O_{2}),\quad\, U_{2}=\varpi(\hat{\sigma}_{0}; O_{3}),\quad\, U_{b}=\varpi(\sigma_{1}; O_{1}),\\[1mm]
&&\Delta\sigma=\sigma_{0}-\sigma_{1},\quad\, \Delta\tilde{\sigma}=\sigma_{2}-\sigma_{1},\quad\, \tilde{x}_{0}=|X-\bar{X}|.
\end{eqnarray*}
Then
\begin{eqnarray*}
&&U_1=\Phi_1(\alpha_1;\varpi(\sigma_2;O_1)),\quad \,
\varpi(\sigma_2;O_1)=\Psi(\Delta\tilde{\sigma},\sigma_1; \Phi_2(\beta_2; U_{b})),\\
&&U_2=\Phi_1(\delta_1;\varpi(\sigma_0;O_1)),\quad\,\,\, \varpi(\sigma_0;O_1)=\Psi(\Delta\sigma, \sigma_{1}; U_{b}).
\end{eqnarray*}
By the construction of the approximate solution, we have
\begin{eqnarray}
\Phi_1(\delta_1;\Psi(\Delta\sigma,\sigma_1; U_{b}))\cdot\mathbf{n}_{k}
=\Phi_1(\alpha_1;\Psi(\Delta\tilde{\sigma},\sigma_1;
\Phi_2(\beta_2; U_{b})))\cdot\mathbf{n}_{k-1}=0,\label{eq:6.11}
\end{eqnarray}
where $\mathbf{n}_{k}=(-\sin\theta_k,\cos\theta_k)$ for each fixed $k\geq0$ is the outer normal vector of the boundary.

To solve equation \eqref{eq:6.11}, we first have the following.

\smallskip
\begin{lemma}\label{lem:6.4}
$\displaystyle{\lim_{M_{\infty}\rightarrow \infty}\frac{r_1(U_b)\cdot\mathbf{n}_{0}}{M_{\infty}}
=c\,(1+b^{2}_{0})^{-\frac{3}{2}}}.
$
\end{lemma}

\smallskip
This can be seen via direct computation by using Lemma \ref{lem:2.11}:
\begin{eqnarray*}
\lim_{M_{\infty}\rightarrow \infty}\frac{r_1(U_b)\cdot\mathbf{n}_{0}}{M_{\infty}}
&=&\lim_{M_{\infty}\rightarrow \infty}\frac{e_1(U_b)}{M_{\infty}}\,
\lim_{M_{\infty}\rightarrow \infty}\big(\lambda_1(U_b)\sin\theta_0+\cos\theta_0\big)\\
&=&c\,(1+b^{2}_{0})^{-\frac{3}{2}}<\infty.
\end{eqnarray*}

\medskip
Then we have the following lemma for the existence and estimate of $\delta_{1}$.

\smallskip
\begin{lemma}\label{lem:6.5}
Equation \eqref{eq:6.11} has a unique solution
$$
\delta_1=\delta_1(\alpha_1,\beta_2, \Delta\sigma,\Delta\tilde{\sigma},
\omega_{k}; U_{\infty})\in C^2
$$
in a neighborhood of
$\alpha_1=\beta_2=\Delta\sigma=\Delta\tilde{\sigma}=\omega_{k}=0$
and $U_{b}=U_{\infty}$
such that
\begin{eqnarray}
\delta_1=\alpha_1+K_{\rm r}\beta_2+K_{\rm b}\omega_{k}+O(1)|\beta_{2}||\Delta\sigma|
+O(1)|x_{k}^{-1}-x_{k-1}^{-1}|\tilde{x}_{0} \label{eq:6.13}
\end{eqnarray}
with
\begin{eqnarray}
&&\sup_{1<M_\infty<\infty} |K_{\rm b}| < \infty,\label{eq:6.14}\\
&&\left. K_{\rm r}\right|_{\{\alpha_1=\beta_2=\omega_{k}=\Delta\sigma=\tilde{x}_{0}=0,\
\theta_{k}=\theta_{0}\}}=\frac{\cos^{2}(\theta_0+\theta_{\rm ma})}{\cos^{2}(\theta_0-\theta_{\rm ma})},\label{eq:6.15}
\end{eqnarray}
which implies that
\begin{eqnarray}\label{eq:6.16}
K_{\rm r}=1- 4 b_0\sqrt{1+b^{2}_0}\,M_\infty^{-1}+O(1)M^{-2}_{\infty}
+O(1)e^{-m_{0}M^{2}_{\infty}},
\end{eqnarray}
where the bound of $O(1)$ depends continuously on $M_{\infty}$.
\end{lemma}

\smallskip
\begin{proof} We divide the proof into four steps.

\smallskip
1. A direct computation gives
\begin{eqnarray*}
\frac{\partial}{\partial\delta_1}\Phi_1(\delta_1;
\Psi(\Delta\sigma,\sigma_1;U_{b}))\cdot\mathbf{n}_{k}
\big|_{\{\alpha_1=\beta_2=\omega_{k}=\Delta\sigma=\Delta\tilde{\sigma}=0,\,\theta_{k}=\theta_{0}\}}
=r_1(U_{b})\cdot\mathbf{n}_{0}.
\end{eqnarray*}
Then, by Lemma \ref{lem:6.4} and the implicit function theorem,
we can find a unique $C^{2}$--solution:
\begin{eqnarray*}
\delta_1=\delta_1(\alpha_1,\beta_2,\omega_{k},\Delta\sigma,\Delta\tilde{\sigma},\sigma_1; U_{b})
\end{eqnarray*}
near $(\alpha_1,\beta_2,\Delta\sigma,\Delta\tilde{\sigma}; U_{b})=(0,0,0,0; U_{\infty})$.

Notice that, by a direct computation,
\begin{eqnarray*}
\Delta\tilde{\sigma}-\Delta\sigma
=\sigma_{2}-\bar{\sigma}_{0}+\bar{\sigma}_{0}-\sigma_{0}=O(1)\tilde{x}_{0}|x_{k}^{-1}-x_{k-1}^{-1}|.
\end{eqnarray*}
Then we have
\begin{eqnarray*}
\delta_1=\delta'_1+O(1)|\Delta\tilde{\sigma}-\Delta\sigma|=\delta'_1+O(1)\tilde{x}_{0}|x_{k}^{-1}-x_{k-1}^{-1}|,
\end{eqnarray*}
where $\delta'_1(\alpha_1,\beta_2,\omega_{k},\Delta\sigma)=\left.\delta_1\right|_{\Delta\tilde{\sigma}=\Delta\sigma}$
solves the equation:
\begin{eqnarray}\label{eq:6.17}
\Phi_1(\delta'_1;\Psi(\Delta\sigma,\sigma_1;U_{b}))\cdot\mathbf{n}_{k}
=\Phi_1(\alpha_1;\Psi(\Delta\sigma,
\sigma_1;\Phi_2(\beta_2;U_{b})))\cdot\mathbf{n}_{k-1}.
\end{eqnarray}
Let
\begin{eqnarray*}
\delta''_{1}(\alpha_{1},\beta_{2},\Delta\sigma)=\left.\delta'_{1}\right|_{\omega_{k}=0}
=\delta_{1}(\alpha_{1},\beta_{2},0,\Delta\sigma, \Delta\sigma,\sigma_{1}; U_{b}).
\end{eqnarray*}
Then there exists some $K_{\rm b}\in C^{1}$ such that
\begin{eqnarray*}
\delta'_1-\delta''_1=K_{\rm b}\omega_{k}.
\end{eqnarray*}

\smallskip
2. To estimate $K_{\rm b}$, we compute $\left.\frac{\partial \delta'_1}
{\partial\omega_k}\right|_{\{\alpha_1=\beta_2=\omega_{k}
=\Delta\sigma=0,\ \theta_{k}=\theta_{0}\}}$.
To do this, we take the derivative of both sides of equation \eqref{eq:6.17} with respect to $\omega_k$,
and then let $\alpha_1=\beta_2=\omega_{k}=\Delta\sigma=0$ and $\theta_{k}=\theta_{0}$
to obtain
\begin{eqnarray*}
r_1(U_b)\cdot\mathbf{n}_{0}\left.\frac{\partial \delta'_1}
{\partial\omega_k}\right|_{\{\alpha_1=\beta_2=\omega_{k}=\Delta\sigma=0,\,\theta_{k}=\theta_{0}\}}
=U_{b}\cdot (\cos\theta_0, \sin\theta_0).
\end{eqnarray*}
Then we have
\begin{eqnarray*}
&& K_{\rm b}\big|_{\{\alpha_{1}=\beta_2=\omega_{k}=\Delta\sigma=0,\ \theta_{k}=\theta_{0}\}}
=\left.\frac{\partial \delta_2}{\partial\omega_{k}}\right|_{\{\alpha_{1}
=\beta'_1=\omega_{k}=\Delta\sigma=\tilde{x}_{0}=0,\ \theta_{k}=\theta_{0}\}}
=\frac{u_{b}\cos\theta_0+v_{b}\sin\theta_0}{r_1(U_{b})\cdot\mathbf{n}_{0}},\\
&&\lim_{M_{\infty}\rightarrow\infty}\left. K_{\rm b}\right|_{\{\beta_2
=\omega_{k}=\Delta\sigma=0,\ \theta_{k}=\theta_{0}\}}
=\lim_{M_{\infty}\rightarrow \infty}\frac{u_{b}\cos\theta_0
+v_{b}\sin\theta_0}{r_1(U_{b})\cdot {\mathbf{n}_{0}}}
=\frac{1}{\sqrt{1+b^{2}_{0}}}<\infty,
\end{eqnarray*}
which is uniformly bounded as $M_{\infty}\rightarrow \infty$.

\medskip
3. Now we are in position to estimate $\delta''_{1}(\alpha_{1},\beta_{2},\Delta\sigma)$.
Notice that
\begin{eqnarray*}
\delta''_{1}(\alpha_{1},0,\Delta\sigma)=\delta''_{1}(\alpha_{1},0,0)=\alpha_{1}.
\end{eqnarray*}
Then, by \eqref{eq:6.2},
we have
\begin{align*}
\delta''_{1}(\alpha_{1},\beta_{2},\Delta\sigma)&=\delta''_{1}(\alpha_{1},0,\Delta\sigma)
+\delta''_{1}(\alpha_{1},\beta_{2},0)-\delta''_{1}(\alpha_{1},0,0)+O(1)|\beta_{2}||\Delta\sigma|\\[5pt]
&=\delta''_{1}(\alpha_{1},\beta_{2},0)+O(1)|\beta_{2}||\Delta\sigma|.
\end{align*}
Let
$\delta'''_{1}(\alpha_{1}, \beta_{2})=\left.\delta''_{1}\right|_{\Delta\sigma=0}$.
Then there exists  $K_{\rm r}\in C^{1}$ such that
\begin{eqnarray*}
\delta'''_{1}(\alpha_{1}, \beta_{2})=\delta'''_{1}(\alpha_{1}, 0)+K_{\rm r}\beta_2.
\end{eqnarray*}

Note that $\delta'''_{1}(\alpha_{1}, \beta_{2})$ solves the following equation:
\begin{equation}
\Phi_1(\delta'''_{1};U_{b})\cdot\mathbf{n}_{k-1}
=\Phi(\alpha_{1}; \Phi_2(\beta_2;U_{b}))\cdot\mathbf{n}_{k-1}.\label{eq:6.18}
\end{equation}
We take the derivative of both sides of equation \eqref{eq:6.18} with respect to $\beta_2$ and let
$\alpha_1=\beta_2=0$ and $\theta_{k}=\theta_{0}$ to obtain
\begin{eqnarray*}
(r_1(U_{b})\cdot\mathbf{n}_{0})\left.\frac{\partial \delta'''_{1}}{\partial\beta_2}\right|_{\{\alpha_1=\beta_2=0,\ \theta_{k}
=\theta_{0}\}}=r_2(U_{b})\cdot\mathbf{n}_{0}.
\end{eqnarray*}
It follows that
\begin{eqnarray*}
\left. K_{\rm r}\right|_{\{\alpha_1=\beta_2=0,\,\theta_{k}=\theta_{0}\}}
=\left.\frac{\partial \delta_3}{\partial\beta_2}\right|_{\{\alpha_1=\beta_2=0,\,\theta_{k}=\theta_{0}\}}
=\frac{r_2(U_{b})\cdot\mathbf{n}_{0}}{r_1(U_{b})\cdot\mathbf{n}_{0}}
=\frac{\cos^{2}(\theta_0+\theta_{\rm ma})}{\cos^{2}(\theta_0-\theta_{\rm ma})},
\end{eqnarray*}
which gives the formula for $K_{\rm r}$.

\medskip
4. Finally, we combine the estimates of $\delta'_1,\ \delta''_1$, and $\delta'''_1$ with the property that
$\delta'''_{1}(\alpha_{1}, 0)=\alpha_{1}$ to conclude the desire result.
\end{proof}

\subsection{$\Lambda$ covers part of $y=\chi_{\Delta x, \vartheta}(x)$ but none of $y=b_{\Delta }(x)$}

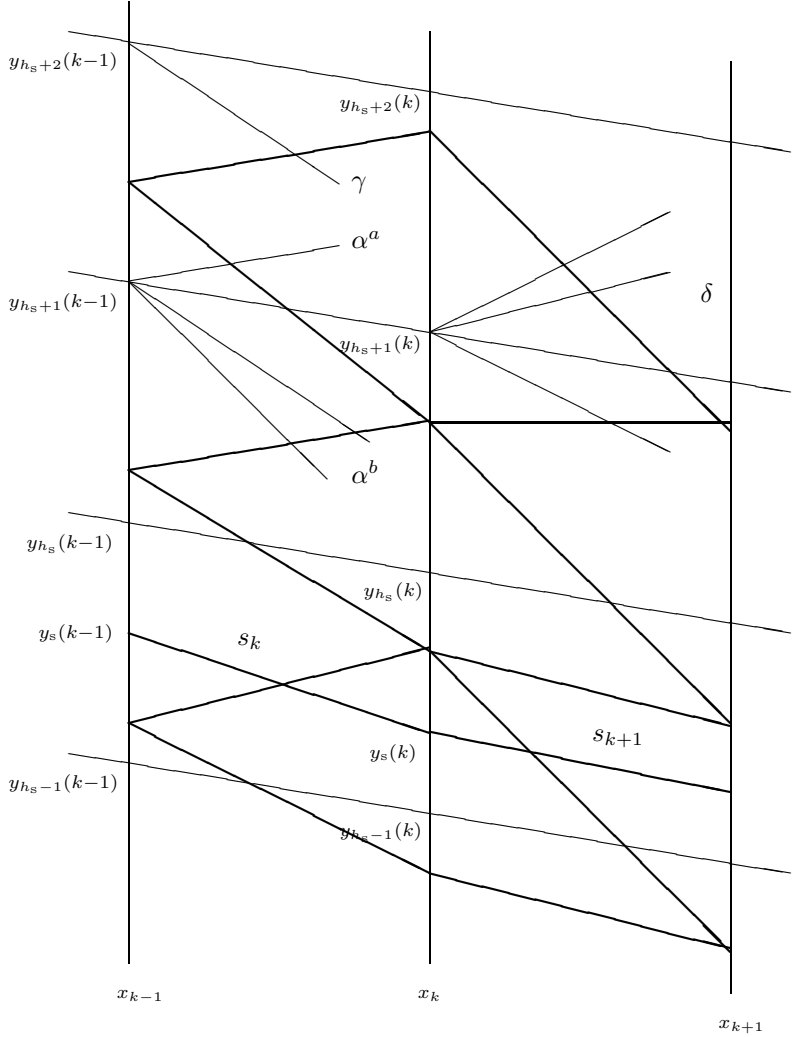
\begin{figure}[ht]
\begin{center}
\setlength{\unitlength}{0.8mm}
\begin{picture}(50,105)(-15,-60)
\put(-40,-115){\line(0,1){160}}\put(10,-115){\line(0,1){155}}\put(60,-120){\line(0,1){155}}
\put(-50,40){\line(6,-1){120}}
\put(-50,0){\line(6,-1){120}}
\put(-50,-40){\line(6,-1){120}}
\put(-50,-80){\line(6,-1){120}}
\put(-40,38){\line(3,-2){35}}
\put(-40,-1.5){\line(6,1){35}}\put(-40,-1.5){\line(1,-1){33}}\put(-40,-1.5){\line(3,-2){40}}
\put(10,-10){\line(2,1){40}}\put(10,-10){\line(4,1){40}}\put(10,-10){\line(2,-1){40}}
\thicklines
\put(-40,15){\line(5,-4){50}}\put(-40,15){\line(6,1){50}}
\put(10,23.5){\line(1,-1){50}}\put(10,-25){\line(6,0){50}}
\put(-40,-33){\line(6,1){50}}\put(10,-25){\line(1,-1){50}}
\put(-40,-33){\line(5,-3){50}}\put(10,-63){\line(4,-1){50}}
\put(-40,-75){\line(4,1){50}}\put(10,-63){\line(1,-1){50}}
\put(-40,-75){\line(2,-1){50}}\put(10,-100){\line(4,-1){50}}
\put(-40,-60){\line(3,-1){50}}\put(10,-76.5){\line(5,-1){50}}
\put(-3,14){$\gamma$}
\put(-3,4){$\alpha^{a}$}
\put(-3,-35){$\alpha^{b}$}
\put(55,-5){$\delta$}
\put(-22,-62){$s_{k}$}
\put(37,-78){$s_{k+1}$}
\put(-42,-120){$_{x_{k-1}}$}\put(8,-120){$_{x_{k}}$} \put(58,-125){$_{x_{k+1}}$}
\put(-60,-85){$_{y_{h_{\rm s}-1}(k-1)}$}\put(-55,-60){$_{y_{\rm s}(k-1)}$}\put(-57,-45){$_{y_{h_{\rs}}(k-1)}$}
\put(-60,-5){$_{y_{h_{\rm s}+1}(k-1)}$}\put(-60,35){$_{y_{h_{\rs}+2}(k-1)}$}
\put(-5,-93){$_{y_{h_{\rs}-1}(k)}$}\put(0,-80){$_{y_{\rs}(k)}$}\put(-1,-53){$_{y_{h_{\rs}}({k})}$}
\put(-5,-12){$_{y_{h_{\rs}+1}({k})}$}\put(-5,28){$_{y_{h_{\rs}+2}({k})}$}
\end{picture}
\end{center}\vspace{150pt}
\caption{Interactions between weak waves and the strong wave}\label{fig6.3}
\end{figure}

We take three diamonds simultaneously.
As shown in Fig. \ref{fig6.3}, let $\Delta_{k, y_{h_{\rs}-1}(k)}$,
$\Delta_{k, y_{h_{\rs}}(k)}$, and $\Delta_{k, y_{h_{\rs}+1}(k)}$
be the diamonds centered in $(x_{k}, y_{h_{\rs}-1}(k))$,
$(x_{k}, y_{h_{\rs}}(k))$, and $(x_{k}, y_{h_{\rs}+1}(k))$, respectively.
Denote $\Lambda=\Delta_{k, y_{h_{\rs}-1}(k)}\cup\Delta_{k, y_{h_{\rs}}(k)} \cup \Delta_{k, y_{h_{\rs}+1}(k)}$.
Let $\alpha$ and $\gamma$ be the weak waves issuing
from $(x_{k-1}, y_{h_{\rs}+1}(k-1))$ and $(x_{k-1}, y_{h_{\rs}+2}(k-1))$,
respectively, and entering $\Lambda$.
We divide $\alpha$ into parts
$\alpha^{b}=(\alpha^{b}_{1}, 0)$ and $\alpha^{a}=(\alpha^{a}_{1}, \alpha^{a}_{2})$
with $\alpha^{b}$ and $\alpha^{a}$ entering $\Delta_{k, y_{h_{\rs}}(k)}$ and $\Delta_{k, y_{h_{\rs}+1}(k)}$,
respectively.
We also assume $\gamma=(\gamma_1, 0)$, and denote $\delta$ as the outgoing waves that issue
from $(x_{k}, y_{h_{\rs}+1}(k))$.

The center in the region between
$s_{k}$
and the lower edge of $\alpha$ is defined as $O_{1}$,
in the region between the upper edge of $\alpha$
and the lower edge of $\gamma$ is defined as $O_{2}$,
and above the lower edge of $\gamma$ is denoted as $O_{3}$.

Denote the $x$--coordinate of $O_{j}$ by $X^{*}_{j}$, $j=1,2,3$, and
denote $\tilde{x}_{0}:=|X^{*}_{1}-X^{*}_{2}|$ and $\tilde{x}_{1}:=|X^{*}_{2}-X^{*}_{3}|$.
We also use the coordinates $\sigma=\sigma(x,y)=\frac{y}{x-X^{*}_{1}}$,
$\bar{\sigma}=\bar{\sigma}(x,y)=\frac{y}{x-X^{*}_{2}}$,
and $\tilde{\sigma}=\tilde{\sigma}(x,y)=\frac{y}{x-X^{*}_{3}}$.
Also denote
\begin{align*}
&\sigma_{\alpha}=\sigma(x_{k-1}, y_{h_{\rs}+1}(k-1)), \,\,
  \sigma_{\rs}(k-1)=\sigma(x_{k-1}, y_{\rs}(k-1)), \,\,
   \sigma_{\rs}(k)=\sigma(x_{k}, y_{\rs}(k)),\\[1mm]
&\bar{\sigma}_{\alpha}=\bar{\sigma}(x_{k-1}, y_{h_{\rs}+1}(k-1)),
 \,\, \bar{\sigma}_{\rs}(k-1)=\bar{\sigma}(x_{k-1}, y_{\rs}(k-1)),
 \,\, \bar{\sigma}_{\rs}(k)=\bar{\sigma}(x_{k}, y_{\rs}(k)),\\[1mm]
&\Delta\sigma_{\alpha}=\sigma_{\alpha}-\sigma_{\rs}(k), \,\,\,
 \Delta\sigma_{s_{k}}=\sigma_{\rs}(k)-\sigma_{\rs}(k-1), \\[1mm]
&\Delta\tilde{\sigma}_{\gamma}=\tilde{\sigma}(x_{k}, y_{h_{\rs}+1}(k))-\tilde{\sigma}(x_{k-1}, y_{h_{\rs}+2}(k-1)).
\end{align*}

\smallskip
To obtain the estimates of $(s_{k+1},\delta)$, we first consider the following equation:
\begin{eqnarray}
&&\Psi(\bar{\sigma}_{\alpha}-\bar{\sigma}_{\rs}(k),
\bar{\sigma}_{\rs}(k);\Phi_{2}(\varepsilon_{2};\Theta(s_{k+1}))) \label{eq:6.19}\\
&&=\Phi_{1}(\alpha^{b}_{1}; \Psi(\sigma_{\alpha}-\sigma_{\rs}(k-1),\sigma_{\rs}(k-1);\Theta(s_{k}))). \nonumber
\end{eqnarray}
With solutions $(s_{k+1},\varepsilon_2)$ of \eqref{eq:6.19} and the construction of the approximate solution,
we now give the estimates on the weak wave $\delta$.

\smallskip
\begin{lemma}\label{lem:6.6}
The following asymptotic expansions hold{\rm :}
\begin{eqnarray}
&&\qquad\qquad \delta_1=\alpha^{a}_1+\gamma_{1}+O(1)Q(\Lambda),\label{eq:6.20}\\[5pt]
&&\qquad\qquad \delta_2=\alpha^{a}_2+K_{\rm w}\alpha^{b}_{1}+\mu_{\rm w}\Delta\sigma_{\rs_{k}}+O(1)Q(\Lambda),\label{eq:6.21}\\[5pt]
&&\qquad\qquad s_{k+1}=s_k+K_{\rm s}\alpha^{b}_{1}+\mu_{\rm s}\Delta\sigma_{\rs_{k}}
+O(1)\big(|\Delta\sigma_{\alpha}|+|\Delta\sigma_{\rs_k}|
+|x_{k}^{-1}-x_{k-1}^{-1}|\big)\tilde{x}_0,\quad \label{eq:6.22}
\end{eqnarray}
where
\begin{align*}
Q(\Lambda)=&\, Q^{0}(\alpha^{a},\gamma)+|\Delta\sigma_{\alpha}|(|\alpha^{b}_{1}|+\tilde{x}_{0})
+|\Delta\sigma_{\rs_{k}}|(|\Delta\sigma_{\alpha}|+\tilde{x}_{0})\\[1mm]
&\, + |\Delta\tilde{\sigma}_{\gamma}|(|\gamma|+\tilde{x}_1)
+|x_{k}^{-1}-x_{k-1}^{-1}|(\tilde{x}_{0}+\tilde{x}_1).
\end{align*}
In addition, for $\alpha^{b}_1=\Delta\sigma_{\alpha}=\Delta\sigma_{\rs_{k}}=\tilde{x}_0=0$ and $s_{k}=s_0$,
\begin{align}
&K_{\rm w}=\frac{e_{1}(s_{0})}{e_{2}(s_{0})}
\frac{\tilde{u}_{s}(s_0)+\lambda_{1}(s_0)\tilde{v}_{s}(s_0)}
{\tilde{u}_{s}(s_0)+\lambda_{2}(s_0)\tilde{v}_{s}(s_0)},
\quad &&K_{\rm s}=\frac{e_{1}(s_{0})\big(\lambda_{2}(s_0)-\lambda_{1}(s_0)\big)}
{\tilde{u}_{s}(s_0)+\lambda_{2}(s_0)\tilde{v}_{s}(s_0)},\label{eq:6.22a} \\[1.5mm]
&\mu_{\rm w}=\frac{\tilde{u}_{s}(s_0)v_{\sigma}(s_0)-\tilde{v}_{s}(s_0)u_{\sigma}(s_0)}
{e_{2}(s_{0})\big(\tilde{u}_{s}(s_0)+\lambda_{2}(s_0)\tilde{v}_{s}(s_0)\big)},
\quad &&\mu_{\rm s}=\frac{u_{\sigma}(s_0)+\lambda_{2}(s_0)v_{\sigma}(s_0)}
{\tilde{u}_{s}(s_0)+\lambda_{1}(s_0)\tilde{v}_{s}(s_0)}.\label{eq:6.23}
\end{align}
Furthermore, for $M_{\infty}$ sufficiently large,
\begin{equation}\label{eq:6.24}
\begin{aligned}
K_{\rm w}&=1+ 2 b_0^{-1}(1+2b^{2}_{0})\sqrt{1+b^{2}_{0}}\,M_{\infty}^{-1}
+O(1)M^{-2}_{\infty}+O(1)e^{-m_{0}M^{2}_{\infty}}, \\[1mm]
K_{\rm s}&=- 2 b_0^{-1}\sqrt{1+b^{2}_{0}}\,M_{\infty}^{-1}+O(1)M^{-2}_{\infty}+O(1)e^{-m_{0}M^{2}_{\infty}},\\[1mm]
\mu_{\rm w}&=-1- b_0^{-1}(1+2b^{2}_{0})\sqrt{1+b^{2}_{0}} \, M_{\infty}^{-1}
+O(1)M^{-2}_{\infty}+O(1)e^{-m_{0}M^{2}_{\infty}},\\[1mm]
\mu_{\rm s}&= b_0^{-1}\sqrt{1+b^{2}_{0}}\, M_{\infty}^{-1}+O(1)M^{-2}_{\infty}+O(1)e^{-m_{0}M^{2}_{\infty}},
\end{aligned}
\end{equation}
where $O(1)$ depends continuously only on $M_{\infty}$.
\end{lemma}

\smallskip
\begin{proof} We divide the proof into four steps.

\medskip
1.  Lemmas \ref{lem:2.13} and \ref{lem:2.17} imply that
\begin{eqnarray*}
\lim_{M_{\infty}\rightarrow \infty}\frac{\det(r_2(\Theta(s_0)),
 \Theta_s(s_0))}{M^{2}_\infty}
&=&-\lim_{M_{\infty}\rightarrow \infty}\frac{e_{2}(s_0)}{M_{\infty}}\frac{\tilde{u}_{s}(s_0)
+\lambda_{2}(s_0)\tilde{v}_{s}(s_0)}{M_{\infty}}\\
&=&\frac{c^2 b_{0}}{(1+b^{2}_{0})^{3}}<\infty.
\end{eqnarray*}
Then, by the implicit function theorem, equation
\eqref{eq:6.19} admits a unique $C^{2}$--solution $(\delta_2, s_{k+1})$
such that
\begin{eqnarray*}
&&\varepsilon_{2}=\varepsilon_{2}(\alpha^{b}_1, s_{k},
 \sigma_{\alpha}-\sigma_{\rs}(k-1),\bar{\sigma}_{\alpha}-\bar{\sigma}_{\rs}(k),
 \sigma_{\rs}(k-1), \bar{\sigma}_{\rs}(k)),\\[1mm]
&&s_{k+1}=s_{k+1}(\alpha^{b}_1, s_{k},
 \sigma_{\alpha}-\sigma_{\rs}(k-1),\bar{\sigma}_{\alpha}-\bar{\sigma}_{\rs}(k),
 \sigma_{\rs}(k-1), \bar{\sigma}_{\rs}(k)).
\end{eqnarray*}

\smallskip
2. Denote
\begin{eqnarray*}
&&\varepsilon'_{2}=\varepsilon'_{2}(\alpha^{b}_1, s_{k},\Delta\sigma_{\alpha},\sigma_{\rs}(k))
=\varepsilon_{2}\big|_{\{\bar{\sigma}_{\alpha}-\bar{\sigma}_{\rs}(k)=
\sigma_{\alpha}-\sigma_{\rs}(k-1),\bar{\sigma}_{\rs}(k)=\bar{\sigma}_{\rs}(k-1)\}},\\[1mm]
&&s'_{k+1}=s'_{k+1}(\alpha^{b}_1, s_{k},\Delta\sigma_{\alpha},\sigma_{\rs}(k))
=s_{k+1}\big|_{\{\bar{\sigma}_{\alpha}-\bar{\sigma}_{\rs}(k)=
\sigma_{\alpha}-\sigma_{\rs}(k-1),\bar{\sigma}_{\rs}(k)=\bar{\sigma}_{\rs}(k-1)\}}.
\end{eqnarray*}
Then, by a direct computation, we have
\begin{align}
&\varepsilon_{2}= \varepsilon'_{2}+O(1)\big(\bar{\sigma}_{\rs}(k)-\bar{\sigma}_{\rs}(k-1)\big)\label{eq:6.25}\\
&\quad\quad +O(1)\big(\bar{\sigma}_{\alpha}-\bar{\sigma}_{\rs}(k)-(\sigma_{\alpha}-\sigma_{\rs}(k-1))\big),\nonumber\\[1mm]
&s_{k+1}= s'_{k+1}+O(1)\big(\bar{\sigma}_{\rs}(k)-\bar{\sigma}_{\rs}(k-1)\big)\label{eq:6.26}\\
&\qquad\quad +O(1)\big(\bar{\sigma}_{\alpha}-\bar{\sigma}_{\rs}(k)-(\sigma_{\alpha}-\sigma_{\rs}(k-1))\big),\nonumber
\end{align}
where $(\varepsilon'_{2}, s'_{k+1})$ solves the equation:
\begin{eqnarray}\label{eq:6.27}
&&\Psi(\sigma_{\alpha}-\sigma_{\rs}(k-1),\bar{\sigma}_{\rs}(k-1);\Phi_{2}(\varepsilon'_{2};
\Theta(s'_{k+1})))\\[1mm]
&&=\Phi_{1}(\alpha_{1}; \Psi(\sigma_{\alpha}-\sigma_{\rs}(k-1),
\sigma_{\rs}(k-1);\Theta(s_{k}))).\nonumber
\end{eqnarray}
Using the Taylor expansion, we have
\begin{eqnarray}\label{eq:6.28}
\varepsilon'_{2}=K_{\rm w}\alpha^{b}_1+\varepsilon''_{2}, \qquad s'_{k+1}=K_{\rm s}\alpha^{b}_1+s''_{k+1},
\end{eqnarray}
where $(\varepsilon''_{2}, s''_{k+1})$ satisfies
\begin{eqnarray}\label{eq:6.29}
&&\Psi(\sigma_{\alpha}-\sigma_{\rs}(k-1),\bar{\sigma}_{\rs}(k-1);
\Phi_{2}(\varepsilon''_{2}; \Theta(s''_{k+1})))\\[1mm]
&&= \Psi(\sigma_{\alpha}-\sigma_{\rs}(k-1), \sigma_{\rs}(k-1);\Theta(s_{k})).\nonumber
\end{eqnarray}
Since $\left. \varepsilon''_{2}\right|_{\{\sigma_{\alpha}-\sigma_{\rs}(k-1)=\tilde{x}_0=0\}}=0$
and $\left. s''_{k+1}\right|_{\{\sigma_{\alpha}-\sigma_{\rs}(k-1)=\tilde{x}_0=0\}}=0$,
by \eqref{eq:6.2},
\begin{eqnarray}\label{eq:6.30}
\varepsilon''_{2}=O(1)\tilde{x}_0|\Delta\sigma_{\alpha}|,
\qquad s''_{k+1}=O(1)\tilde{x}_0|\Delta\sigma_{\alpha}|.
\end{eqnarray}
Notice that
\begin{align}
&\bar{\sigma}_{\rs}(k)-\bar{\sigma}_{\rs}(k-1)=\Delta\sigma_{\rs_{k}}
+O(1)\tilde{x}_0\big(|\Delta\sigma_{\rs_{k}}|+|x_{k}^{-1}-x_{k-1}^{-1}|\big),\label{eq:6.31a}\\[1.5pt]
&\bar{\sigma}_{\alpha}-\bar{\sigma}_{\rs}(k)-\big(\sigma_{\alpha}-\sigma_{\rs}(k-1)\big)\label{eq:6.31}\\[1mm]
& \ \ \  =\bar{\sigma}_{\alpha}-\bar{\sigma}_{\rs}(k-1)-\big(\sigma_{\alpha}-\sigma_{\rs}(k-1)\big)
+\bar{\sigma}_{\rs}(k-1)-\bar{\sigma}_{\rs}(k)\nonumber\\[1mm]
&\ \ \ =-\Delta\sigma_{\rs_{k}}+O(1)\tilde{x}_0\big(|\Delta\sigma_{\alpha}|
+|\Delta\sigma_{\rs_{k}}|+|x_{k}^{-1}-x_{k-1}^{-1}|\big).\nonumber
\end{align}
Combining estimates \eqref{eq:6.25}--\eqref{eq:6.31} together, we obtain
the estimates of $(\varepsilon_2, s_{k+1})$:
\begin{align}
&\varepsilon_2=K_{\rm w}\alpha^{b}_1+\mu_{\rm w}\Delta\sigma_{\rs_{k}}
+O(1)\tilde{x}_0\big(|\Delta\sigma_{\alpha}|
+|\Delta\sigma_{\rs_{k}}|+|x_{k}^{-1}-x_{k-1}^{-1}|\big),\label{eq:6.32}\\
&s_{k+1}=s_{k}+K_{\rm s}\alpha^{b}_1+\mu_{\rm s}\Delta\sigma_{\rs_{k}}
+O(1)\tilde{x}_0\big(|\Delta\sigma_{\alpha}|
+|\Delta\sigma_{\rs_{k}}|+|x_{k}^{-1}-x_{k-1}^{-1}|\big).\label{eq:6.33}
\end{align}

\medskip
3. To compute coefficients $K_{\rm w}, K_{\rm s}, \mu_{\rm w}$, and $\mu_{\rm s}$,
we differentiate \eqref{eq:6.18} with respect to
$\alpha^{b}_1$ and $\Delta\sigma_{\rs_{k}}$, and
let $\alpha^{b}_1=\Delta\sigma_{\alpha}=\Delta\sigma_{\rs_{k}}=\tilde{x}_{0}=0$
and $s_{k}=s_{0}$ to obtain
\begin{eqnarray*}
&& r_{2}(\Theta(s_{0}))K_{\rm w}+\Theta_{s}(s_{0})K_{\rm s}=r_{1}(\Theta(s_{0})),\\
&& r_{2}(\Theta(s_{0}))\mu_{\rm w}+\Theta_{s}(s_{0})\mu_{\rm s}
=\frac{\partial\Psi}{\partial(\Delta\sigma_{\rs_k})}(0, \sigma_{\rs}(k);\Theta(s_{0})).
\end{eqnarray*}
Cramer's rule implies
\begin{eqnarray*}
&&K_{\rm w}=\frac{\det(r_1(\Theta(s_0)),\ \Theta_s(s_0))}
{\det(r_2(\Theta(s_0)),\ \Theta_s(s_0))}=\frac{e_1(s_0)}{e_2(s_0)}\frac{\tilde{u}_{s}(s_0)+\lambda_{1}(s_0)\tilde{v}_{s}(s_0)}
{\tilde{u}_{s}(s_0)+\lambda_{2}(s_0)\tilde{v}_{s}(s_0)},\\
&& K_{\rm s}=\frac{\det(r_2(\Theta(s_0)),\ r_1(\Theta(s_0)))}
{\det(r_2(\Theta(s_0)),\ \Theta_s(s_0))}=\frac{e_1(s_0)(\lambda_{2}(s_0)-\lambda_{1}(s_0))}
{\tilde{u}_{s}(s_0)+\lambda_{2}(s_0)\tilde{v}_{s}(s_0)},
\end{eqnarray*}
and
\begin{eqnarray*}
&&\mu_{\rm w}=\frac{\det(\frac{\partial\Psi}{\partial(\Delta\sigma_{\rs})}
(0,\sigma_{\rs}(k); \Theta(s_0)), \Theta_s(s_0))}
{\det(r_2(\Theta(s_0)), \Theta_s(s_0))},\\[1mm]
&&\mu_{\rm s}=\frac{\det(r_2(\Theta(s_0)), \frac{\partial\Psi}
{\partial(\Delta\sigma_{\rs})}(0,\sigma_{\rs}(k); \Theta(s_0)))}
{\det(r_2(\Theta(s_0)), \Theta_s(s_0))}.
\end{eqnarray*}
Then, by Lemmas \ref{lem:2.13}--\ref{lem:2.14} and \ref{lem:2.17},
we can estimate $ K_{\rm w}$, $K_{\rm s}$, $ \mu_{\rm w}$,
and $\mu_{\rm s}$ as expected, when $M_{\infty}$ is sufficiently large.

\smallskip
4. We finally give the estimates of $\delta$.
By the construction of the approximate solution, we have
\begin{eqnarray}\label{eq:6.34}
&&\Phi(\delta;\Psi(\Delta\bar{\sigma}_{\alpha}, \bar{\sigma}_{\rs}(k);\Theta(s_{k+1})))\\[1mm]
&&=\Psi(\Delta\tilde{\sigma}_{\gamma},\tilde{\sigma}(x_{k-1}, y_{h_{\rs}+2}(k-1));\Phi(\gamma;
\Psi(\bar{\sigma}_{\gamma}-\bar{\sigma}_{\alpha},\bar{\sigma}_{\alpha}; U_{\rm f}))) \nonumber
\end{eqnarray}
with
\begin{eqnarray}\label{eq:6.35}
U_{\rm f}=\Phi(\alpha^{a}; \Psi(\Delta\bar{\sigma}_{\alpha}, \bar{\sigma}_{\rs}(k);
\Phi_{2}(\varepsilon_{2};\Theta(s_{k+1})))).
\end{eqnarray}
Then, as was done in \S 6.1,  we obtain  \eqref{eq:6.20}--\eqref{eq:6.22}.
\end{proof}

\smallskip
\begin{lemma}\label{lem:6.7}
For $\Delta x$ sufficiently small,
\begin{eqnarray}\label{eq:6.36}
\big|\sigma_{\rs}(k-1)-s_{k}\big|\geq6|\Delta\sigma_{\rs_{k}}|.
\end{eqnarray}
\end{lemma}

\begin{proof}
Notice that
\begin{eqnarray*}
s_{k}=\frac{y_{\rs}(k)-y_{\rs}(k-1)}{\Delta x},\quad \,   \sigma_{\rs}(k)=\frac{y_{\rs}(k)}{x_{k}-X^{*}_{1}}.
\end{eqnarray*}
Then, by a direct computation, we have
\begin{eqnarray*}
\big|\sigma_{\rs}(k-1)-s_{k}\big|&=&\Big|\frac{y_{\rs}(k)-y_{\rs}(k-1)}{\Delta x}-\sigma_{\rs}(k-1)\Big|\\[5pt]
&=&\Big|\frac{\sigma_{\rs}(k)(x_{k}-X^{*}_{1})-\sigma_{\rs}(k-1)(x_{k-1}-X^{*}_{1})}{\Delta x}-\sigma_{\rs}(k-1)\Big|\\[5pt]
&=&\Big|\frac{\sigma_{\rs}(k)-\sigma_{\rs}(k-1)}{\Delta x}(x_{k}-X^{*}_{1})\Big|\\[5pt]
&\geq& 6\big|\sigma_{\rs}(k)-\sigma_{\rs}(k-1)\big|\qquad\,\, \mbox{for $\Delta x$ sufficiently small}.
\end{eqnarray*}
\end{proof}

Denote $\theta_{\rs}(k)=|\sigma_{\rs}(k-1)-s_{k}|$, which measures the angle between the leading shock
$s_{k}$ and the line passing through $( x_{k-1}, y_{\rs}(k-1))$ and the center of $s_{k}$. Moreover,
we have the following estimate for $\theta_{\rs}(k)$.

\smallskip
\begin{lemma}\label{lem:6.8}
For $M_{\infty}$ sufficiently large and $\Delta x$ sufficiently small,
\begin{equation}\label{eq:6.37}
\,\,\, \theta_{\rs}(k)-\theta_{\rs}(k+1)
\geq |\Delta\sigma_{\rs_{k}}| -|K_{\rm s}||\alpha^{b}_{1}|
 -C\tilde{x}_{0}\big(|\Delta\sigma_{\alpha}|
+|\Delta\sigma_{\rs_{k}}|+\big|x_{k}^{-1}-x_{k-1}^{-1}\big|\big),
\end{equation}
where $ k\geq 0$, $K_{\rm s}$ is given by Lemma {\rm \ref{lem:6.6}},
and constant $C>0$ is independent of $M_{\infty}$ and  $\Delta x$.
\end{lemma}

\smallskip
\begin{proof}
The proof is divided into two subcases.

\medskip
1. $\sigma_{\rs}(k-1)<s_{k}$ so that $\sigma_{\rs}(k-1)<\sigma_{\rs}(k)$.

\medskip
\begin{itemize}
\item If $s_{k+1}>\sigma_{\rs}(k)$, then, by Lemma \ref{lem:6.6},
\begin{align*}
\theta_{\rs}(k)-\theta_{\rs}(k+1)&=s_{k}-\sigma_{\rs}(k-1)-\big(s_{k+1}-\sigma_{\rs}(k)\big)\\[1mm]
&=(1-\mu_{\rm s})\Delta\sigma_{\rs_{k}}-K_{\rm s}\alpha^{b}_{1} \\
&\quad +O(1)\tilde{x}_{0}\big(|\Delta\sigma_{\alpha}|
+|\Delta\sigma_{\rs_{k}}|+\big|x_{k}^{-1}-x_{k-1}^{-1}\big|\big)\\[1mm]
&\geq|\Delta\sigma_{\rs_{k}}|-|K_{\rm s}||\alpha^{b}_{1}|\\[1mm]
&\quad -C\tilde{x}_{0}\big(|\Delta\sigma_{\alpha}|
+|\Delta\sigma_{\rs_{k}}|+\big|x_{k}^{-1}-x_{k-1}^{-1}\big|\big).
\end{align*}

\medskip
\item If $s_{k+1}<\sigma_{\rs}(k)$, then, by Lemmas \ref{lem:6.6}--\ref{lem:6.7},
\begin{align*}
\theta_{\rs}(k)-\theta_{\rs}(k+1)&=s_{k}-\sigma_{\rs}(k-1)-\big(s_{k+1}-\sigma_{\rs}(k)\big)\\[1mm]
&=2\big(s_{k}-\sigma_{\rs}(k-1)\big)+s_{k+1}-\sigma_{\rs}(k)-\big(s_{k}-\sigma_{\rs}(k-1)\big)\\[1mm]
&\geq(11+\mu_{\rs})|\Delta\sigma_{\rs_{k}}|+K_{\rm s}\alpha^{b}_{1}\\
&\quad  -C\tilde{x}_{0}\big(|\Delta\sigma_{\alpha}|
+|\Delta\sigma_{\rs_{k}}|+\big|x_{k}^{-1}-x_{k-1}^{-1}\big|\big)\\[1mm]
&\geq|\Delta\sigma_{\rs_{k}}|-|K_{\rm s}| |\alpha^{b}_{1}|\\[1mm]
&\quad -C\tilde{x}_{0}\big(|\Delta\sigma_{\alpha}|
+|\Delta\sigma_{\rs_{k}}|+\big|x_{k}^{-1}-x_{k-1}^{-1}\big|\big).
\end{align*}
\end{itemize}

\smallskip
2. $s_{k}<\sigma_{\rs}(k-1)$ so that $\sigma_{\rs}(k)<\sigma_{\rs}(k-1)$.

\medskip
\begin{itemize}
\item If $s_{k+1}>\sigma_{\rs}(k)$, then, by Lemmas \ref{lem:6.6}--\ref{lem:6.7},
\begin{align*}
\theta_{\rs}(k)-\theta_{\rs}(k+1)&=s_{k}-\sigma_{\rs}(k)-\big(s_{k+1}-\sigma_{\rs}(k+1)\big)\\[1mm]
&=2\big(s_{k}-\sigma_{\rs}(k)\big)+s_{k}-\sigma_{\rs}(k-1)-\big(s_{k+1}-\sigma_{\rs}(k)\big)\\[1mm]
&\geq(\mu_{\rs}+11)|\Delta\sigma_{\rs_{k}}|-K_{\rs}\alpha^{b}_{1}\\
&\quad -C\tilde{x}_{0}\big(|\Delta\sigma_{\alpha}|
+|\Delta\sigma_{\rs_{k}}|+\big|x_{k}^{-1}-x_{k-1}^{-1}\big|\big)\\[1mm]
&\geq|\Delta\sigma_{\rs_{k}}|-|K_{\rm s}| |\alpha^{b}_{1}|\\[1mm]
&\quad -C\tilde{x}_{0}\big(|\Delta\sigma_{\alpha}|
+|\Delta\sigma_{\rs_{k}}|+\big|x_{k}^{-1}-x_{k-1}^{-1}\big|\big).
\end{align*}

\smallskip
\item If $s_{k+1}<\sigma_{\rs}(k)$, then, by Lemma \ref{lem:6.6},
\begin{align*}
\theta_{\rs}(k)-\theta_{\rs}(k+1)&=\sigma_{\rs}(k-1)-s_{k}-\sigma_{\rs}(k)+s_{k+1}\\[5pt]
&=(1-\mu_{\rm s})|\Delta\sigma_{\rs_{k}}|+K_{\rm s}\alpha^{b}_{1}\\
&\quad +O(1)\tilde{x}_{0}\big(|\Delta\sigma_{\alpha}|
+|\Delta\sigma_{\rs_{k}}|+\big|x_{k}^{-1}-x_{k-1}^{-1}\big|\big)\\[1mm]
&\geq|\Delta\sigma_{\rs_{k}}|-|K_{\rm s}||\alpha^{b}_{1}|\\[1mm]
&\quad -C\tilde{x}_{0}\big(|\Delta\sigma_{\alpha}|
+|\Delta\sigma_{\rs_{k}}|+\big|x_{k}^{-1}-x_{k-1}^{-1}\big|\big).
\end{align*}
\end{itemize}

In the above estimates, we have used the fact that $\mu_{\rm s}\in(-1,0)$ for $M_{\infty}$
sufficiently large. This completes the proof.
\end{proof}

\section{The Glimm-Type Functional and the Convergence of the Approximate Solutions}\setcounter{equation}{0}
In this section, we first apply the difference scheme and the local interaction estimates
obtained in \S 6 above to construct a suitable Glimm-type functional for the approximate solutions,
and then prove its monotonicity
so that the total variation of the approximate solutions in $y$ is uniformly bounded in $x$.
Thus, we first state a lemma  which is important to
prove the monotonicity of the Glimm-type functional.

\smallskip
\begin{lemma} \label{lem:7.1}
Let $K_{\rm r}$, $K_{\rm w}$, $K_{\rm s}$, and $\mu_{\rm w}$ be given by Lemmas {\rm \ref{lem:6.5}}--{\rm \ref{lem:6.7}}.
Then, for $M_\infty$ sufficiently large,
\begin{eqnarray}\label{eq:7.1}
|K_{\rm r}|\big(|K_{\rm w}|+|K_{\rm s}||\mu_{\rm w}|\big)<1.
\end{eqnarray}
Moreover, there exist positive constants $K_1$, $K_2$, and $K_3$ such that
\begin{eqnarray}\label{eq:7.2}
|K_{\rm r}|-K_{2}<0,\quad  \ K_{2}|\mu_{\rm w}|-K_{3}<0,\quad K_{2}|K_{\rm w}|+K_{3}|K_{\rm s}|-1<0 .
\end{eqnarray}
\end{lemma}

\smallskip
\begin{proof}
By (6.15)
in Lemma \ref{lem:6.5}
and  (6.24)
in Lemma \ref{lem:6.6},
we obtain that, for $M_\infty$ sufficiently large,
\begin{align}
&|K_{\rm r}|\big(|K_{\rm w}|+|K_{\rm s}||\mu_{\rm w}|\big)\label{eq:7.3}\\[1mm]
&=1- 2 b_0^{-2}(1+b^{2}_{0})(8b^{4}_0+2b^{2}_0+1)\,M_\infty^{-1}
 +O(1)M^{-2}_{\infty}+O(1)e^{-m_{0}M^{2}_{\infty}}.\nonumber
\end{align}

Note that the second term on the right-hand side of \eqref{eq:7.3}
is negative.
Thus, we can choose $M_{\infty}$ sufficiently large such that \eqref{eq:7.1} holds.
It follows that there exists a constant $K_2>0$ such that
\begin{eqnarray*}
K_2>|K_{\rm r}|,\qquad   K_2\big(|K_{\rm w}|+|K_{\rm s}||\mu_{\rm w}|\big)<1,
\end{eqnarray*}
which leads to
\begin{eqnarray*}
 K_2|K_{\rm s}||\mu_{\rm w}|<1-K_2|K_{\rm w}|.
\end{eqnarray*}
Then we can choose another constant $K_{3}>0$ such that
\begin{eqnarray*}
 K_2|\mu_{\rm w}|-|K_3|<0,\qquad\, |K_3||K_{\rm s}|+K_2|K_{\rm w}|-1<0.
\end{eqnarray*}
This completes the proof.
\end{proof}

\smallskip
We now turn to the construction of the Glimm functional and study its properties.
Let $J$ be a space-like mesh curve connecting the mesh points.
Denote $\Gamma_J$ as the set of the corner points
with $A_{k}$ lying in $J^{+}$, {\it i.e.},
\begin{eqnarray*}
\Gamma_J=\{A_{k}\,:\,A_{k}=(x_k, b_{k}), A_{k}\in J^{+}, k\geq 0\}.
\end{eqnarray*}
Then we define the following Glimm-type functional:

\smallskip
\begin{definition}[Weighted total variation]\label{def:7.1} Denote
\begin{align*}
&L^{(i)}_{0}(J):=\sum\Big\{|\alpha_i|\,:\,\mbox{$\alpha_{i}$ is an $i$-weak wave crossing $J$} \Big\},\quad  i=1,2,\\[4pt]
&L_{1}(J):=\sum\Big\{(1+|b_{k}|)|\omega(A_k)|\,:\, A_k\in \Gamma_J\Big\},\\[4pt]
&L_{\rm s}(J):= \theta_{\rm s}(J) \qquad \mbox{for $\theta_{\rm s}(J)$ as $\theta_{\rm s}(k)$ in  Lemma {\rm \ref{lem:6.8}}
     when ${\rm s}$ crosses $J$},\\[4pt]
&L_{\rm c}(J):=\sum \Big\{\big|X^{*}_{\alpha+}-X^{*}_{\alpha-}\big|\big(1+x_{\alpha}^{-1}\big)\,:\,
  \mbox{$\alpha$ is a $1$-wave crossing $J$ from $x= x_{\alpha}$} \Big\},
\end{align*}
where $x_{\alpha}\in\{x_{k}\,:\, k\geq 0\}$, and $X^{*}_{\alpha\pm}$ denote the limits of $X^{*}$
on the right and left of $\alpha$.
Then the weighted total variation is defined as
\begin{eqnarray}\label{eq:7.4}
L(J):=L^{(1)}_{0}(J)+K_{2}L^{(2)}_{0}(J)+K_{1}L_{1}(J)+K_{3}L_{\rm s}(J)+K_{4}L_{\rm c}(J),
\end{eqnarray}
where $K_{1}, K_{2}$, and $K_{3}$ are given as in Lemma {\rm \ref{lem:7.1}},
and $K_{4}>0$ is a constant to be specified later.
\end{definition}

\smallskip
Next, we turn to the construction of quadratic terms for the total interaction potential.

\smallskip
\begin{definition}[Total interaction potential]\label{def:7.2}
Denote
\begin{align*}
&Q_{0}(J):=\sum\Big\{|\alpha_i||\beta_j|\,:\,\mbox{$\alpha_{i}$ and $\beta_j$ are weak waves that cross $J$ and approach} \Big\},\\[1mm]
&Q_{1}(J):=\sum\Big\{|\alpha|(\sigma_{\alpha}-\sigma_{*})\,:\, \mbox{$\alpha$ is a $1$-weak wave crossing $J$}\Big\},\\[1mm]
&Q_{2}(J):=\sum\Big\{|\alpha|(\sigma^{*}-\sigma_{\alpha})\,:\, \mbox{$\alpha$ is a $2$-weak wave crossing $J$} \Big\},\\[1mm]
&Q_{\rm c}(J):=\sum\Big\{|X^{*}_{\alpha+}-X^{*}_{\alpha-}|(\sigma^{\rm c}_{\alpha}(J)-\sigma_{*})\,:\,
  \mbox{$\alpha$ is an $i$-weak wave crossing $J$, $i=1,2$}\Big\},\\[1mm]
&Q_{\rm ce}(J):=\sum\Big\{|X^{*}_{\alpha+}-X^{*}_{\alpha-}||X^{*}_{\beta+}-X^{*}_{\beta-}|\,:\,
   \mbox{$\alpha$ and $\beta$ are weak waves crossing $J$}\Big\},\\[1mm]
&Q^{(j)}_{\rm wc}(J):=\sum\Big\{|\beta_{j}||X^{*}_{\alpha+}-X^{*}_{\alpha-}|:\,\mbox{$\alpha$ is a $1$-weak wave
    above a $j$-weak wave $\beta_{j}$ on $J$}\Big\}
\end{align*}
for $j=1, 2$,
where $X^{*}_{\alpha\pm}$ denote the right and left limits of $X^{*}$
of $\alpha$,
$\sigma_\alpha$ is the $\sigma$-coordinate of the grid point where $\alpha$ issues, and $\sigma^{\rm c}_\alpha(J)$ is the
$\sigma$-coordinate of the grid point where the center of the self-similar solution passing
through $J$ changes from $X^{*}_{\alpha-}$ to $X^{*}_{\alpha+}$.
In addition, denote
\begin{eqnarray*}
\sigma^{*}:=b_0+C_{1}\sum(1+|b_{k}|)|\omega_{k}|, \qquad \sigma_{*}=s_0-\varrho,
\end{eqnarray*}
where $s_0$ is the speed of the leading shock-front for the background problem, $b_{0}$ is the unperturbed boundary slope,
$\varrho$ and $C_{1}$ are positive constants chosen so that
$Q_{1}(J)$, $Q_{2}(J)$, and $Q_{\rm c}(J)$ are nonnegative.
Note that $\varrho$ and $(1+|b_{k}|)|\omega_{k}|$ are chosen small so that the largeness of $M_\infty$
implies the smallness of $s_0-b_0$ that leads to the smallness of $\sigma^{*}-\sigma_{*}$.
The summation in $Q^{(j)}_{\rm wc}(J)$ is taken over for all couples of weak waves $(\alpha, \beta_{j})$, $j=1,2$.

\par Then the total interaction potential is defined as
\begin{eqnarray}\label{eq:7.5}
Q(J):=\sum_{i=0,1,2, {\rm c}}Q_{i}(J)+\sum_{i=1,2}K^{(i)}_{\rm wc}\,Q^{(i)}_{\rm wc}(J)+K_{\rm ce}\,Q_{\rm ce}(J),
\end{eqnarray}
where
$K^{(i)}_{\rm wc}, i=1,2$, and $K_{\rm ce}$ are positive constants to be specified later.
\end{definition}

\smallskip
Finally, we give the definition of the Glimm-type functional.

\smallskip
\begin{definition}[Glimm-type functional]\label{def:7.3}
\begin{eqnarray}\label{eq:7.6}
F(J):=L(J)+K Q(J),
\end{eqnarray}
where $K>0$ is a constant to be given below.
\end{definition}

\smallskip
Let
\begin{eqnarray}\label{eq:7.7}
\qquad E_{\Delta x, \vartheta}(\Lambda)
=
\begin{cases}
Q(\Lambda)
\quad\quad\quad\quad \quad\quad \quad \ \  &\mbox{(defined in \S 6.1),}\\[5pt]
|\beta_2|+(1+|b_{k}|)|\omega_{k}| \quad\quad \quad\quad\quad \  \  &\mbox{(defined in \S 6.2),} \\[5pt]
Q(\Lambda)+|\alpha^{b}_1|+|\Delta\sigma_{\rs_{k}}|  \ \ &\mbox{(defined in \S 6.3).}
\end{cases}
\end{eqnarray}
With the notations given above, we now prove the decreasing property of our functional $F(J)$ by specifying
constants $K$, $K_i$ with $i=1,2,3,4$, $K_{\rm ce}$, and $K_{\rm wc}^{(i)}$ with $i=1,2$.

We first have the main global interaction estimate as stated below.

\smallskip
\begin{proposition}\label{prop:7.1}
Suppose that $M_\infty$ is sufficiently large and $\sum^{\infty}_{k=0}\big(1+|b_{k}|\big)|\omega_{k}|$
is sufficiently small. Let $I$ and $J$ be a pair of space-like mesh curves
with $I<J$, and let $\Lambda$ be the diamond between $I$ and $J$.
Then there exist positive constants $\eta$,
$K_{i}$ with $i=1,2,3,4$, $K^{(i)}_{\rm wc}$ with $i=1,2$, $K_{\rm ce}$,
and $K$ such that, if $F(I)<\eta$,
\begin{eqnarray}\label{eq:7.8}
F(J)\leq F(I)-\frac{1}{4}E_{\Delta x, \vartheta}(\Lambda),
\end{eqnarray}
where $E_{\Delta x,\vartheta}(\Lambda)$ is given by \eqref{eq:7.7}.
\end{proposition}

\smallskip
\begin{proof}
By induction, on the mesh curves, it suffices to consider the case
that $J$ is an immediate successor to $I$ with
only one diamond $\Lambda$ between $I$ and $J$.
Let $I=I_{0}\cup I'$ and $J=I_{0}\cup J'$.
As in \S 4, we also divide our analysis into three cases depending on the location of the diamond.
From now on, we denote $C>0$ a universal constant depending only on the system,
which may be different at each occurrence.

\medskip
\emph{Case} 1.\ \emph{$\Lambda$ lies between the cone boundary and the leading shock-front}.
We now consider the case as in Lemma \ref{lem:6.3}. Notice that
\begin{eqnarray*}
&&\big(L^{(1)}_{0}+K_{2}L^{(2)}_{0}\big)(J)-\big(L^{(1)}_{0}+K_{2}L^{(2)}_{0}\big)(I)\le C \,Q(\Lambda),\\[1mm]
&&\big(K_{1}L_{1}+K_{3}L_{\rm s}\big)(J)-\big(K_{1}L_{1}+K_{3}L_{\rm s}\big)(I)=0,
\end{eqnarray*}
and
\begin{align*}
 L_{\rm c}(J)-L_{\rm c}(I)
 &=|X-\tilde{X}|(1+x_{k}^{-1})-|X-\bar{X}|(1+x_{k-1}^{-1})
 -|\bar{X}-\tilde{X}|(1+ x_{k-1}^{-1})\\[1mm]
 &\leq -|x_{k}^{-1}-x_{k-1}^{-1}|(\tilde{x}_{0}+\tilde{x}_{1}).
\end{align*}
Then
\begin{eqnarray*}
L(J)-L(I)\leq C\,Q(\Lambda)-K_{4}|x_{k}^{-1}-x_{k-1}^{-1}|(\tilde{x}_{0}+\tilde{x}_{1}).
\end{eqnarray*}
For $Q$, we have
\begin{eqnarray*}
&&Q_{0}(J)-Q_{0}(I)\leq C\,L(I)Q(\Lambda)-Q^{0}(\Lambda),\\[1.5mm]
&&\big(Q_{1}+Q_{2}\big)(J)-\big(Q_{1}+Q_{2}\big)(I)
  \leq C(\sigma^{*}-\sigma_{*})Q(\Lambda)+ C\tilde{x}_{0}|\beta|-|\Delta\sigma||\alpha|,\\[1.5mm]
&&Q_{\rm c}(J)-Q_{\rm c}(I)\leq C\tilde{x}_{0}\tilde{x}_{1}-|\Delta \sigma|\tilde{x}_{0},\\[1.5mm]
&&\sum_{i=1,2}Q^{(i)}_{\rm wc}(J)-\sum_{i=1,2}Q^{(i)}_{\rm wc}(I)\leq
-\sum_{i=1,2}K^{(i)}_{\rm wc}|\beta_{i}|\tilde{x}_{0},\\[1.5mm]
&&Q_{\rm ce}(J)-Q_{\rm ce}(I)\leq-\tilde{x}_{0}\tilde{x}_{1}.
\end{eqnarray*}
This implies that
\begin{align*}
Q(J)-Q(I)
&\leq -\big(1-C(L(I)+\sigma^{*}-\sigma_{*})\big)Q(\Lambda)
+|x_{k}^{-1}-x_{k-1}^{-1}|\tilde{x}_{0}\\[5pt]
&\quad -\big(K^{(1)}_{\rm wc}-C\big)|\beta_{1}|\tilde{x}_{0}
 -\big(K^{(2)}_{\rm wc}-C\big)|\beta_{2}|\tilde{x}_{0}
 -\big(K_{\rm ce}-C\big)\tilde{x}_{0}\tilde{x}_{1}.
\end{align*}
Therefore, it follows that
\begin{align*}
F(J)-F(I)&\leq -\Big\{K\big(1- C(L(I)+\sigma^{*}-\sigma_{*})\big)-C\Big\}Q(\Lambda)\\[1mm]
&\quad -(K_{4}-K )\,|x_{k}^{-1}-x_{k-1}^{-1}|\tilde{x}_{0}
-K\big(K^{(1)}_{\rm wc}-C\big)|\beta_{1}|\tilde{x}_{0}\\[1mm]
&\quad -K\big(K^{(2)}_{\rm wc}-C\big)|\beta_{2}|\tilde{x}_{0}
-K\big(K_{\rm ce}-C\big)\tilde{x}_{0}\tilde{x}_{1}\\[1mm]
&\leq -\frac{1}{4}Q(\Lambda),
\end{align*}
provided that $L(I)$ and $\sigma^{*}-\sigma_{*}$ are small enough,
and $K^{(i)}_{\rm wc}$ with $i=1,2$, $K_{\rm ce}$, and $K>K_{4}$ are sufficiently large.

\smallskip
\emph{Case} 2.\ \emph{$\Lambda$ covers part of the cone boundary but not the leading shock-front}.
We consider only the case as given in Lemma \ref{lem:6.5}.
Since $|\bar{X}-\hat{X}|=O(1)|b_{k}||\omega_{k}|$ by Lemma \ref{lem:3.1},
a direct computation yields that
\begin{eqnarray*}
&&L^{(1)}_{0}(J)-L^{(1)}_{0}(I)\leq|K_{\rm r}||\beta_2|+|K_{\rm b}||\omega_k|+ C |\beta_{2}||\Delta \sigma|
   + C |x_{k}^{-1}-x_{k-1}^{-1}|\tilde{x}_{0},\\[1mm]
&&L^{(2)}_{0}(J)-L^{(2)}_{0}(I)=-|\beta_2|,\\[1mm]
&&L_{1}(J)-L_{1}(I)=-(1+|b_{k}|)|\omega_k|,\\[1mm]
&&L_{\rm c}(J)-L_{\rm c}(I)\leq -|x_{k}^{-1}-x_{k-1}^{-1}|\tilde{x}_{0}
+ C(1+ x_{k}^{-1})|b_{k}||\omega_{k}|,\\[1mm]
&&L_{\rm S}(J)-L_{\rm S}(I)=0.
\end{eqnarray*}
Then
\begin{align*}
L(J)-L(I)
&\leq -(K_2-|K_{\rm r}|)|\beta_2|-(K_1-|K_{\rm b}|)|\omega_k|
   -(K_{4}- C)|x_{k}^{-1}-x_{k-1}^{-1}|\tilde{x}_{0}\\[1mm]
&\ \ \  \ -\big(K_{1}-K_{4}C\,(1+x_{k}^{-1})\big)|b_{k}||\omega_{k}|+ C\,|\beta_{2}||\Delta \sigma|.
\end{align*}
For $Q$, we have
\begin{align*}
&Q_{0}(J)-Q_{0}(I)\leq C\big(|\beta_2|+|\omega_k|+|\beta_{2}||\Delta \sigma|
+|x_{k}^{-1}-x_{k-1}^{-1}|\tilde{x}_{0}\big)L(I),\\[1mm]
&Q_{1}(J)-Q_{1}(I)
\leq C\big(|\beta_2|+|\omega_k|+|\beta_{2}||\Delta \sigma|
  +|x_{k}^{-1}-x_{k-1}^{-1}|\tilde{x}_{0}\big)(\sigma^{*}-\sigma_{*})\\
&\qquad\qquad\qquad\quad\,\,\,  + C L(I)|b_{k}||\omega_k|,\\[1mm]
&Q_{2}(J)-Q_{2}(I)\leq-|\beta_{2}||\Delta\sigma|,\\[1mm]
&Q_{\rm c}(J)-Q_{\rm c}(I)\leq  C(\sigma^{*}-\sigma_{*})|b_{k}||\omega_{k}|
+C\,L(I)|x_{k}^{-1}-x_{k-1}^{-1}|\tilde{x}_{0},\\[1mm]
&\big(K^{(1)}_{\rm wc}Q^{(1)}_{\rm wc}+K^{(2)}_{\rm wc}Q^{(2)}_{\rm wc}\big)(J)
-\big(K^{(1)}_{\rm wc}Q^{(1)}_{\rm wc}+K^{(2)}_{\rm wc}Q^{(2)}_{\rm wc}\big)(I)\leq0,\\[1mm]
&Q_{\rm ce}(J)-Q_{\rm ce}(I)\leq C\,L(I)|b_{k}||\omega_{k}|.
\end{align*}
Thus, we obtain the following estimate for $Q$:
\begin{align*}
Q(J)-Q(I)&\leq -\big(1-C(L(I)+\sigma^{*}-\sigma_{*})\big)|\beta_{2}||\Delta \sigma|\\[1mm]
&\quad\ + C\big(L(I)+\sigma^{*}-\sigma_{*}\big)\big(|\beta_2|+|\omega_k|
+|b_{k}||\omega_{k}|+ |x_{k}^{-1}-x_{k-1}^{-1}|\tilde{x}_{0}\big).
\end{align*}
Finally, we obtain the estimate for $F(J)$:
\begin{align*}
F(J)-F(I)&
\leq -\big(K_2-|K_{\rm r}|-K C(L(I)+\sigma^{*}-\sigma_{*})\big)|\beta_2|\\[1mm]
&\quad -\big(K_1-|K_{\rm b}|-K C(L(I)+\sigma^{*}-\sigma_{*})\big)|\omega_k|\\[1mm]
&\quad -\big(K_{1}-K_{4} C(1+x_{k}^{-1})-K C(L(I)+\sigma^{*}-\sigma_{*})\big)|b_{k}||\omega_{k}|\\[1mm]
&\quad -\big(K_{4}-C-K C(L(I)+\sigma^{*}-\sigma_{*})\big)|x_{k}^{-1}-x_{k-1}^{-1}|\tilde{x}_{0}\\[1mm]
&\quad  -\big(K (1-C(L(I)+\sigma^{*}-\sigma_{*}))-C\big)|\beta_{2}||\Delta \sigma|.
\end{align*}
Using Lemma \ref{lem:7.1}, choosing $K$ sufficiently large, and letting $L(I)$ and $\sigma^{*}-\sigma_{*}$
be sufficiently small,
we have
\begin{eqnarray*}
F(J)-F(I)\leq -\frac{1}{4}\big(|\beta_2|+(1+|b_{k}|)|\omega_k|\big).
\end{eqnarray*}

\smallskip
\emph{Case} 3. \ \emph{$\Lambda$ covers a part of the leading shock-front}.
By Lemma \ref{lem:6.6}, we have
\begin{eqnarray*}
&&L^{(1)}_{0}(J)-L^{(1)}_{0}(I)\leq-|\alpha^{b}_1|+ C Q(\Lambda),\\[1mm]
&&L^{(2)}_{0}(J)-L^{(2)}_{0}(I)\leq|K_{\rm w}||\alpha^{b}_1|
+|\mu_{\rm s}||\Delta\sigma_{\rs_{k}}|+ C Q(\Lambda),\\[1mm]
&&L_{1}(J)-L_{1}(I)=0,\\[1mm]
&&L_{\rm s}(J)-L_{\rm s}(I)
\leq C\big(|\Delta\sigma_{\alpha}|+|\Delta\sigma_{\rs_{k}}|
+|x_{k}^{-1}-x_{k-1}^{-1}|\big)\tilde{x}_{0}
-|\Delta\sigma_{\rs_{k}}|-|K_{\rm s}||\alpha^{b}_1|,\\[1mm]
&&L_{\rm c}(J)-L_{\rm c}(I)
=|X^{*}_{1}-X^{*}_{3}|(1+x_{k}^{-1})-|X^{*}_{1} -X^{*}_{2}|(1+x_{k-1}^{-1})\\
&&\qquad\qquad\qquad\quad\,\,\, -|X^{*}_{2}-X^{*}_{3}|(1+ x_{k-1}^{-1})\\[1mm]
&&\qquad\qquad\qquad\,\,\leq -|x_{k}^{-1}-x_{k-1}^{-1}|(\tilde{x}_{0}+\tilde{x}_{1}).
\end{eqnarray*}
Then we combine the estimates for $L^{(1)}_{0}, L^{(2)}_{0}$, $L_1, L_{\rm s}$, and $L_{\rm c}$ to obtain
\begin{align*}
L(J)-L(I)&\leq-\big(1-K_{2}|K_{\rm w}|+K_{3}|K_{\rm s}|\big)|\alpha^{b}_1|
  -\big(K_{3}-K_2|\mu_{\rm w}|\big)|\Delta\sigma_{\rs_{k}}|\\[1mm]
&\quad\, -\big(K_{4}-K_{3} C\big)|x_{k}^{-1}-x_{k-1}^{-1}|\tilde{x}_{0}
-K_{4}|x_{k}^{-1}-x_{k-1}^{-1}|\tilde{x}_{1} \\[1mm]
&\ \ \ \ + C Q(\Lambda)+ C\big(|\Delta\sigma_{\alpha}|+|\Delta\sigma_{\rs_{k}}|\big)\tilde{x}_{0}.
\end{align*}

Next, we estimate $Q$:
\begin{eqnarray*}
&& Q_{0}(J)-Q_{0}(I) \leq-Q^{0}(\alpha^{a}, \gamma)
  +\big(|\mu_{\rm w}||\Delta\sigma_{\rs_{k}}|+|K_{\rm w}||\alpha^{b}_1|-|\alpha^{b}_1|\big)L(I)+ C L(I)Q(\Lambda),\\[1mm]
&&Q_{1}(J)-Q_{1}(I)
 \leq -|\gamma||\Delta\tilde{\sigma}_{\gamma}|-|\alpha^{b}_1||\Delta\sigma_{\alpha}|
  + C|\alpha^{a}_{1}|\tilde{x}_{1}+ C(\sigma^{*}-\sigma_{*})Q(\Lambda),\\[1mm]
&&Q_{2}(J)-Q_{2}(I)
 \leq(\sigma^{*}-\sigma_{*})\big(|K_{\rm w}||\alpha^{b}_1|+|\mu_{\rm w}||\Delta\sigma_{\rs_{k}}|\big)
 + C|\alpha^{a}_{2}|\tilde{x}_{1}+ C(\sigma^{*}-\sigma_{*})Q(\Lambda),\\[1mm]
&&Q_{\rm c}(J)-Q_{\rm c}(I)\leq -|\Delta\sigma_{\alpha}|\tilde{x}_{0}-|\Delta\tilde{\sigma}_{\gamma}|\tilde{x}_{1}
 + C\tilde{x}_{0}\tilde{x}_{1},\\[1mm]
&&\big(K^{(1)}_{\rm wc}Q^{(1)}_{\rm wc}+K^{(2)}_{\rm wc}Q^{(2)}_{\rm wc}\big)(J)-
 \big(K^{(1)}_{\rm wc}Q^{(1)}_{\rm wc}+K^{(2)}_{\rm wc}Q^{(2)}_{\rm wc}\big)(I)
  \leq -K^{(1)}_{\rm wc}|\alpha^{a}_1|\tilde{x}_{1}-K^{(2)}_{\rm wc}|\alpha^{a}_2|\tilde{x}_{1},\\[1mm]
&&Q_{\rm ce}(J)-Q_{\rm ce}(I)\leq-\tilde{x}_{0}\tilde{x}_{1}.
\end{eqnarray*}
Then
\begin{eqnarray*}
&&Q(J)-Q(I)\\[1mm]
&&\leq -\big(1- C(L(I)+\sigma^{*}-\sigma_{*})\big)Q(\Lambda)
-\big(K^{(1)}_{\rm wc}-C\big)|\alpha^{a}_1|\tilde{x}_{1}
-\big(K^{(2)}_{\rm wc}-C\big)|\alpha^{a}_2|\tilde{x}_{1}\\[1mm]
&&\quad -\big(K_{\rm ce}-C\big)\tilde{x}_{0}\tilde{x}_{1}
+\big(|K_{\rm w}|(\sigma^{*}-\sigma_{*})+(|K_{\rm w}|-1)L(I)\big)|\alpha^{b}_1|\\[1mm]
&&\quad  +(|\mu_{\rm w}|+1)(L(I)+\sigma^{*}-\sigma_{*})|\Delta\sigma_{\rs_{k}}|
 +|x_{k}^{-1}-x_{k-1}^{-1}|(\tilde{x}_{0}+\tilde{x}_{1}).
\end{eqnarray*}
Finally, we combine the estimates of $L$ and $Q$ to obtain
\begin{align*}
F(J)-F(I)&\leq -\Big\{K\big(1- C(L(I)+\sigma^{*}-\sigma_{*})\big)-C\Big\}Q(\Lambda)\\[1mm]
&\quad -\Big\{1-K_{2}|K_{\rm w}|+K_{3}|K_{\rm s}|-K\big(|K_{\rm w}|(\sigma^{*}-\sigma_{*})
  +(|K_{\rm w}|-1)L(I)\big)\Big\}|\alpha^{b}_1|\\[1mm]
&\quad -\Big\{K_{3}-K_2|\mu_{\rm w}|-K(1+ |\mu_{\rm w}|)
  \big(L(I)+\sigma^{*}-\sigma_{*}\big)\Big\}|\Delta\sigma_{\rs_{k}}|\\[1mm]
&\ \ \ -\big(K_{4}-K_{3}C-K\big)|x_{k}^{-1}-x_{k-1}^{-1}|\tilde{x}_{0}
 -\big(K_{4}-K\big)|x_{k}^{-1}-x_{k-1}^{-1}|\tilde{x}_{1}\\[1mm]
&\ \ \ -K\Big\{\big(K^{(1)}_{\rm wc}-C\big)|\alpha^{a}_1|\tilde{x}_{1}
  +\big(K^{(2)}_{\rm wc}-C\big)|\alpha^{a}_2|\tilde{x}_{1}+\big(K_{\rm ce}-C\big)\tilde{x}_{0}\tilde{x}_{1}\Big\}.
\end{align*}
Using \eqref{eq:6.2},
choosing $K>K_{4}$, $K^{(1)}_{\rm wc}, K^{(2)}_{\rm wc}$, and $K_{\rm ce}$
sufficiently large, and taking $L(I)$ and $\sigma^{*}-\sigma_{*}$ sufficiently small, we have
\begin{eqnarray*}
F(J)-F(I)\leq -\frac{1}{4}\big(Q(\Lambda)+|\alpha^{b}_1|+|\Delta\sigma_{\rs_{k}}|\big).
\end{eqnarray*}

\smallskip
Now we choose an appropriate constant $\varrho$ such that, for any $1$-wave $\alpha$ after interaction,
$\sigma_\alpha\geq s_{0}-\varrho$. By \eqref{eq:6.22a}, we have
\begin{eqnarray*}
|s_{k+1}-s_k|\leq |K_{\rm s}\alpha_{1}|+|\mu_{\rm s}\Delta\sigma_{\rs_{k}}|
+ C(|\Delta\sigma_{\alpha}|+|\Delta\sigma_{\rs_k}|)\tilde{x}_0.
\end{eqnarray*}
Then the monotonicity of the Glimm functional implies that there exists a constant $C_{2}>0$ such that
\begin{eqnarray*}
\sum_{k\geq0}\big|s_{k+1}-s_k\big|\leq C_{2}\sum_{J>I}\big(F(I)-F(J)\big)
\leq C_{2}F(0),
\end{eqnarray*}
which leads to
\begin{eqnarray*}
s_0-C_{2}F(0)\leq s_{k+1}\leq s_0+C_{2}F(0).
\end{eqnarray*}
Since  $\sigma_{\alpha}$ satisfies $\sigma_{\alpha}\geq s_{k+1}- C F(0)$,
there exists a positive constant $C_{3}$ such that $\sigma_{\alpha}\geq s_0-C_{3}F(0)$.
Then we choose $\varrho=C_{3}F(0)$.
\end{proof}

\smallskip
For any weak wave $\alpha$, denote $\sigma(x_{\alpha}, y_{\alpha})$ as the corresponding
self-similar variable for point $(x_{\alpha}, y_{\alpha})$
from where the weak wave $\alpha$ is issued, with $x_{\alpha}\in \{x_{k}\,:\, k\geq0\}$.
Denote $U_{\Delta x, \vartheta}(x,y)$ as the approximate solution.
Then we have the following.

\smallskip
\begin{proposition}\label{prop:7.2}
For $M_{\infty}$ sufficiently large and  $\sum^{\infty}_{k=0}(1+|b_{k}|)|\omega_{k}|$ sufficiently small, if
\begin{eqnarray}\label{eq:7.9}
\lambda_{1}(U_{\Delta x, \vartheta}(x_{\alpha}-,\cdot))
<\sigma(x_{\alpha}-,\cdot)<\lambda_{2}(U_{\Delta x, \vartheta}(x_{\alpha}-,\cdot)),
\end{eqnarray}
then
\begin{eqnarray}\label{eq:7.10}
\lambda_{1}(U_{\Delta x, \vartheta}(x_{\alpha}+,\cdot))
<\sigma(x_{\alpha}+,\cdot)<\lambda_{2}(U_{\Delta x, \vartheta}(x_{\alpha}+,\cdot)).
\end{eqnarray}
\end{proposition}

\begin{proof}
For any two mesh curves $I$ and $J$ satisfying $I<J$, we prove the lemma by induction.

Since
$|b_{k}-b_{k-1}|=O(1)|\omega_{k}|$,
we assume that
\begin{eqnarray*}
|b_{k}-b_{0}|\leq\sum^{k}_{j=1}|b_{j}-b_{j-1}|
=O(1)\sum^{k}_{j=0}|\omega_{k}|\leq C \big(F(0)-F(I)\big).
\end{eqnarray*}
Then, using assumption \eqref{eq:7.9}, Lemma \ref{lem:6.5}, and Proposition \ref{prop:7.1},
we have
\begin{align}
|b_{k+1}-b_{0}|&\leq |b_{k+1}-b_{k}|+ |b_{k}-b_{0}|\label{eq:7.11}\\[1mm]
&\leq C |\omega_{k+1}|+ C\big(F(0)-F(I)\big)\nonumber\\[1mm]
&\leq C\big(F(I)-F(J)+F(0)-F(I)\big)\nonumber\\[1mm]
&\leq CF(0).\nonumber
\end{align}
Similarly, for $s_{k}$, assume that
\begin{eqnarray*}
\big|s_{k}-s_{0}\big|\leq C\big(F(0)-F(I)\big).
\end{eqnarray*}
Then, using assumption \eqref{eq:7.9}, Lemma \ref{lem:6.6}, and Proposition \ref{prop:7.1} again,
we have
\begin{align}\label{eq:7.12}
|s_{k+1}-s_{0}|&\leq |s_{k+1}-s_{k}|+ |s_{k}-s_{0}|\\[1mm]
&\leq C\big(|\alpha^{b}_1|+|\Delta\sigma_{\rs_{k}}|+\tilde{x}_{0}\big)+ C\big(F(0)-F(I)\big)\nonumber\\[1mm]
&\leq C\big(F(I)-F(J)+F(0)-F(I)\big)\nonumber\\[1mm]
&\leq CF(0).\nonumber
\end{align}
Since
\begin{eqnarray*}
s_{k+1}-CF(0)<\sigma(x_{\alpha}+, \cdot)<b_{k+1}+CF(0),
\end{eqnarray*}
it follows from \eqref{eq:7.11}--\eqref{eq:7.12} that
\begin{eqnarray}\label{eq:7.13}
s_{0}-CF(0)<\sigma(x_{\alpha}+, \cdot)<b_{0}+CF(0).
\end{eqnarray}
On the other hand, since
\begin{eqnarray*}
\lambda_{1}(U_{\Delta x, \vartheta}(x_{\alpha}+, \cdot))-\lambda_{1}(\Theta(s_{k+1}))=O(1)F(0),
\end{eqnarray*}
it follows from Lemma \ref{lem:6.6} and Proposition \ref{prop:7.1} that
\begin{eqnarray*}
\lambda_{1}(\Theta(s_{k+1}))-\lambda_{1}(\Theta(s_0))=O(1)(s_{k+1}-s_0)=O(1)F(0).
\end{eqnarray*}
Then
\begin{eqnarray*}
&&\lambda_{1}(U_{\Delta x, \vartheta}(x_{\alpha}+, \cdot))-\lambda_{1}(\Theta(s_0))\\[1mm]
&&=\lambda_{1}(U_{\Delta x, \vartheta}(x_{\alpha}+, \cdot))-\lambda_{1}(\Theta(s_{k+1}))
+\lambda_{1}(\Theta(s_{k+1}))-\lambda_{1}(\Theta(s_0))\\[1mm]
&&=O(1)F(0),
\end{eqnarray*}
which leads to
\begin{eqnarray}\label{eq:7.14}
\lambda_{1}(U_{\Delta x, \vartheta}(x_{\alpha}+, \cdot))
&=& b_{0}- (1+b^{2}_{0})^{\frac{3}{2}}M_{\infty}^{-1} + O(1)F(0)\\
&& +O(1)M^{-2}_{\infty}+O(1)e^{-m_{0}M^{2}_{\infty}}.\nonumber
\end{eqnarray}
Similarly, for $\lambda_{2}(U_{\Delta x, \vartheta})$, we have
\begin{eqnarray}\label{eq:7.15}
\lambda_{2}(U_{\Delta x, \vartheta}(x_{\alpha}+, \cdot))
&=& b_{0}+ (1+b^{2}_{0})^{\frac{3}{2}}M_{\infty}^{-1}
 +O(1)F(0)\\
&& +O(1)M^{-2}_{\infty}+ O(1)e^{-m_{0}M^{2}_{\infty}}.\nonumber
\end{eqnarray}
Finally, by \eqref{eq:7.13}--\eqref{eq:7.15}, we obtain
\begin{align}
&\lambda_{1}(U_{\Delta x, \vartheta}(x_{\alpha}+, \cdot))-\sigma(x_{\alpha}+, \cdot)\label{eq:7.16}\\
& <- (1+b^{2}_{0})^{\frac{3}{2}} M_{\infty}^{-1}
+ C\big(F(0)+M^{-2}_{\infty}+ e^{-m_{0}M^{2}_{\infty}}\big),\nonumber\\[1mm]
&\sigma(x_{\alpha}+, \cdot)-\lambda_{2}(U_{\Delta x, \vartheta}(x_{\alpha}+,\cdot))\label{eq:7.17}\\
&<- (1+b^{2}_{0})^{\frac{3}{2}}M_{\infty}^{-1}
+ C\big(F(0) +M^{-2}_{\infty}+ e^{-m_{0}M^{2}_{\infty}}\big).\nonumber
\end{align}
Since  $M_{\infty}$ is sufficiently large
and $F(0)=O(1)\sum^{\infty}_{k=0}(1+|b_{k}|)|\omega_{k}|$ is sufficiently small,
then, by  \eqref{eq:7.16}--\eqref{eq:7.17}, we obtain
\begin{eqnarray*}
\lambda_{1}\big(U_{\Delta x, \vartheta}(x_{\alpha}+,\cdot)\big)
<\sigma\big(x_{\alpha}+, \cdot\big)<\lambda_{2}\big(U_{\Delta x, \vartheta}(x_{\alpha}+,\cdot)\big).
\end{eqnarray*}
This completes the proof.
\end{proof}

\smallskip
Then, applying Proposition \ref{prop:7.1} and following the methods
as done in \cite{chen-zhang-zhu, zhang}, we conclude the following.

\smallskip
\begin{theorem}\label{thm:7.1}
Under assumptions $\mathbf{(H_{1})}$--$\mathbf{(H_{2})}$,
if $M_{\infty}$ is sufficiently large and
$\sum_{k\geq0}(1+|b_{k}|)|\omega_{k}|$ is sufficiently small,
then, for any $\vartheta\in\Pi_{k=0}^{\infty}(-1,1)$
and every $\Delta x>0$, the modified Glimm scheme developed above
defines a sequence of global approximate solutions $U_{\Delta x, \vartheta}(x, y)$ such that
\begin{eqnarray}
&&\qquad\quad \underset{x>0}{\sup}\,\,T.V.\,\big\{U_{\Delta x, \vartheta}( x, y)\,:\, -\infty<y<b_{\Delta}(x)\big\}<\infty,\label{eq:7.18}\\
&&\qquad\quad \int^{0}_{-\infty}\big|U_{\Delta x, \vartheta}(x_1, y+b_{\Delta}(x_1))-U_{\Delta x, \vartheta}(x_2, y+b_{\Delta}(x_2))\big|\,\dd y
  \le C_4|x_1-x_2|,\label{eq:7.18b}
\end{eqnarray}
where $C_{4}>0$ is independent of $U_{\Delta x, \vartheta}$, $\Delta x$, and $\vartheta$.
\end{theorem}

Denote
\begin{eqnarray}\label{eq:7.19}
s_{\Delta x, \vartheta}(x)= \mathbf{1}_{(k\Delta x, (k+1)\Delta x)}s_{k}\qquad\,\, \mbox{for $k\geq 0$},
\end{eqnarray}
where $\mathbf{1}_{A}$ stands for the characteristic function on set $A$.
Then, by direct computation, we have
\begin{eqnarray}\label{eq:7.20}
\chi_{\Delta x, \vartheta}(x)= \int^{x}_{0}s_{\Delta x, \vartheta}(\tau)\, \dd\tau.
\end{eqnarray}
Moreover, by Lemma \ref{lem:6.6} and Proposition \ref{prop:7.1}, we have the following.

\smallskip
\begin{corollary}\label{coro:7.1}
There exists a constant $C_{5}>0$ independent of $U_{\Delta x, \vartheta}$, $\Delta x$, and $\vartheta$ such that
\begin{equation}\label{eq:7.21}
T.V.\,\{s_{\Delta x, \vartheta}(x)\,:\, x\in [0, \infty)\}\le C_{5}.
\end{equation}
\end{corollary}

Once the uniform boundedness of the total variation
of the approximate solutions $U_{\Delta x, \vartheta}$ is obtained, then, by  Proposition \ref{prop:7.1}
and Corollary \ref{coro:7.1}, the convergence of $U_{\Delta x, \vartheta}$  follows.
We can prove that its limit $U_{\vartheta}$ is actually an entropy solution of problem \eqref{eq:1.4}--\eqref{eq:1.6}.
This can be summarized as the following theorem whose proof is standard and similar
to \cite{chen-zhang-zhu, glimm, zhang}, so we omit the details here.

\smallskip
\begin{theorem}\label{thm:7.2}
Let assumptions $\mathbf{(H_{1})}$--$\mathbf{(H_{2})}$ hold.
Assume that
\begin{eqnarray}\label{eq:7.22}
\int^{\infty}_{0}(1+|b(x)|)\,\dd \mu(x) <\tilde{\varepsilon},
\end{eqnarray}
where $\mu(x)=T.V.\,\big\{b'_{+}(\tau)\,:\,\tau\in [0, x)\big\}$.
Then there is a null set $\mathcal{N}$ such that,
if $M_{\infty}$ is sufficiently large and $\tilde{\varepsilon}>0$ sufficiently small,
for each $\vartheta\in(\Pi_{k=0}^{\infty}(-1,1)\setminus \mathcal{N})$, there exist
both
a subsequence $\{\Delta_{i}\}_{i=0}^{\infty}\subset\{\Delta x\}$ of mesh sizes with $\Delta_{i}\rightarrow 0$ as $i\rightarrow \infty$
and a pair of functions $U_{\vartheta}(x,y)\in O_{\hat{\varepsilon}}(\Gamma(b_{0}, u_{\infty}))$ and $\chi_{\vartheta}(x)$
with $\chi_{\vartheta}(0)=0$ such that

\smallskip
{\rm (i)}\, $U_{\Delta_{i}, \vartheta}(x,\cdot)$ converges to $U_{\vartheta}(x,\cdot)$ in $L^{1}(-\infty, b(x))$
for every $x>0$ as $i\to\infty$,

$\quad$ and $U_{\vartheta}$ is a global entropy solution of problem \eqref{eq:1.4}--\eqref{eq:1.6}{\rm ;}

{\rm (ii)}\, $s_{\Delta_i, \vartheta}(x)$ converges to $s_{\vartheta}(x)\in BV([0,\infty))$
with $|s_{\vartheta}(x)-s_{0}|< C\tilde{\varepsilon}${\rm ;}

{\rm (iii)}\, $\chi_{\Delta_i, \vartheta}(x)$ converges to $\chi_{\vartheta}(x)$
uniformly in any bounded $x$--interval such that
\begin{eqnarray}\label{eq:7.23}
\chi_{\vartheta}(x)=\int^{x}_{0}s_{\vartheta}(\tau)\,\dd\tau,
\end{eqnarray}

$\quad$ and $\chi_{\vartheta}(x)<b(x)$ for any $x>0$, where $C>0$ is a constant depending only

$\quad$ on the system.
\end{theorem}

\section{Asymptotic Behavior of Global Entropy Solutions}\setcounter{equation}{0}

To understand the asymptotic behavior of global entropy solutions $U_{\vartheta}(x,y)$,
we need further estimates of the approximate solutions $U_{\Delta_{i}, \vartheta}(x,y)$.

\smallskip
\begin{lemma}\label{lem:8.1}
There exists a constant $M_{1}$ independent of $U_{\Delta x, \vartheta}$, $\Delta x$, and $\vartheta$ such that
\begin{equation}\label{eq:8.1}
\sum_{\Lambda}E_{\Delta x, \vartheta}(\Lambda)<M_{1}
\end{equation}
for $E_{\Delta x, \vartheta}(\Lambda)$ given as in \eqref{eq:7.7}.
\end{lemma}

\smallskip
\begin{proof}
By Proposition \ref{prop:7.1}, for any interaction diamond
$\Lambda\subset\{(k-1)\Delta x \leq x \leq (k+1)\Delta x \}$, $k\geq 1$,
we have
\begin{eqnarray*}
\sum_{\Lambda} E_{\Delta x, \vartheta}(\Lambda)
\leq 4 \sum_{\Lambda}\big(F(I)-F(J)\big) \leq 4F(0).
\end{eqnarray*}
Then  estimate \eqref{eq:8.1} follows by choosing $M_{1}=4F(0)+1$.
\end{proof}

\smallskip
For any $\tau>0$, let $\mathcal{L}_{j,\vartheta}(\tau-)$, $j=1,2$,
be the total variation of $j$-weak waves in $U_{\vartheta}$ crossing line $x=\tau$,
and let $\mathcal{L}_{j,\Delta x,\vartheta}(\tau-)$, $j=1,2$,
be the total variation of $j$-weak waves in $U_{\Delta x,\vartheta}$ crossing line $x=\tau$.
In addition, denote by $\mathcal{C}_{\vartheta}(\tau-)$ the total variation
for the centers in $U_{\vartheta}$
when the self-similar lines cross line $x=\tau$,
and let $\mathcal{C}_{\Delta x,\vartheta}(\tau-)$
be the total variation of the center changes in $U_{\Delta x,\vartheta}$
when the self-similar lines cross line $x=\tau$.
Then we have the following.

\smallskip
\begin{lemma}\label{lem:8.2}
As $x\rightarrow \infty$,
\begin{eqnarray*}
\sum_{j=1}^2 \mathcal{L}_{j,\vartheta}(x-)
+\mathcal{C}_{\vartheta}(x-)\longrightarrow 0.
\end{eqnarray*}
\end{lemma}

\smallskip
\begin{proof}
Let $U_{\Delta_{i}, \vartheta}(x,y)$ be a sequence of the approximate solutions stated
in Theorem \ref{thm:7.2},
and let the corresponding term $E_{\Delta_{i}, \vartheta}(\Lambda)$ be
defined in \eqref{eq:7.7}.
As in \cite{glimm}, denote by $\dd E_{ \Delta_{i}, \vartheta}$
the measures of the assigning quantities $E_{\Delta_{i}, \vartheta}(\Lambda)$
of the centers of $\Lambda$.
Then, by Lemma \ref{lem:8.1}, we can select a subsequence (still
denoted as $\dd E_{ \Delta_{i}, \vartheta}$) such that
\begin{eqnarray*}
\dd E_{\Delta_{i}, \vartheta}\rightarrow  \dd E_{\vartheta}\qquad \mbox{as $\Delta_{i}\rightarrow 0$}
\end{eqnarray*}
with $E_{\vartheta}(\Lambda)<\infty$.

Therefore, for $\hat{\varepsilon}>0$ sufficiently small, we can choose $x_{\hat{\varepsilon}}>0$
(independent of $U_{\Delta_{i}, \vartheta}$), $\Delta_{i}$, and $\vartheta$ such that
\begin{eqnarray*}
\sum _{k>[x_{\hat{\varepsilon}}/\Delta x]}E_{\Delta_{i}, \vartheta}(\Lambda_{k,n})<\hat{\varepsilon}.
\end{eqnarray*}
Let $X^{1}_{\hat{\varepsilon}}=(x_{\hat{\varepsilon}}, \chi_{\Delta_{i}, \vartheta}(x_{\hat{\varepsilon}}))$
and $X^{2}_{\hat{\varepsilon}}=(x_{\hat{\varepsilon}}, b_{\Delta_{i}}(x_{\hat{\varepsilon}}))$
be the two points lying in the approximate $1$-shock
$y=\chi_{\Delta_{i}, \vartheta}(x)$ and the approximate cone boundary $\Gamma_{\Delta_{i}}$, respectively.
Let $\chi^{j}_{\Delta_{i}, \vartheta}$ be the approximate $j$--generalized characteristic
issuing from $X^{j}_{\hat{\varepsilon}}$ for $j=1, 2$, respectively.
According to the construction of the approximate solution,
there exist constants $\hat{M}_{j}>0$ for $j=1,2$, independent of
$U_{\Delta_{i}, \vartheta}$, $\Delta_{i}$, and $\vartheta$, such that
\begin{eqnarray*}\label{eq:9.6}
\big|\chi^{j}_{\Delta_{i}, \vartheta}(x_1)-\chi^{j}_{\Delta_{i}, \vartheta}(x_2)\big|
\leq \hat{M}_{j} \big(|x_1-x_2|+\Delta_{i}\big)
\qquad\,\,\mbox{for $x_1, x_2>x_{\hat{\varepsilon}}$}.
\end{eqnarray*}
Then we can choose a subsequence (still denoted by) $\Delta_{i}$ such that
\begin{eqnarray*}
\chi^{j}_{\Delta_{i}, \vartheta}(x)\rightarrow \chi^{j}_{\vartheta}(x) \qquad \mbox{as $\Delta_{i}\rightarrow 0$}
\end{eqnarray*}
for some $\chi^{j}_{\vartheta}\in {\rm Lip}$ with $(\chi^{j}_{\vartheta})'$ bounded.

Let the two characteristics $\chi^{1}_{\vartheta}(x)$
and $\chi^{2}_{\vartheta}(x)$ intersect with the cone boundary $\partial \Omega$ and shock-front
$y=\chi_{\vartheta}(x)$ at points $(t^{1}_{\hat{\varepsilon}},\chi^{1}(t^{1}_{\hat{\varepsilon}}))$
and $(t^{2}_{\hat{\varepsilon}}, \chi^{2}(t^{2}_{\hat{\varepsilon}}))$ for some
$t^{1}_{\hat{\varepsilon}}$ and $t^{2}_{\hat{\varepsilon}}$, respectively.
Then, as in \cite{glimm-lax}, we apply the approximate
conservation law to the domain below $\chi^{1}_{\Delta_{i}, \vartheta}$ and above $\chi^{1}_{\Delta_{i}, \vartheta}$,
and use Lemma \ref{lem:8.1} to obtain
\begin{eqnarray*}
&\mathcal{L}_{j,\Delta_{i}, \vartheta}(x-)\leq C\sum _{k>[x_{\hat{\varepsilon}}/\Delta x]}E_{\Delta_{i}, \vartheta}(\Lambda_{k,n})
\leq C\hat{\varepsilon},\\
&\mathcal{C}_{\Delta x,\vartheta}(x-)\leq C\sum_{k>[x_{\hat{\varepsilon}}/\Delta x]}(1+|b_{k}|)|\omega_{k}|\leq C\hat{\varepsilon}
\end{eqnarray*}
for $j=1,2$, and $x>t^{1}_{\hat{\varepsilon}}+t^{2}_{\hat{\varepsilon}}$,
where the bound of $O(1)$ is independent of $U_{\Delta x, \vartheta}$, $\Delta x$, and $ \vartheta$.
These lead to
\begin{eqnarray*}
\mathcal{L}_{j,\vartheta}(x-)=O(1)\hat{\varepsilon},\qquad  \mathcal{C}_{\vartheta}(x-)=O(1)\hat{\varepsilon}
\end{eqnarray*}
for $j=1,2$, and $x>t^{1}_{\hat{\varepsilon}}+t^{2}_{\hat{\varepsilon}}$. This completes the proof.
\end{proof}

Denote
\begin{eqnarray}\label{eq:8.4}
X^{*}_{\infty}=\lim_{x\rightarrow \infty}X^{*}(x, b(x)).
\end{eqnarray}

\begin{theorem}\label{thm:8.1}
Let $U_\vartheta$ be the entropy solution of problem
\eqref{eq:1.4}--\eqref{eq:1.6}
given by Theorem {\rm \ref{thm:7.2}}.
Denote $s_{\infty}=\lim_{x\rightarrow \infty}s_{\vartheta}(x)$
and $b'_{\infty}=\lim_{x\rightarrow \infty}b'_+(x)$.
Then
\begin{eqnarray}\label{eq:8.5}
\lim_{x\rightarrow \infty}\sup\big\{|U_{\vartheta}(x,y)-\varpi(\sigma_{\infty}; O_{\infty})|\,:\, \chi_{\vartheta}(x)<y<b(x)\big\}=0,
\end{eqnarray}
where $\varpi(\sigma_{\infty}; O_{\infty})$ is the state of the self-similar solutions with
$\sigma_{\infty}=\frac{y}{x-X^{*}_{\infty}}$ and $O_{\infty}=(X^{*}_{\infty}, 0)$
as its self-similar variable and center, respectively,
and satisfies
\begin{eqnarray}\label{eq:8.6}
\varpi(s_{\infty}; O_{\infty})=\Theta(s_{\infty}),\quad \ \varpi(b'_{\infty}; O_{\infty})\cdot (-b'_{\infty},1)=0
\end{eqnarray}
with $\Theta(s)$ as the state connected to state $U_{\infty}$ by the $1$-shock of speed $s$.
\end{theorem}

\smallskip
\begin{proof}
From the construction of the approximate solution,
there exists a state $\varpi(\sigma_{k}; O_{k})$ such that
$U_{\Delta_{i}, \vartheta}(x,y)=\varpi(\sigma_{k}; O_{k})$ for some $k\geq 1$
with $\sigma_{k}=\frac{y}{x-X^{*}_{k}}$ and $O_{k}=(X^{*}_{k},0)$.
Let $x\in[(l-1)\Delta x, l\Delta x)$ for some $l\geq 1$, and
let
\begin{eqnarray*}
X^{*}_{l}=x_{l-1}-\frac{b_{l-1}}{b_{l}-b_{l-1}}\Delta x.
\end{eqnarray*}
Then, for every $x>0$,
\begin{align*}
\big|U_{\Delta_{i}, \vartheta}(x,y)-\varpi(\sigma_{l}; O_{l})\big|
=&\, |\varpi(\sigma_{k}; O_{k})-\varpi(\sigma_{l}; O_{l})|\\
\leq&\, C\Big(\sum_{j=1,2}\mathcal{L}_{j,\Delta_{i}, \vartheta}(x-)
+\mathcal{C}_{\Delta_{i},\vartheta}(x-)\Big),
\end{align*}
where $C>0$ is independent of $U_{\Delta_{i}, \vartheta}$, $\Delta_{i}$, and $\vartheta$.

On the other hand, for every $x>0$
\begin{eqnarray*}
 |\varpi(\sigma_{l}; O_{l})-\varpi(\sigma_{\infty}; O_{\infty})|
 \leq C\,\mathcal{C}_{\Delta x,\vartheta}(x-),
\end{eqnarray*}
where $C>0$ is independent of $U_{\Delta_{i}, \vartheta}$, $\Delta_{i}$, and $\vartheta$.
Thus, for every $x>0$, we obtain
\begin{align*}
&|\varpi(s_{\Delta_{i}; \vartheta}; O_{l})-\Theta(s_{\Delta_{i},\vartheta}(x))|
+ |\varpi(b'_{\Delta_{i}}(x); O_{l})\cdot(-b'_{\Delta_{i}}, 1)|\\
&+|U_{\Delta_{i}, \vartheta}(x,\cdot)-\varpi(\sigma_{\infty}; O_{\infty})|\\[1mm]
&\leq \sup_{\chi_{\Delta_{i}, \vartheta}(x)<y<b_{\Delta_{i}}(x)}
   |U_{\Delta_{i},\vartheta}(x, y)-\varpi(\sigma_{l}; O_{l})|\\
&\quad\,\, +\sup_{\chi_{\Delta_{i}, \vartheta}(x)<y<b_{\Delta_{i}}(x)}
|\varpi(\sigma_{l}; O_{l})-\varpi(\sigma_{\infty}; O_{\infty})|\\[1mm]
&\leq C\Big(\sum_{j=1,2}\mathcal{L}_{j,\Delta_{i}, \vartheta}(x-)
+\mathcal{C}_{\Delta_{i},\vartheta}(x-)\Big).
 \end{align*}

By Theorem \ref{thm:7.2}, letting $\Delta_{i}\rightarrow 0$, we have
\begin{eqnarray*}\label{eq:9.6b}
&&|\varpi(s_{\vartheta}; O_{l})-\Theta(s_{\vartheta}(x)) |+ |\varpi(b'_+(x); O_{l})\cdot(-b'_+, 1)|\\
&&\, +\sup_{\chi_{\vartheta}(x)<y<b(x)}|U_{\vartheta}(x,y)-\varpi(\sigma_{\infty}; O_{\infty})|
\\[1mm]
&& \leq C\Big(\sum_{j=1,2}\mathcal{L}_{j,\vartheta}(x-)+\mathcal{C}_{\vartheta}(x-)\Big)
\qquad\,\,\mbox{for every $x>0$},
\end{eqnarray*}
which leads to the desire result by using Lemma \ref{lem:8.2}.
\end{proof}

\medskip
\appendix
\section{Proof of Lemma \ref{lem:2.1}}\setcounter{equation}{0}
In this appendix, we give a proof of Lemma \ref{lem:2.1} by showing the fact that system \eqref{eq:2.1}
is genuinely nonlinear for $u>c$.
We first introduce some notations
for the computational convenience.

Denote $q=\sqrt{u^{2}+v^{2}}$, $M=\frac{q}{c}$, $\theta=\arctan\frac{v}{u}$, and $\lambda_j=\lambda_{j}(U)$, $j=1,2$.
Then the eigenvalues can be rewritten as
\begin{eqnarray}\label{eq:9.1}
\lambda_{j}=\tan(\theta+(-1)^{j}\theta_{\rm ma}) \qquad\mbox{for $j=1, 2$},
\end{eqnarray}
where $\theta_{\rm ma}$ is the Mach angle:
\begin{equation}\label{eq:9.1b}
\theta_{\rm ma}:=\arctan(\frac{1}{\sqrt{M^{2}-1}}).
\end{equation}
Moreover, $\theta_{\rm ma}=\arcsin(\frac{1}{M})\in(0,\frac{\pi}{2})$
for supersonic flow.
Then we have the following.

\smallskip
\begin{lemma}\label{lem:9.1}
If $u>c$, then
\begin{equation}\label{eq:9.2}
\cos(\theta+(-1)^{j}\theta_{\rm ma})
=\frac{u\sqrt{M^{2}-1}+(-1)^{j+1}v}{cM^{2}}>0\qquad\,\,\mbox{for $j=1, 2$}.
\end{equation}
\end{lemma}

\smallskip
\begin{proof}
By direct computation, we obtain the first equality:
\begin{equation*}
 cM^{2}\cos(\theta+(-1)^{j}\theta_{\rm ma})=u\sqrt{M^{2}-1}+(-1)^{j+1}v\qquad\,\,\mbox{for $j=1,2$}.
\end{equation*}
Since
$\big(u\sqrt{q^{2}-c^2}\big)^{2}-c^2v^{2}=(u^{2}-c^2)q^{2}
>0$,
then
\begin{eqnarray*}
u\sqrt{M^{2}-1}> |u|.
\end{eqnarray*}
This completes the proof.
\end{proof}

\smallskip
\begin{lemma}\label{lem:9.2}
If $u>c$, then
$$
\frac{\partial q}{\partial u}=\cos\theta,\quad \frac{\partial q}{\partial v}=\sin\theta,
\quad \frac{\partial\theta}{\partial u}=-\frac{\sin\theta}{q},\quad \frac{\partial \theta}{\partial v}=\frac{\cos\theta}{q},
\quad\frac{\partial\theta_{\rm ma}}{\partial q}=-\frac{1}{cM\sqrt{M^{2}-1}}.
$$

\begin{proof}
We prove only for the last identity above,  since the proofs for the others are similar.
From \eqref{eq:9.1b}, we have
\begin{equation*}
\cos\theta_{\rm ma}=\frac{\sqrt{M^{2}-1}}{M}.
\end{equation*}
Therefore, we have
\begin{eqnarray*}
\frac{\partial\theta_{\rm ma}}{\partial q}=-\frac{M}{(M^{2}-1)^{\frac{3}{2}}}\cos^{2}\theta_{\rm ma}
=-\frac{1}{cM\sqrt{M^{2}-1}}.
\end{eqnarray*}
This completes the proof.
\end{proof}
\end{lemma}

\smallskip
\begin{lemma}\label{lem:9.3}
If $u>c$, then
$$
\frac{\partial \lambda_{j}}{\partial \theta}=\sec^{2}(\theta+(-1)^{j}\theta_{\rm ma}),
\,\,\,\,\frac{\partial\lambda_{j}}{\partial q}=(-1)^{j+1}\frac{1}{cM\sqrt{M^{2}-1}}\sec^{2}(\theta+(-1)^{j}\theta_{\rm ma})
$$
for $j=1,2$.
\end{lemma}

\smallskip
\begin{lemma}\label{lem:9.4}
If $u>c$, then
\begin{eqnarray}
&&\frac{\partial \lambda_{j}}{\partial u}
=\frac{(-1)^{j}}{c\sqrt{M^{2}-1}}\sin(\theta +(-1)^{j}\theta_{\rm ma})\sec^{2}(\theta+(-1)^{j}\theta_{\rm ma}), \label{eq:9.8}
\\[1mm]
&&\frac{\partial\lambda_{j}}{\partial v}
=\frac{(-1)^{j+1}}{c\sqrt{M^{2}-1}}\cos(\theta+(-1)^{j}\theta_{\rm ma})\sec^{2}(\theta+(-1)^{j}\theta_{\rm ma})
\qquad\mbox{for $j=1,2$}.
\label{eq:9.9}
\end{eqnarray}
\end{lemma}

\smallskip
\begin{proof}
For $j=1$, from Lemmas \ref{lem:9.2}--\ref{lem:9.3}, we have
\begin{align*}
\frac{\partial \lambda_{1}}{\partial u}
&=\frac{\partial \lambda_{1}}{\partial \theta}\frac{\partial \theta}{\partial u}
+\frac{\partial \lambda_{1}}{\partial q}\frac{\partial
q}{\partial u}\\
&=-\sec^{2}(\theta-\theta_{\rm ma})\,\frac{\sin\theta}{q}
+\frac{1}{cM\sqrt{M^{2}-1}}\sec^{2}(\theta-\theta_{\rm ma})\cos\theta\\
&=-\frac{1}{c\sqrt{M^{2}-1}}\sin(\theta -\theta_{\rm ma})\sec^{2}(\theta-\theta_{\rm ma}),\\[2mm]
\frac{\partial \lambda_{1}}{\partial v}&=\frac{\partial \lambda_{1}}{\partial \theta}
\frac{\partial \theta}{\partial v}+\frac{\partial \lambda_{1}}{\partial q}\frac{\partial q}{\partial v}\\
&=\sec^{2}(\theta-\theta_{\rm ma})\,\frac{\cos\theta}{cM}
+\frac{1}{cM\sqrt{M^{2}-1}}\sec^{2}(\theta-\theta_{\rm ma})\sin\theta\\
&=\frac{1}{c\sqrt{M^{2}-1}}\cos(\theta -\theta_{\rm ma})\sec^{2}(\theta-\theta_{\rm ma}).
\end{align*}
The case for $j=2$ can be carried out in the same way. This completes the proof.
\end{proof}

\smallskip
\begin{lemma}\label{lem:9.5}
For $u>c$,
\begin{eqnarray}
(\frac{\partial\lambda_{j}}{\partial u}, \frac{\partial\lambda_{j}}{\partial v})\cdot(-\lambda_{j},1)
=\frac{\sec^{3}(\theta+(-1)^{j}\theta_{\rm ma})}{c\sqrt{M^{2}-1}}\qquad\,\,\mbox{for $j=1, 2$}. \label{eq:9.10}
\end{eqnarray}
\end{lemma}

\smallskip
\begin{proof}
We consider only the case that $j=1$, since it is similar to $j=2$.
By \eqref{eq:9.1}--\eqref{eq:9.9} and Lemma \ref{lem:9.4}, we know that
\begin{align*}
(\frac{\partial\lambda_{1}}{\partial u}, \frac{\partial\lambda_{1}}{\partial v})\cdot (-\lambda_{1}, 1)
&=\frac{\sec^{3}(\theta-\theta_{\rm ma})\sin^{2}(\theta-\theta_{\rm ma})}{c\sqrt{M^{2}-1}}
+\frac{\sec^{2}(\theta-\theta_{\rm ma})\cos(\theta-\theta_{\rm ma})}{c\sqrt{M^{2}-1}}\\[5pt]
&=\frac{\sec^{3}(\theta-\theta_{\rm ma})}{c\sqrt{M^{2}-1}}.
\end{align*}
This completes the proof.
\end{proof}

\smallskip
From Lemma {\rm \ref{lem:9.1}}, we know that system \eqref{eq:2.1}--\eqref{eq:2.2} is genuinely-nonlinear for $u > c$.

Then, according to Lemmas \ref{lem:9.1} and \ref{lem:9.5}, we have the following property
that leads to the proof of Lemma \ref{lem:2.1}.

\smallskip
\begin{lemma}\label{lem:9.6}
For $u>c$,
\begin{eqnarray}\label{eq:9.11}
e_{j}(U)=\sqrt{q^{2}-c^2}\cos^{3}\big (\theta+(-1)^{j}\theta_{\rm ma}\big)>0 \qquad\mbox{for $j=1,2$}.
\end{eqnarray}
\end{lemma}

\smallskip
\medskip
\section*{Acknowledgment}.
The authors would like to thank the anonymous referees for helpful comments.

\end{document}